\documentclass[12pt,leqno]{amsart}

\usepackage[T1]{fontenc}
\usepackage{amsmath}
\usepackage{amssymb}
\usepackage{amsthm}
\usepackage{tikz}
\usepackage[all]{xy}
\usepackage{setspace}
\usepackage{multirow}
\usepackage{supertabular}
\usepackage{longtable}
\usetikzlibrary{patterns}
\usepackage{hyperref}
\usepackage{xcolor}
\usepackage{pdflscape}
\usepackage{xspace}
\usepackage{comment}
\usepackage{caption}
\numberwithin{equation}{section}
\newtheorem{theorem}{Theorem}[subsection]
\newtheorem{definition}[theorem]{Definition}
\newtheorem{lemma}[theorem]{Lemma}
\newtheorem{proposition}[theorem]{Proposition}
\newtheorem{corollary}[theorem]{Corollary}

\newtheorem{sublemma}[theorem]{Sublemma}

\theoremstyle{remark}
\newtheorem{rem}[theorem]{Remark}

\def\fl#1{\smash{\mathop{\hbox to 11mm{ \rightarrowfill\ }}\limits^{\textstyle #1}}}

\DeclareMathOperator{\Hom}{Hom}
\DeclareMathOperator{\card}{card}

\newcommand{\Z}{\mathbb{Z}}
\newcommand{\Q}{\mathbb{Q}}
\newcommand{\C}{\mathbb{C}}
\newcommand{\R}{\mathbb{R}}

\newcommand{\clmcg}[1]{[#1]_{\M}}
\newcommand{\Aff}{\mathrm{Aff}(\C)}
\newcommand{\D}{\mathbb{D}}
\newcommand{\Sph}{\Pu}
\newcommand{\PBn}{\mathrm{PB}_n\Sph}
\newcommand{\Bn}{\mathrm{B}_n\Sph}
\newcommand{\PB}[1]{\mathrm{PB}_{#1}\Sph}
\newcommand{\PBD}{\mathrm{PB}_n\D}
\newcommand{\PBDm}{\mathrm{PB}_{n-1}\D'}
\newcommand{\PBDnum}[1]{\mathrm{PB}_{#1}\D}
\newcommand{\BDnum}[1]{\mathrm{B}_{#1}\D}
\newcommand{\BD}{\mathrm{B}_n\D}
\newcommand{\M}{\mathrm{MCG}}
\newcommand{\PM}{\mathrm{P}\M}

\newcommand{\PGL}{\mathrm{PGL}}
\newcommand{\GL}{\mathrm{GL}}
\newcommand{\SL}{\mathrm{SL}}
\newcommand{\U}{\mathbf{U}}
\newcommand{\gf}{\Lambda_n}
\newcommand{\PC}[1]{\C\mathbb{P}^{#1}}
\newcommand{\Pu}{\PC{1}}
\newcommand{\Pd}{\PC{2}}
\newcommand{\Pt}{\PC{3}}

\newcommand{\HH}[1]{\mathrm{H}^{1}(\Sph_{#1},\C_\lambda)}
\newcommand{\Hn}{\mathrm{H}^{1}(\Sph_n,\C_\lambda)}
\newcommand{\PH}[1]{\mathrm{PH}^{1}(\Sph_{#1},\C_\lambda)}

\newcommand{\Hyp}{\mathcal{L}}
\newcommand{\HR}{\mathcal{H}}
\newcommand{\A}{\mathcal{A}}
\newcommand{\plan}{\mathcal{P}}
\newcommand{\ord}{\mathrm{ord}}
\newcommand{\CRG}{fcrg\xspace}
\definecolor{grass}{rgb}{0.14,0.72,0.2}

\newcommand{\scalefactor}{0.8}

\numberwithin{theorem}{subsection}

\title[Finite braid group orbits in $\mathrm{Aff}(\C)$-character varieties]{Finite Braid group orbits in $\mathrm{Aff}(\C)$-character varieties of the punctured sphere}

\author[G. Cousin]{Ga\"el Cousin} 
\address{G. Cousin\\LAREMA\\ Facult\'e des Sciences\\ 2 Boulevard Lavoisier\\ 49045 Angers, France}
\email{gcousin333@gmail.com}

\author[D. Moussard]{Delphine Moussard}
\address{D. Moussard\\ Research Institute for Mathematical Sciences\\Kyoto University\\ Kyoto 606-8502, Japan}
\email{moussardd@yahoo.fr}

\keywords{character varieties, reflection groups, connections, isomonodromic deformations, hypergeometric functions}
\subjclass[2010]{14F35,20F36,20F55,33C70,34M56}

\begin{document}
\vspace{-0,5cm}
\begin{abstract}
We give a complete description of finite braid group orbits in $\mathrm{Aff}(\C)$-charac\-ter varieties of the punctured Riemann sphere. 
This is performed thanks to a coalescence procedure and to the theory of finite complex reflection groups.
We then derive consequences in the theory of differential equations. These concern algebraicity of isomonodromic deformations for reducible rank two logarithmic 
connections on the sphere, the Riemann-Hilbert problem and $F_D$-type Lauricella hypergeometric functions.
\end{abstract}
\enlargethispage{\baselineskip}
\maketitle
\vspace{-0,8cm}
\begin{flushright} \begin{itshape} Ai nostri amici di Pisa. \end{itshape}
\end{flushright}
\vspace{-0,6cm}
\tableofcontents

\section{Introduction}

In this paper, we give a thorough study of finite orbits for the action of the pure braid group $\PBn$ on the character varieties $\mathrm{Hom}(\gf,G)/G$, 
where $G=\mathrm{Aff}(\C)$ is the affine group of the complex line and $\gf$ is the fundamental group of the Riemann sphere $\Pu$ minus $n\geq 4$ punctures 
$x_1,\ldots,x_n$. In \cite{cousinisom}, the first cited author established a tight relation between such finite orbits for $G=\mathrm{GL}_m(\C)$ and algebraizable 
universal isomonodromic deformations of rank $m$ logarithmic connections on $\Pu$ with  poles $x_1,\ldots,x_n$. The present study is tantamount to the one of finite braid 
group orbits for rank $2$ reducible representations (see Section \ref{sec Garnier}). Hence it describes the algebraizable isomonodromic deformations of rank $2$ 
reducible logarithmic connections on $\Pu$.

Let us present the main line of the paper together with relations with other researches.

After recalling the basics of braid groups (Section \ref{subsecprelim}), we remark that the pure braids fix the linear parts of affine representations and 
that their action can be reduced to linear dynamics parametrized by the linear part (Section \ref{subsecdesc}). Then, for $n=4$, the linear group at hand is a 
two generators subgroup of $\mathrm{GL}_2(\C)$. From Schwarz \cite{MR1579568}, these groups are well known, and we may determine what linear parts allow 
the presence of non trivial finite orbits (Section \ref{sec4}). Subsequently, sufficient information on the cases $n>4$ is  obtained by a topological coalescence 
procedure, based on the idea that, putting two punctures together, one goes from $n$ to $n-1$ (Section \ref{seccontraintes}). In turn, we need to study only 
finitely many linear groups to complete the study (Theorem~\ref{thcasfinis}). For this sake, we introduce a piece of theory of finite complex reflection 
groups \cite{LT} (Section~\ref{sec fcgr}) and provide an exhaustive study of the orbits in $\Pd$ and $\Pt$ respectively for the groups $G_{25}$ and $G_{32}$ 
of Shephard and Todd's list of finite complex reflection groups \cite{ST} (Sections \ref{sec Hessian} and \ref{sec 25920}).

We conclude the paper by illustrating several applications to the theory of differential equations (Section~\ref{secappli}).
This serves as a motivation to explain why, through Schlesinger deformation of triangular rank $2$ systems, the above linear braid group action is -- at least 
projectively -- the monodromy of a flat logarithmic connection on $(\Pu)^{n-3}$. We then retrieve the well known relation (see \cite{MR830631,MR849651})  with hypergeometric functions by providing an explicit conjugation between this connection and the one describing Lauricella's 
hypergeometric functions of type $F_D$.
From this conjugation and the above study, we recover the Schwarz list of Bod \cite{MR2852217} for irreducible $F_D$-hypergeometric functions 
(Theorem \ref{th Bod}). Conversely, it is possible to recover part of our results 
from those of Bod.

However, on the side of representations, our direct approach gives clearer information for the reducible cases. Moreover, we believe that our detailed account 
of the orbits for the various groups is completely original. We also stress that the coalescence method we use here should be helpful for other character varieties 
$\mathrm{Hom}(\gf,G)/G$. In particular, it could be an efficient tool to tackle the similar study for irreducible rank $2$ representations. The initialization of 
this investigation has already been performed in \cite{MR3253555} through a complete description of finite braid group orbits 
in $\mathrm{Hom}(\Lambda_4,\mathrm{SL}_2(\C))/\mathrm{SL}_2(\C)$. To this respect, we consider the present work as a model case for forthcoming inquiries.

Apart from finite representations, there is indeed very few information concerning irreducible finite orbits on $\mathrm{Hom}(\gf,\mathrm{SL}_2(\C))/\mathrm{SL}_2(\C)$, $n\geq 5$. 
The only works we know in this direction are \cite{MR3077637,GirandGarnier}. For irreducible rank $3$ representations, the initial case $n=4$ seems 
to be completely unstudied. It should certainly be challenging, as we have to study polynomial dynamics on a character (affine) variety of very large minimal 
embedding dimension. Indeed, even prescribing the conjugacy classes at infinity, this dimension exceeds~$36$ \cite{MR2457722}.

Beyond the study of finite orbits, there are many interesting questions (and results) concerning mapping class groups dynamics on character varieties, even 
for low dimensional groups such as $\mathrm{Aff}(\C)$ or $\mathrm{SL}_2(\C)$. Without claiming for exhaustivity, let us mention the survey \cite{MR2264541} and 
the recent contributions \cite{MR2353464,MR2302065,MR2649343,MR2553877,MR2807844,MR3457677,ghazouani2016mapping}.

   \section{Statement of the main results}

\subsection{Preliminaries: braid groups and mapping class groups} \label{subsecprelim}
In this section, we recall the definition of the braid groups and the related notions, and we fix the notation. For a more detailed exposition 
and for proofs of the assertions of this section, we refer the reader to \cite{Bir}.

For a connected manifold $M$,
the {\em configuration space of $n$ ordered points in $M$} is $F_{0,n}M:=\{(y_1,\ldots,y_n)\in M^n\vert\, i\neq j \Rightarrow y_i\neq y_j \}$.
The {\em configuration space of $n$ unordered points in $M$} is the quotient space $G_{0,n}M:=F_{0,n}M/\mathfrak{S}_n$ by the natural action 
of the symmetric group $\mathfrak{S}_n$ on $F_{0,n}M$. 
The natural map $F_{0,n}M \rightarrow G_{0,n}M, (y_1,\ldots,y_n)\mapsto [y_1,\ldots,y_n]$ is a covering map.
For a choice of base point $z=(z_1,\ldots,z_n)\in F_{0,n}M$, the \emph{pure braid group} of $M$ with $n$ strands is defined as $\mathrm{PB}_nM:=\pi_1(F_{0,n}M,z)$, 
the {\em (full) braid group} of $M$ with $n$ strands is $\mathrm{B}_nM:=\pi_1(G_{0,n}M,[z])$. The above covering identifies $\mathrm{PB}_nM$ with a subgroup 
of $\mathrm{B}_nM$.

The index $0$ in $F_{0,n}M$ is justified by the following notation that we will use in the sequel:
\[F_{k,n-k}M:=\{y \in F_{0,n}M \vert i>n-k \Rightarrow y_i=z_i\}.\] 

An element of the braid group $\mathrm{B}_nM$ can be represented by a {\em geometric braid}: a continuous and injective map 
$f : \{1,\ldots,n\}\times[0,1] \hookrightarrow M\times[0,1]$ such that $f(i,0)=(z_i,0)$, $f(i,1)=(z_{\sigma(i)},1)$ and $f(i,t)\in M\times\{t\}$ 
for all $i\in\{1,\ldots,n\}$, all $t\in[0,1]$ and a permutation $\sigma\in\mathfrak{S}_n$. Figure \ref{figbraids} shows diagrams representing such geometric braids. 
The datum of an element of $\mathrm{B}_nM$, called a {\em braid}, is equivalent to the datum of a geometric braid up to homotopy.
\begin{figure}[htb] 
\begin{center}
\begin{tikzpicture} [scale=0.6] 
\begin{scope} [xscale=1.2]
 \draw (4,5) .. controls +(0,-1.5) and +(0,1.5) .. (2,0);
 \draw[white,line width=5pt] (1,5) .. controls +(0,-1.5) and +(0,1.5) .. (3,0);
 \draw (1,5) .. controls +(0,-1.5) and +(0,1.5) .. (3,0);
 \draw[white,line width=5pt] (3,5) .. controls +(0,-1.5) and +(0,1.5) .. (1,0);
 \draw (3,5) .. controls +(0,-1.5) and +(0,1.5) .. (1,0);
 \draw[white,line width=5pt] (2,5) .. controls +(0,-1.5) and +(0,1.5) .. (4,0); 
 \draw (2,5) .. controls +(0,-1.5) and +(0,1.5) .. (4,0); 
 \draw (0,0) -- (5,0);
 \draw (0,5) -- (5,5);
\end{scope}
\begin{scope} [xshift=12cm,yshift=2.5cm]
 \draw[domain=1:3, samples=200] plot [variable=\r] (36*\r:\r);
 \draw[domain=1:3, samples=200] plot [variable=\r] (-144*\r+180+36:\r);
 \draw[domain=1:3, samples=200] plot [variable=\r] (-36*\r+144+36:\r);
 \draw[white,line width=5pt,domain=1:3, samples=200] plot [variable=\r] (36*\r:\r);
 \draw[domain=1:3, samples=200] plot [variable=\r] (36*\r:\r);
 \draw[white,line width=5pt,domain=1:3, samples=200] plot [variable=\r] (-36*\r+144:\r);
 \draw[domain=1:3, samples=200] plot [variable=\r] (-36*\r+144:\r);
\foreach \x in {1,3}{
 \draw (0,0) circle (\x);
 \foreach \t in {36,72,108,144} {
 \draw (\x*cos{\t},\x*sin{\t}) node{$\bullet$};}}
\end{scope}
\end{tikzpicture}
\end{center} \caption{Braids of the 2-disk and of the 2-sphere} \label{figbraids}
\end{figure}
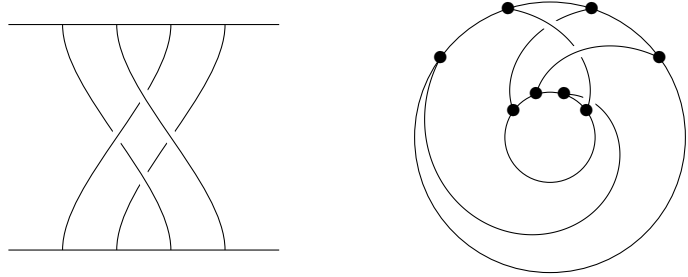

Note that a braid induces a permutation in $\mathfrak{S}_n$ -- the permutation $\sigma$ in the definition of a geometric braid. This provides a morphism 
$\mathrm{B}_nM\to\mathfrak{S}_n$ whose kernel is the pure braid group $\mathrm{PB}_nM$.

In this article, we are chiefly interested in the braid groups of the 2-dimensional disk $\D$ and of the 2-dimensional sphere $\Pu$. 
Viewing $\D$ as the unit disk of $\R^2$, we define the groups $\mathrm{B}_n\D$ and $\mathrm{PB}_n\D$ with respect to the base point $z$ 
given by $z_i=\frac{2i}{n+1}-1$. We denote by $x=(x_1,\ldots,x_n)$ the base point used to define $\mathrm{B}_n\Pu$ and $\mathrm{PB}_n\Pu$. 
Fix an embedding of the closed unit disk $\chi : \overline{\D} \rightarrow \Pu$ such that $\chi(z_i)=x_{i}$ for $i=1,\ldots,n$. 
Every loop in $G_{0,n}\Pu$ is homotopic to a loop in $G_{0,n}\chi(\D)$ and $\chi$ induces an epimorphism 
$\BD \stackrel{\chi_* }{\longrightarrow} \Bn, \beta\mapsto [\beta]_{\Pu}$. When it does not seem to cause confusion, a braid $[\beta]_{\Pu}$ 
is simply denoted $\beta$. In Figure~\ref{figPB}, we depict a braid $\sigma_{i,j}$ and its square.
The braid group $\mathrm{B}_n\D$ is generated by the braids $\sigma_i=\sigma_{i,i+1}$.

More precisely, we have the following presentations:
$$\BD=\left\langle{\sigma_i,1\leq i<n\,\left| \begin{array}{l} \sigma_i\sigma_j=\sigma_j\sigma_i \textrm{ if } |i-j|>1 \\ 
\sigma_i\sigma_{i+1}\sigma_i=\sigma_{i+1}\sigma_i\sigma_{i+1}\end{array}\right.}\right\rangle $$
$$\Bn=\left\langle\sigma_i,1\leq i<n\,\left| \begin{array}{l} \sigma_i\sigma_j=\sigma_j\sigma_i \textrm{ if } |i-j|>1 \\ 
\sigma_i\sigma_{i+1}\sigma_i=\sigma_{i+1}\sigma_i\sigma_{i+1}\\ \sigma_1\dots\sigma_{n-1}\sigma_{n-1}\dots\sigma_1=1\end{array}\right.\right\rangle $$

The pure braid group $\mathrm{PB}_n\D$ is generated by the elements $\sigma_{i,j}^2$, for $1\leq i<j\leq n$.
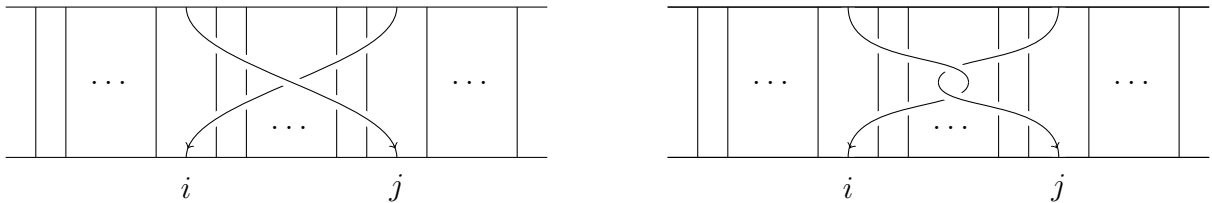
\begin{figure}[htb] 
\begin{center}
\begin{tikzpicture}
\begin{scope} [scale=0.4]
 \foreach \x in {1,2,5,7,8,11,12,14,17} \draw (\x,0) -- (\x,5);
 \draw (3.5,2.5) node {$\dots$};
 \draw (15.5,2.5) node {$\dots$};
 \draw (9.5,1) node {$\dots$};
 \draw[white,line width=5pt] (13,5) .. controls +(0,-2) and +(0,2) .. (6,0);
 \draw (13,5) .. controls +(0,-2) and +(0,2) .. (6,0);
 \draw[white,line width=5pt] (6,5) .. controls +(0,-2) and +(0,2) .. (13,0);
 \draw (6,5) .. controls +(0,-2) and +(0,2) .. (13,0);
 \draw (0,0) -- (18,0);
 \draw (0,5) -- (18,5);
 \draw (6,-1) node {$i$};
 \draw (13,-1) node {$j$};
 \draw[->] (6.1,0.4) -- (6.05,0.3);
 \draw[->] (12.9,0.4) -- (12.95,0.3);
\end{scope}
\begin{scope} [scale=0.4,xshift=22cm]
 \draw (0,0) -- (18,0);
 \draw (0,5) -- (18,5);
 \foreach \x in {1,2,5,7,8,11,12,14,17} \draw (\x,0) -- (\x,5);
 \draw (3.5,2.5) node {$\dots$};
 \draw (15.5,2.5) node {$\dots$};
 \draw (9.5,1) node {$\dots$};
 \draw[white,line width=5pt] (13,5) .. controls +(0,-2) and +(0,1) .. (9,2.5);
 \draw (13,5) .. controls +(0,-2) and +(0,1) .. (9,2.5);
 \draw[white,line width=5pt] (10,2.5) .. controls +(0,-1) and +(0,2) .. (6,0);
 \draw (10,2.5) .. controls +(0,-1) and +(0,2) .. (6,0);
 \draw[white,line width=5pt] (6,5) .. controls +(0,-2) and +(0,1) .. (10,2.5);
 \draw (6,5) .. controls +(0,-2) and +(0,1) .. (10,2.5);
 \draw[white,line width=5pt] (9,2.5) .. controls +(0,-1) and +(0,2) .. (13,0);
 \draw (9,2.5) .. controls +(0,-1) and +(0,2) .. (13,0);
 \draw (0,0) -- (18,0);
 \draw (0,5) -- (18,5);
 \draw (6,-1) node {$i$};
 \draw (13,-1) node {$j$};
 \draw[->] (6.1,0.4) -- (6.05,0.3);
 \draw[->] (12.9,0.4) -- (12.95,0.3);
\end{scope}
\end{tikzpicture}
\end{center} \caption{The braid $\sigma_{i,j}$ and the pure braid $\sigma_{i,j}^2$} \label{figPB}
\end{figure}

We now define an action of $\BD$ on the fundamental group of $\Pu_n:=\Pu \setminus \{x_1,\ldots,x_n\}$. 
Classically, the fundamental group $\pi_1(\overline{\D}\setminus \{z_1,\ldots,z_n\},1)$ is freely generated by the $n$ loops 
$\tilde{\alpha}_i, 1\leq i \leq n$ depicted in Figure \ref{figalphai}; their images $\alpha_i=\chi_*(\tilde{\alpha_i})$ generate  
$\gf:=\pi_1(\Pu_n, x_{n+1})$, where $x_{n+1}=\chi(1)$, and every relation among them is derived from $\alpha_1\dots \alpha_n=1$.
\begin{figure}[htb] 
\begin{center}
\begin{tikzpicture} [scale=0.7]
\draw (5,0) circle (5 and 1.5);
\foreach \x in {1.5,4,5,6,8.5} \draw (\x,0) node {$\bullet$};
\draw (10,0) node {$\times$};
\draw (2.75,0) node {$\dots$};
\draw (7.25,0) node {$\dots$};
\draw (1.5,-0.5) node {$z_1$};
\draw (4,-0.5) node {$z_{i-1}$};
\draw (5,-0.5) node {$z_i$};
\draw (6,-0.5) node {$z_{i+1}$};
\draw (8.5,-0.5) node {$z_n$};
\draw (10.4,-0.2) node {$1$};
\draw (10,0) .. controls +(-1.5,0.7) and +(0,0.8) .. (4.7,0) .. controls +(0,-0.3) and +(0,-0.3) .. (5.3,0) .. controls +(0,0.5) and +(-1.5,0.4) .. (10,0);
\draw[->] (5.3,0.45) -- (5.2,0.4);
\draw (5.2,0.8) node {$\tilde{\alpha}_i$};
\end{tikzpicture}
\end{center} \caption{The loop $\tilde{\alpha}_i$} \label{figalphai}
\end{figure}
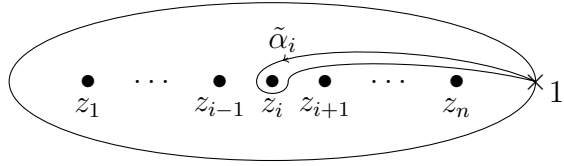
First define the Hurwitz action of $\BD$ on $\pi_1(\overline{\D}\setminus\{z_1,..,z_n\},1)$ by:
$$\sigma_i\cdot\tilde{\alpha}_j=\left\lbrace\begin{array}{l l}
                                             \tilde{\alpha}_i\tilde{\alpha}_{i+1}\tilde{\alpha}_i & \textrm{if }j=i, \\
                                             \tilde{\alpha}_i & \textrm{if }j=i+1, \\
                                             \tilde{\alpha}_j & \textrm{otherwise}.
                                            \end{array}\right.$$
This action can be understood geometrically as ``letting the loop slide along the braid'' (see Figure \ref{figHur}). 
\begin{figure}[htb] 
\begin{center}
\begin{tikzpicture} [scale=0.8]
\draw (2.5,0) circle (2.5 and 0.8);
\draw (5,0) .. controls +(-2,0.7) and +(0,1) .. (1.5,0) .. controls +(0,-0.8) and +(0,-0.8) .. (3.2,0) .. controls +(0,0.4) and +(0,0.4) .. (2.8,0) 
 .. controls +(0,-0.5) and +(0,-0.5) .. (1.7,0) .. controls +(0,0.8) and +(-2,0.4) .. (5,0);
\draw[white,line width=5pt] (1,0) -- (1,4) (4,0) -- (4,4) (2,0) -- (3,4);
\draw (1,0) -- (1,4) (4,0) -- (4,4) (2,0) -- (3,4);
\draw[white,line width=5pt] (3,0) -- (2,4);
\draw (3,0) -- (2,4);
\foreach \x in {1,...,4} \draw (\x,0) node {$\bullet$};
\draw (5,0) node {$\times$};
\draw[fill=white,opacity=0.5] (2.5,4) circle (2.5 and 0.8);
\draw (2.5,4) circle (2.5 and 0.8);
\draw (5,4) .. controls +(-1.5,0.7) and +(0,0.8) .. (1.7,4) .. controls +(0,-0.4) and +(0,-0.4) .. (2.3,4) .. controls +(0,0.5) and +(-1.5,0.4) .. (5,4);
\foreach \x in {1,...,4} \draw (\x,4) node {$\bullet$};
\draw (5,4) node {$\times$};
\end{tikzpicture}
\end{center} \caption{Hurwitz action} \label{figHur}
\end{figure}
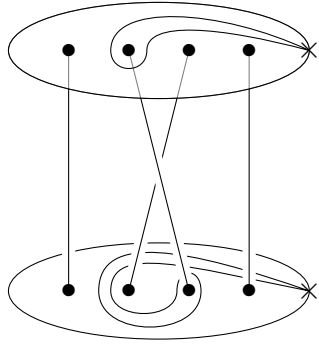
Now, the Hurwitz action preserves the product $\tilde{\alpha}_1\dots \tilde{\alpha}_n$ 
and gives rise to an antimorphism $\mathrm{B}_n\D \stackrel{Hur}{\longrightarrow} \mathrm{Aut}(\gf)$, where $\mathrm{Aut}(\gf)$ is the group 
of automorphisms of $\gf$. 
\begin{rem}
 A similar action of $\Bn$ on $\gf$ can be defined, with a subtlety concerning the base point, see \cite{cousinisom}. However, it is technically simpler 
 to consider the action of $\BD$, which provides the same action up to inner automorphisms of $\gf$.
\end{rem}

The {\em mapping class group} of $\Pu_n$ is the group of isotopy classes of self-homeomorphisms of $\Pu_n$, denoted by $\M_n(\Pu)$. 
The \emph{pure mapping class group} of $\Pu_n$, denoted by $\PM_n(\Pu)$, is the subgroup 
of $\M_n(\Pu)$ which does not permute the punctures $x_i$. From any self-homeomorphism $h$ of $\Pu$, we get an isomorphism 
$h_*: \gf\rightarrow \pi_1(\Pu_n,h(x))$. With any path $\gamma$ in $\Pu$ from $x$ to $h(x)$ is associated an isomorphism 
$\gamma_*:\pi_1(\Pu_n,h(x))\simeq \gf$, and $\gamma_*\circ h_*$ is an element of $\mathrm{Aut}(\gf)$. If we change the path $\gamma$, the class of $\gamma_*\circ h_*$ 
in $\mathrm{Out}(\gf)$ does not change. This yields a morphism  $$\M_n(\Pu)\rightarrow \mathrm{Out}(\gf).$$
We also have a natural surjective antimorphism $\varphi :\Bn \twoheadrightarrow \M_n(\Pu)$. It is defined as follows: every element $\beta \in \mathrm{B}_n\Pu$ 
is represented by a path $(\beta_1(t),\dots,\beta_n(t))_{t\in [0,1]}$ in $F_{0,n}\Pu$, with starting point $x$, that projects to a loop in $G_{0,n}\Pu$. 
There exists a continuous family $(h_t)$ of self-homeomorphisms of $\Pu$ such that $h_t(x_i)=\beta_i(t)$ for $i=1,\ldots,n$ and $t\in[0,1]$. 
The image $\varphi(\beta)$ is defined to be the isotopy class of $h_1$ in $\M_n(\Pu)$.
We have the following commutative diagram.

\[ \xymatrix{
    \BD \ar@{->>}[r]^-{\chi_*} \ar[d]_{Hur}  & \Bn \ar@{->>}[r]^-{\varphi} & \M_n(\Pu) \ar[d] \\
    \mathrm{Aut}(\gf) \ar@{->>}[rr] && \mathrm{Out}(\gf)
  }\]
  
Let us write a version of the above diagram for the ``pure'' case. Collapsing the boundary of the disk into a puncture of the sphere, we get a natural 
identification $\PBDnum{n-1}\cong\pi_1(F_{1,n-1}\Pu)$. Now, an $n$-strand geometric braid on the sphere is clearly isotopic to a braid whose $n$-th strand is constant. 
This gives a natural surjection $\pi_1(F_{1,n-1}\Pu)\twoheadrightarrow\PBn$. We shall see that the associated surjective map $\PBDnum{n-1}\twoheadrightarrow\PBn$ 
factorizes through $\chi_*$. 

Consider the subdisk $\D'$ of $\D$ depicted in Figure~\ref{figdisks}. 
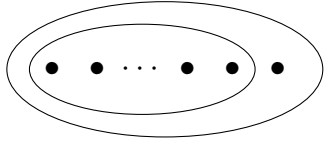
\begin{figure}[htb] 
\begin{center}
\begin{tikzpicture} [scale=0.6]
 \draw (4.5,0) circle (3.5 and 1.5);
 \draw (4,0) circle (2.5 and 1);
 \foreach \x in {2,3,5,6,7} \draw (\x,0) node {$\bullet$};
 \draw (4,0) node {$\cdots$};
\end{tikzpicture}
\end{center} \caption{Embedding of punctured disks} \label{figdisks}
\end{figure}
We have an induced injection $(F_{0,n-1}\D',z')\rightarrow (F_{1,n-1}\D,z)$, where $z'=(z_1,\ldots,z_{n-1})$. 
The composition $(F_{0,n-1}\D',z')\rightarrow (F_{1,n-1}\D,z)\rightarrow (F_{0,n}\D,z)$ induces an injective map $\mathrm{PB}_{n-1}\D'\hookrightarrow \PBD$ on 
fundamental groups. Composing by $\chi_*$, we recover the above surjective map $\PBDnum{n-1}\twoheadrightarrow\PBn$. 
Keeping the notation $\PBDm$ when we consider it as a subgroup of $\PBD$ {\em via} the above injection, we have the following commutative diagram. 
\[ \xymatrix{
    \PBDm \ar@{->>}[r]^-{\chi_*} \ar[d]_{Hur}  & \PBn \ar@{->>}[r]^-{\varphi} & \PM_n(\Pu) \ar[d] \\
    \mathrm{Aut}(\gf) \ar@{->>}[rr] && \mathrm{Out}(\gf)
  }\]

For any group $G$, we have a natural action of $\mathrm{Out}(\gf)$ on $\Hom(\gf,G)/G$, the quotient being by overall conjugation in the target.
In view of the first diagram, this provides an action of $\M_n(\Pu)$ on $\Hom(\gf,G)/G$. The goal of this paper is to determine the finite orbits of the action 
of the subgroup $\PM_n(\Pu)$ when $G$ is the affine group of the complex line $\Aff$, for every $n\geq 4$. 
As $\chi_*$ and $\varphi$ are onto, this is tantamount to the analogous study for the action of $\PBDm$. In practice, we will compute the action of $\PBDm$ 
which is technically simpler. Actually, we will often consider the whole pure braid group $\PBD$ before passing to $\PBDm$, reducing the number of generators.

\begin{rem}
In this paper we determine the finite orbits of $\Hom(\gf,\Aff)/\Aff$ under the action of $\PM_n(\Pu)$. 
 The orbits under the action of $\M_n(\Pu)$ can easily deduced since $\PM_n(\Pu)$ is a finite index subgroup of $\M_n(\Pu)$. 
 More precisions on that point are given in Lemma \ref{lemmaperm}.
\end{rem}

\subsection{Description of the braid group action} \label{subsecdesc}
In this section, we compute explicitly the action of $\PBD$ on $\Hom(\gf,\Aff)$ and its quotients.

The Hurwitz action of $\sigma_{i,j}$ on $\gf$ is given by:
$$\left\{ \begin{array}{l l l} 
           \alpha_i & \mapsto & (\alpha_i\dots\alpha_{j-1})\alpha_j(\alpha_i\dots\alpha_{j-1})^{-1}, \\
           \alpha_j & \mapsto & (\alpha_{i+1}\dots\alpha_{j-1})^{-1}\alpha_i(\alpha_{i+1}\dots\alpha_{j-1}),\\
           \alpha_k & \mapsto & \alpha_k \quad \textrm{ if } k\neq i,j.
          \end{array} \right. $$
Consequently, the action of $\sigma_{i,j}^2$ is:
$$\left\{\begin{array}{l l l} 
          \alpha_i & \mapsto & (\alpha_i\dots\alpha_j)(\alpha_{i+1}\dots\alpha_{j-1})^{-1}\alpha_i(\alpha_{i+1}\dots\alpha_{j-1})(\alpha_i\dots\alpha_j)^{-1}, \\
          \alpha_j & \mapsto & (\alpha_{i+1}\dots\alpha_{j-1})^{-1}(\alpha_i\dots\alpha_{j-1})\alpha_j(\alpha_i\dots\alpha_{j-1})^{-1}(\alpha_{i+1}\dots\alpha_{j-1}), \\
          \alpha_k & \mapsto & \alpha_k \quad \textrm{ if } k\neq i,j.
         \end{array} \right. $$

Let $\rho\in\Hom(\gf,\Aff)$. Write $\rho(\alpha_i):z\mapsto\lambda_i z+\tau_i$. Define the {\em linear part of $\rho$} as 
the map $\lambda\in\Hom(\gf,\C^*)$ given by $\lambda(\alpha_i)=\lambda_i$. Since any element in $\PBD$ sends the generators $\alpha_i$ to self conjugates, 
abelianity of $\C^*$ ensures that the linear part is preserved by the action of the pure braids. Let $\Hom_\lambda(\gf,\Aff)$ be the set of representations 
in $\Hom(\gf,\Aff)$ with linear part $\lambda$. If the linear part is trivial, since any two translations commute, the action of $\PBD$ on $\Hom_1(\gf,\Aff)$ 
is trivial. 

A nontrivial linear part $\lambda$ being fixed, a representation $\rho\in\Hom_\lambda(\gf,\Aff)$ is characterized by $(\tau_1,\ldots,\tau_{n-1})\in\C^{n-1}$ 
--this endows $\Hom_\lambda(\gf,\Aff)$ with a $\C$-vector space structure. 
If $f\in\Aff$ is defined by $f(z)=az+b$, the representation $\rho'=f\rho f^{-1}$ is determined by 
\begin{equation}\tag{$\star$}\label{eqconj}(\tau_1',\ldots,\tau_{n-1}')=a(\tau_1,\ldots,\tau_{n-1})+b(1-\lambda_1,\ldots,1-\lambda_{n-1}).\end{equation}
Hence the quotient $V_{\lambda}=\Hom_\lambda(\gf,\Aff)/(\C,+)$ by the translations is isomorphic to $\C^{n-2}$, 
and the quotient $\Hom_\lambda(\gf,\Aff)/\Aff$ by all conjugations is the disjoint union of $[0]$ -- the class of the representation 
by homotheties, which is a fixed point of our action, by abelianity -- and of $\mathrm{P}V_{\lambda} \cong\PC{n-3}$. The conjugacy class 
of a representation $\rho$ will be denoted by $[\rho]$.

The vector space $V_{\lambda}$ is naturally isomorphic to $\Hn$, the first \v{C}ech cohomology group of $\Pu_n$ with coefficients in the local 
system $\C_{\lambda}$ with monodromy $\lambda$. 
In the sequel, we identify completely these vector spaces and drop the notation $V_{\lambda}$.

We aim at determining for which linear parts the action of $\PBn$ on $\mathrm{P}\Hn$ has finite orbits, and to describe them. 
We shall first note that this only depends on the family $(\lambda_1,\ldots,\lambda_n)$ up to permutation.
\begin{lemma} \label{lemmaperm}
Let $\rho \in \Hom(\gf,\Aff)$ and let $\lambda$ given by $(\lambda_i)_{1\leq i\leq n}$ be its linear part. For any $\sigma \in \mathfrak{S}_n$, 
let $\beta_{\sigma}\in \BD$ be a braid that induces the permutation $\sigma$. The orbit $\BD\cdot[\rho]$ splits as a disjoint union of equipotent subsets 
$\BD\cdot[\rho]=\sqcup_{\sigma\in S}\PBD\cdot[\beta_{\sigma}\cdot\rho]$, for some $S\subset \mathfrak{S}_n$.

Fix $\sigma \in \mathfrak{S}_n$ and let $\lambda^{\sigma}$ be the linear part defined by $\lambda^{\sigma}_i=\lambda_{\sigma^{-1}(i)}$. 
The action of the braid $\beta_\sigma$ induces an isomorphism $\Hn\rightarrow \mathrm{H}^1(\Pu_n,\C_{\lambda^{\sigma}})$. In addition, if $\sigma$ is a power of $n$-cycle, 
we have $[\rho^{\sigma}]=[\beta_{\sigma}\cdot\rho]$ where $\rho^\sigma$ 
is defined by $\rho^{\sigma}(\alpha_i) : z\mapsto \lambda_{\sigma^{-1}(i)}z+\tau_{\sigma^{-1}(i)}$.
\end{lemma}
\begin{proof}
The first part is standard, because $\PBD$ is a normal subgroup of $\BD$. The linearity assertion is easily checked by direct computation for standard generators 
of the braid group. Bijectivity follows from the fact that the braid $\beta_\sigma^{-1}$ induces the inverse bijection. 
Now, note that the powers of the $n$-cycle $(1\ldots n)$ are induced by ``cyclic braids'' as depicted 
in Figure \ref{figcycbraid}. Such braids satisfy $[\beta_{\sigma}\cdot\rho]=[\rho^{\sigma}]$. 
This concludes since any $n$-cycle is conjugate to $(1\ldots n)$.
\end{proof}
\begin{figure}[htb] 
\begin{center}
\begin{tikzpicture}
\begin{scope} [scale=0.6]
\draw (2.5,0) circle (2.5 and 0.8);
\foreach \x in {1,2,3} \draw[white,line width=5pt] (\x +1,4) .. controls +(0,-1) and +(0,1) .. (\x,0);
\foreach \x in {1,2,3} \draw (\x +1,4) .. controls +(0,-1) and +(0,1) .. (\x,0);
\draw[white,line width=5pt] (1,4) .. controls +(0,-1) and +(0,1) .. (4,0);
\draw (1,4) .. controls +(0,-1) and +(0,1) .. (4,0);
\foreach \x in {1,...,4} \draw (\x,0) node {$\bullet$};
\draw[fill=white,opacity=0.5] (2.5,4) circle (2.5 and 0.8);
\draw (2.5,4) circle (2.5 and 0.8);
\foreach \x in {1,...,4} \draw (\x,4) node {$\bullet$};
\draw[->] (1.21,1.1) -- (1.18,1);
\draw[->] (2.27,1.3) -- (2.24,1.2);
\draw[->] (2.9,1.6) -- (3,1.5);
\draw[->] (3.43,1.8) -- (3.4,1.7);
\end{scope}
\begin{scope} [scale=0.5,xshift=12cm,yshift=2cm] 
 \foreach \x in {1,3}{
 \draw (0,0) circle (\x);
 \foreach \t in {90,180,270,360} {
 \draw (\x*cos{\t},\x*sin{\t}) node{$\bullet$};}}
\draw[->] (2*cos{45},2*sin{45}) -- (1.9*cos{50},1.9*sin{50});
\draw[->] (2*cos{135},2*sin{135}) -- (1.9*cos{140},1.9*sin{140});
\draw[->] (2*cos{225},2*sin{225}) -- (1.9*cos{230},1.9*sin{230});
\draw[->] (2*cos{315},2*sin{315}) -- (1.9*cos{320},1.9*sin{320});
\foreach \t in {0,90,180,270} {
 \draw[domain=1:3, samples=200] plot [variable=\r] (-45*\r+135+\t:\r);}
\end{scope}
\end{tikzpicture}
\end{center} \caption{A ``cyclic braid'' and its image in $\Bn$.} \label{figcycbraid}
\end{figure}

Fix a linear part $\lambda: \gf\to\C^*$. Let $\rho\in\Hom_\lambda(\gf,\Aff)$, $\rho(\alpha_\nu):z\mapsto\lambda_\nu z+\tau_\nu$, $\nu=1,\ldots,n$. 
The action of $\sigma_{i,j}^2$ on the translation part is given by $\sigma_{i,j}^2.\tau=\tau'$ with: 
\begin{equation}\label{descriptionaction1} \left \lbrace\begin{array}{l}
\tau_i'=\lambda_j\tau_i +(1-\lambda_i)\left(\prod\limits_{k=i}^j\lambda_k\right)\tau_j +(1-\lambda_i)(1-\lambda_j)\sum\limits_{k=i}^j \left 
(\prod\limits_{l=i}^{k-1}\lambda_l\right)\tau_k; \\
\tau_{\nu}'=\tau_{\nu}\mbox{ if }\nu\neq i,j;\\
\tau_j'=\lambda_i\tau_j+(1-\lambda_j)\left(\prod\limits_{k=i+1}^{j-1}\lambda^{-1}_k\right)\tau_i-(1-\lambda_i)(1-\lambda_j)\sum\limits_{k=i+1}^{j-1} \left 
(\prod\limits_{l=k}^{j-1}\lambda^{-1}_l \right)\tau_k.\\
\end{array} \right.\end{equation}

Direct computation shows that the linear transformation $L_{i,j}$ of $\Hom_{\lambda}(\gf,\Aff)$ induced by $\sigma_{i,j}^2$ is either 
a \emph{complex reflection}, meaning that $L_{i,j}-id$ has rank 1, or the identity map. 
The trace of $L_{i,j}$ is $\lambda_i \lambda_j+n-2$, so it equals the identity if and only if $\lambda_i \lambda_j=1$. Otherwise, the unique 
nontrivial eigenvalue is $\lambda_i \lambda_j$.

The subspace of abelian representations $\Delta=\mathrm{Vect}\left( (1-\lambda_1,\ldots,1-\lambda_{n-1})\right)$ is fixed by $L_{i,j}$ pointwise, 
thus the action of $\sigma_{i,j}^2$ on the quotient $\Hn=\Hom_{\lambda}(\gf,\Aff)/\Delta$ is a complex reflection with nontrivial eigenvalue 
$\lambda_i\lambda_j$ if $L_{i,j}$ is not the identity.

The action of $\PBD$ on $\Hn$ defines a map $\PBD\to\mathrm{Aut}(\Hn)\cong\GL_{n-2}(\C)$. Let $\hat{\Gamma}_{\lambda}$ be 
the image of $\PBDm$ by this map and $\Gamma_{\lambda}\subset \mathrm{Aut}(\PH{n})\cong\PGL_{n-2}(\C)$ the corresponding group of projective transformations. 
Note that we would obtain the same group $\Gamma_\lambda$ by considering the image of $\PBD$; however $\hat{\Gamma}_\lambda$ may be impacted.

To give explicit matrices generating the group $\hat{\Gamma}_{\lambda}$, we proceed as follows. 

Assume $\lambda_1\neq1$. By Relation (\ref{eqconj}), a section of  $\Hom_\lambda(\gf,\Aff)\twoheadrightarrow \Hn$ is given by the subspace $\Sigma=\{\tau_1=0\}$, 
and $\Hn$ identifies with $\Sigma$. Any element of $\Hn$ is thus represented by a unique tuple of the form $(0,\tau_2,\ldots,\tau_{n-1})$ in $\C^{n-1}$. 
Then, provided $i\neq 1$, the action of $\sigma_{i,j}^2$ on $\Hn$ is given by $(0,\tau_2,\ldots,\tau_{n-1})\mapsto(0,\tau'_2,\ldots,\tau'_{n-1})$
with $\tau_k'$ obtained from (\ref{descriptionaction1}) by setting $\tau_1=0$.
If $i=1$, the \textit{formulae} are the following.
\begin{equation} \label{descriptionaction2}\left \lbrace\begin{array}{l}
\tau_j'= \left (\lambda_1-(1-\lambda_j)\prod\limits_{l=1}^{j-1}\lambda_l\right)\tau_j-(1-\lambda_j)\sum\limits_{k=2}^{j-1}\left [ (1-\lambda_1)\left (\prod\limits_{l=k}^{j-1}\lambda_l^{-1} \right) +(1-\lambda_j)\prod\limits_{l=1}^{k-1} \lambda_l \right ]\tau_k \\
\\
\tau_{\nu}'=\tau_{\nu}-(1-\lambda_{\nu})\left [ \left (\prod\limits_{l=1}^{j-1}\lambda_l \right ) \tau_j +(1-\lambda_j)\sum\limits_{k=2}^{j-1}\left ( \prod\limits_{l=1}^{k-1}\lambda_l \right) \tau_k\right]\mbox{ if }\nu\neq j\\
\end{array} \right.\end{equation}

  \subsection{Main statements}

To formulate our results, we shall use the following.
\begin{definition}
 Let $\lambda\in\Hom(\gf,\C^*)$ be a linear part. The {\em nontriviality index $\iota(\lambda)$ of $\lambda$} is 
 the number of $\lambda_i$ different from 1.
\end{definition}

We first treat the special case $\iota(\lambda)=2$. 
\begin{proposition} \label{propind2}
 Let $n\geq 3$ and $\lambda: \gf\to\C^*$ be a linear part given by $(\lambda_1,\ldots,\lambda_n)=(a,1,\ldots,1,a^{-1})$ with $a\neq1$. 
 Then the action of $\PBD$ on $\PH{n}$ has $n-2$ fixed points, and the other points of $\PH{n}$ have a finite orbit if and only if $a$ is a root of unity. 
 
 More precisely, let $\rho\in\Hom_\lambda(\gf,\Aff)$ be a representation whose conjugacy class is given by the homogeneous coordinates 
 $[\tau_2:\dots:\tau_{n-1}]\in\PC{n-3}$. Set $k=\card \{i ~\vert~ 1<i<n,  \tau_i \neq 0\}$. Then $[\rho]$ is a fixed point if and only if $k=1$, 
 and if $a$ has order $\omega$, the orbit of $[\rho]$ has cardinality $\omega^{k-1}$.
\end{proposition}
\begin{proof}
 First note that conjugating the representation $\rho$ by a translation only modifies $\tau_{1}$ and $\tau_n$, the other $\tau_i$ being preserved. 
 Now, by \textit{formulae} (\ref{descriptionaction1}) and (\ref{descriptionaction2}), the action of $\sigma_{i,j}^2$ for $1< i<j< n$ is trivial, 
 and the action of $\sigma_{1,j}^2$ for $1< j< n$ is given by $[\tau_2:\ldots:\tau_{n-1}]\mapsto [\tau'_2:\ldots:\tau'_{n-1}]$, with $\tau'_j=a\tau_j$ 
 and $\tau'_l=\tau_l$ for $l \in \{2,\ldots,n-1\}\setminus \{j\}$. Hence the orbit of $[\rho]$ is given by the set of homogeneous coordinates obtained 
 from $[\tau_2:\ldots:\tau_{n-1}]$ by multiplying each coordinate by some power of $a$.
\end{proof}

For $\iota(\lambda)>2$, the following, proved in Section~\ref{seccontraintes}, reduces the study to the case $\iota(\lambda)=n$. 
\begin{theorem} \label{thnqcq}
 Let $\lambda\in\Hom(\gf,\C^*)$ be a linear part such that $\iota(\lambda)>2$.
 If $[\rho]\in\PH{n}$ has a finite orbit under the action of $\PBD$, then for each $i\in\{1,\ldots,n\}$ such that $\lambda_i=1$, 
 we have $\rho(\alpha_i)=Id_\C$. 
\end{theorem}

Note that when $n=3$, for any nontrivial linear part $\lambda$, $\PH{3}$ is a point. Thus the above theorem implies in particular that, for $n\geq4$ 
and $\iota(\lambda)=3$, the action of $\PBn$ on $\PH{n}$ has exactly one fixed point and all other orbits are infinite. 

After the reduction allowed by Theorem \ref{thnqcq}, the determination of finite orbits is completed by the next result. A detailed description 
of the finite orbits for $n=5,6$ is given in Section \ref{sec56}. 

\begin{theorem} \label{thcasfinis}
 Let $n\geq4$. Let $\lambda\in\Hom(\gf;\C^*)$ be a linear part such that $\iota(\lambda)=n$. 
 If $[\rho]\in\PH{n}$ has a finite orbit under $\PBD$ then $n\leq6$.
 
 If $n=5$ (resp. $n=6$), 
 the representations whose conjugacy classes have finite orbits are the representations with linear part given by 
 $(\lambda_1,\ldots,\lambda_5)=(\zeta,\zeta,\zeta,\zeta,\zeta^2)$ up to permutation (resp. 
 $(\lambda_1,\ldots,\lambda_6)=(\zeta,\zeta,\zeta,\zeta,\zeta,\zeta)$), where $\zeta$ is any primitive $6^{th}$ root of unity.
 
If $n=4$, 
 the representations whose conjugacy classes have finite orbits are the ones described in Tables~$\ref{tabn4/1}$, $\ref{tabn4/2}$ and $\ref{tabn4/3}$.
\end{theorem}

\begin{table}[ht]
 \begin{center}
 \scalebox{\scalefactor}{
  \begin{tabular}{|c|c|c|c|} 
   \hline Group $\Gamma_\lambda$ & Linear part & Translation part & Size of the orbit \\ 
   \hline \multirow{7}{*}{Reducible} & $(1,1,1,1)$ & any & 1 \\
   \cline{2-4} & $(1,1,a,a^{-1})$ & $(0,1,-1,0)$ & \multirow{2}{*}{1} \\
   \cline{3-3} & $a\neq1$ & $(1,0,-1,0)$ & \\
   \cline{3-4} & $a$ $k$-th root of 1 & other orbits & $k$ \\
   \cline{2-4} & & $(0,1,-1,0)$ & \multirow{2}{*}{1} \\
   \cline{3-3} & $(1,1,a,a^{-1})$  & $(1,0,-1,0)$ & \\
   \cline{3-4} & $a$ not root of 1 & other orbits & infinity \\
  \cline{2-4} & $(1,\lambda_2,\lambda_3,\lambda_4)$ & $(0,0,-\lambda_3,1)$ & 1 \\
   \cline{3-4} & $\lambda_i\neq1$ & other orbits & infinity \\
  \hline \multirow{6}{*}{$\begin{array}{l}\mbox{Irreducible}\\ \mbox{imprimitive} \end{array}$}& & $(0,1,a,0)$ & 2 \\
   \cline{3-4} & $(a,-a^{-1},-a^{-1},a)$ & $(0,0,1,a)$ & \multirow{2}{*}{$\ord(a^2$)} \\ 
   \cline{3-3} & $a^2\neq 1$ primitive $n$-th root of $1$ & $\begin{array}{l}(0,1,0,-a^2)\mbox{ if }n\mbox{ even,}\\ (0,2,1+a,a-a^2)\mbox{ otherwise.}\end{array}$ & \\
   \cline{3-4} & & other orbits & $2\,\ord(a^2)$ \\
   \cline{2-4} & $(a,-a^{-1},-a^{-1},a)$ & $(0,1,a,0)$ & 2 \\
   \cline{3-4} & $a$ not root of 1 & other orbits & infinity \\
   \hline 
  \end{tabular}}\vspace{1ex}\caption{Finite orbits for $n=4$: reducible and irreducible imprimitive cases.} \label{tabn4/1} 
 \end{center} 
\end{table}
\begin{table}[ht]
 \begin{center}
 \scalebox{\scalefactor}{
  \begin{tabular}{|c|c|c|c|} 
   \hline Group $\Gamma_\lambda$ & Linear part & Translation part & Size of the orbit \\ 
   \hline \multirow{8}{*}{Tetrahedral} & & $(\eta^3,0,0,1)$ & \multirow{2}{*}{4} \\
   \cline{3-3} & $(\eta,\eta^5,\eta^3,\eta^3)$ & $(0,\eta^5,-1,0)$ & \\
   \cline{3-4} & $\eta$ primitive $12^{th}$ root of 1 & $(0,0,\eta^3,-1)$ & 6 \\
   \cline{3-4} & & other orbits & 12 \\
   \cline{2-4} & & $(0,0,\zeta,-1)$ & \multirow{2}{*}{4} \\
   \cline{3-3} & $(-1,\zeta,\zeta,\zeta)$ & $(0,1-\zeta^2,\zeta,\zeta)$ & \\
   \cline{3-4} & $\zeta$ primitive $6^{th}$ root of 1 & $(0,1,\zeta+\eta,\zeta^2\eta+\zeta-1),\,\eta^2=\zeta$ & 6 \\
   \cline{3-4} & & other orbits & 12 \\
   \hline \multirow{8}{*}{Octahedral} & & $(0,0,1,\eta^5)$ & 6 \\
   \cline{3-4} & $(\eta,\eta^5,\eta^7,\eta^{11})$ & $(0,1,0,1)$ & 8 \\
   \cline{3-4} & $\eta$ primitive $24^{th}$ root of 1 & $(\eta,0,0,1)$ & 12 \\
   \cline{3-4} & & other orbits & 24 \\
   \cline{2-4} & & $(\eta^4,0,0,1)$ & 6 \\
   \cline{3-4} & $(\eta,-\eta,\eta^2,\eta^2)$ & $(0,0,\eta^2,-1)$ & 8 \\
   \cline{3-4} & $\eta$ primitive $12^{th}$ root of 1 & $(0,\nu+\eta^3,\nu^{-1},1),\,\nu^2=\eta$ & 12 \\
   \cline{3-4} & & other orbits & 24 \\
   \hline 
  \end{tabular}}\vspace{1ex}\caption{Finite orbits for $n=4$: tetrahedral and octahedral cases.} \label{tabn4/2} 
 \end{center} 
\end{table}
\begin{table}[ht]
 \begin{center}
  \scalebox{\scalefactor}{
  \begin{tabular}{|c|c|c|} 
   \hline Linear part & Translation part & Size of the orbit \\ 
   \hline  & $(0,1,0,\alpha^{50})$ & 12 \\
   \cline{2-3} $(\alpha,\alpha^{29},\alpha^{11},\alpha^{19})$ & $(1,0,0,\alpha^{49})$ & 20 \\
   \cline{2-3} $\ord(\alpha)=60$ & $(0,0,1,\alpha^{19})$ & 30 \\
   \cline{2-3}  & other orbits & 60 \\
   \cline{1-3} & $(\alpha^{12},0,0,\alpha^{5})$ & 12 \\
   \cline{2-3} $(\alpha,\alpha^{9},\alpha^{7},\alpha^{3})$ & $\left(0,\beta+\alpha^{5}+\alpha^{2}+\alpha,1+\alpha^{3},(\beta+\alpha)(1-\alpha^3\beta)\right)$ & 20 \\
   \cline{2-3} $\ord(\beta)=60$& $(0,0,\alpha^{17},1)$ & 30 \\
   \cline{2-3} $\alpha=\beta^3$ & other orbits & 60 \\
   \cline{1-3} & $(0,0,\alpha^{16},1)$ & 12 \\
   \cline{2-3} $(\alpha^9,\alpha^{9},\alpha,\alpha^{11})$ & $(\alpha^{10},0,0,\alpha^6)$ & 20 \\
   \cline{2-3} $\ord(\beta)=60$ & $(0,(\alpha+1)(\alpha^{11}+1),\beta+\alpha^{14}-\alpha^{5}+\alpha^{4},(\beta+1)\alpha^{14}+\alpha^4)$ & 30 \\
   \cline{2-3}  $\alpha=\beta^2$& other orbits & 60 \\
   \cline{1-3} & $(\alpha^{11},0,0,-1)$ & 12 \\
   \cline{2-3} $(\alpha^5,\alpha^{5},\alpha,\alpha^{19})$ & $(0,0,1,\alpha^{14})$ & 20 \\
   \cline{2-3} $\ord(\beta)=60$  &$(0,1,\alpha^5(1+\beta),\alpha^4(\alpha^{10}-\beta))$& 30 \\
   \cline{2-3} $\alpha=\beta^2$ & other orbits & 60 \\
   \cline{1-3}  & $(\alpha^{12},0,0,-\alpha^{5})$ & 12 \\
   \cline{2-3}  $(\alpha,\alpha^{4},\alpha^2,\alpha^{8})$ & $(0,0,\alpha^{2},-1)$ & 20 \\
   \cline{2-3}  $\ord(\beta)=60$ & $(0,\beta^{56}+\beta^{13}+\beta^{12}+\beta^8-\beta^6-\beta^2,1+\beta^{22},\beta^{19}+\beta^{18}-\beta^8)$ & 30 \\
   \cline{2-3} $\alpha=\beta^4$  & other orbits & 60 \\
   \cline{1-3}  & $(0,0,\alpha,1)$ & 12 \\
   \cline{2-3}  $(-\alpha,-\alpha,-\alpha,-\alpha^2)$ & $(0,\gamma^{5}+\gamma^4-1,\gamma^4+\gamma-\alpha,\alpha\gamma-\gamma^{5})$& 20 \\
   \cline{2-3}  $\ord(\beta)=60$& $(0,\beta^{11}+\beta^9-\beta,\beta^9-1,\alpha(\beta^6+\beta^3))$ & 30 \\
   \cline{2-3} $\alpha=\beta^{12}, \gamma=\beta^2$ & other orbits & 60 \\
   \hline 
  \end{tabular}}\vspace{1ex}\caption{Finite orbits for $n=4$: icosahedral case.} \label{tabn4/3} 
 \end{center} 
\end{table}

\section{Case of four punctures} \label{sec4}

  \subsection{Linear parts that give finite orbits} \label{subsec4linear}

\begin{lemma}
 Fix a linear part $\lambda:\Lambda_4\to\C^*$. The image $\Gamma_\lambda$ of the associated map $\PBDnum{4}\to \mathrm{PGL}(\HH{4})$ 
 is generated by the images of $\sigma_1^2$ and $\sigma_2^2$.
 \end{lemma} 
 \begin{proof} 
Let us first mention the following result -- see \cite[\S $2.1$]{cousinisom} for a proof.
\begin{sublemma}\label{sublemma}
Consider the injection
$F_{3,n-3}\Pu\stackrel{i}{\hookrightarrow}F_{0,n}\Pu.$ The composition \[\pi_1(F_{3,n-3}\Pu,x)\stackrel{i^*}{\longrightarrow}{\pi_1(F_{0,n}\Pu,x)} \stackrel{\varphi}{\longrightarrow} \M_n(\Pu)\] induces an anti-isomorphism 
$\pi_1(F_{3,n-3}\Pu,x)\cong \PM_n(\Pu)$.
\end{sublemma} 
The map $\PBDnum{4}\to \mathrm{PGL}(\HH{4})$ decomposes into $$\PBDnum{4}\to\mathrm{Out}(\gf)\to\mathrm{PGL}(\HH{4}).$$ 
In section \ref{subsecprelim}, we have seen that the map $\PBDnum{4}\to\mathrm{Out}(\gf)$ factorizes through 
$\varphi \circ {\chi_*}_{\vert\PBDnum{4}}:\PBDnum{4}\rightarrow \M_4\Pu$, $\beta\mapsto\clmcg{\beta}$. 
By the sublemma, $\PM_4\Pu$ is the image of this latter map and is equal to the image of $\pi_1(F_{3,1}\Pu,x)\cong\pi_1(\Pu\setminus\{x_2,x_3,x_4\},x_1)$. 
Thus $\PM_4\Pu$ is generated by $\clmcg{\sigma_1^2}$ and $\clmcg{\sigma_{1,3}^2}=\clmcg{\sigma_1\sigma_2^2\sigma^{-1}_1}$. This yields the lemma 
since the automorphism $g\mapsto \clmcg{\sigma_1}^{-1}g\clmcg{\sigma_1}$ of $\M_4\Pu$ fixes the subgroup $\PM_4\Pu$.
\end{proof}

The next result will be useful in the study of the group $\Gamma_\lambda$. It was already known to Schwarz in his work on hypergeometric 
functions \cite{MR1579568}. A modern reference is \cite[Theorem $6.1$]{MR1731936}. The statement will rely on the following.
\begin{definition}
Let $V$ be a $\C$-vector space.
A subgroup $G$ of $\mathrm{GL}(V)$ is called \emph{reducible} if it fixes globally a nontrivial subspace of $V$, \emph{irreducible} otherwise.
The group $G$ is called \emph{imprimitive} if there exists a nontrivial decomposition $V=\oplus_{i \in I}V_i$ whose terms are permuted 
by the elements of $G$, {\em primitive} otherwise.
\end{definition}

\begin{theorem} \label{thChur}
Let $A,B\in \mathrm{SL}_2(\C)$. Define $C=AB$. 
Denote $t_A,t_B,t_C$ the traces of these matrices.
The subgroup $G=\langle A,B\rangle \subset \mathrm{SL}_2(\C)$ is irreducible if and only if $t_A^2+t_B^2+t_C^2-t_At_Bt_C\neq 4$.
In this case, the conjugacy class of the triple $(A,B,C)$ is uniquely determined by the triple of traces $(t_A,t_B,t_C)$.
Moreover, if $G$ is irreducible:
\begin{itemize}
 \item $G$ is imprimitive if and only if at least two of the three traces $t_A$, $t_B$, $t_C$ vanish. In this case, 
   $G$ is finite if and only if two traces vanish and the third trace has the form $a+1/a$ for some root of unity $a$. 
   In this latter case, $G$ is projectively the dihedral group of order $2\, \ord(a^2)$.
 \item $G$ is finite and primitive if and only if $G$ is projectively tetrahedral, octahedral or icosahedral. Moreover:
   \begin{itemize}
    \item $G$ is projectively tetrahedral if and only if $t_A^2+t_B^2+t_C^2-t_At_Bt_C = 2$ and $t^2_A,t^2_B,t^2_C\in\{0,1\}$,
    \item $G$ is projectively octahedral if and only if $t_A^2+t_B^2+t_C^2-t_At_Bt_C = 3$ and $t^2_A,t^2_B,t^2_C\in\{0,1,2\}$,
    \item $G$ is projectively icosahedral if and only if $t_A^2+t_B^2+t_C^2-t_At_Bt_C \in\{2-\mu_2,3,2+\mu_1\}$ and $t^2_A,t^2_B,t^2_C\in\{0,1,\mu^2_1,\mu^2_2\}$,
     where $\mu_1=\frac{1}{2}(1+\sqrt{5})$ and $\mu_2=\mu_1^{-1}=\frac{1}{2}(\sqrt{5}-1)$.
   \end{itemize}
 \item $G$ is Zariski dense in $\mathrm{SL}_2(\C)$ if it is neither imprimitive nor finite.
\end{itemize}
\end{theorem}
\begin{rem}
Note that all cases of finite groups above correspond to algebraic traces -- $t_A,t_B,t_C\in \overline{\Q}$ -- and that these cases are stable under Galois 
automorphisms of $\Q\rightarrow \overline{\Q}$. This will be implicitly used several times in the sequel. 
\end{rem}
Fix a linear part $\lambda:\Lambda_4\to\C^*$ given by $(\lambda_1,\ldots,\lambda_4)\in(\C^*)^4$. Assume $\lambda_1\neq1$. 
By \textit{formulae} (\ref{descriptionaction1}) and (\ref{descriptionaction2}), using the coordinates $(\tau_2,\tau_3)$ for the class of representations given 
by $[0,\tau_2,\tau_3]$, the action of the braids $\sigma_1^2$ and $\sigma_2^2$ on $\HH{4}$ is given by multiplication of column vectors by the following matrices. 
\begin{equation} \label{mat4trous} \sigma_1^2 \mapsto \hat{A}_3=\begin{pmatrix} 
   \lambda_1\lambda_2 & 0\\
   \lambda_1(\lambda_3-1)&1
  \end{pmatrix}
  \quad \textrm{ and } \quad \sigma_2^2 \mapsto
  \hat{A}_1=\begin{pmatrix} 
    \lambda_2(\lambda_3-1)+1 & \lambda_2(1-\lambda_2)  \\
   1-\lambda_3 & \lambda_2  \\
  \end{pmatrix}.
  \end{equation}
Fix square roots $\sqrt{\lambda_1},\sqrt{\lambda_2},\sqrt{\lambda_3}$ for $\lambda_1,\lambda_2,\lambda_3$, such that $\lambda_i=\lambda_j$ implies 
$\sqrt{\lambda_i}=\sqrt{\lambda_j}$. The matrices $A_3=\frac{1}{\sqrt{\lambda_1}\sqrt{\lambda_2}}\hat{A}_3 $ and 
$A_1=\frac{1}{\sqrt{\lambda_2}\sqrt{\lambda_3}}\hat{A}_1$ belong to $\mathrm{SL_2(\C)}$ and induce the same transformations of $\PH{4}$ as $\hat{A}_3$ 
and $\hat{A}_1$. Set $A_2=A_1A_3$ and denote by $t_i$  the trace of $A_i$, for $i=1,2,3$. The notation is arranged so that, if $\{i,j,k\}=\{1,2,3\}$ we have: 
\[t_i=\sqrt{\lambda_j}\sqrt{\lambda_k}+\frac{1}{\sqrt{\lambda_j}\sqrt{\lambda_k}}\]

We can now apply Theorem~\ref{thChur}. 

Set $P(\lambda):=t_1^2+t_2^2+t_3^2-t_1t_2t_3\in\C$. 
Set $G:=\langle A_1,A_2\rangle =\langle A_1,A_3\rangle \subset \mathrm{SL}_2(\C)$ so that $\Gamma_\lambda$ is the image of $G$ in $\mathrm{PGL}_2(\C)$. We denote $\bar{A}_i$ the image 
of $A_i$ in $\mathrm{PGL}_2(\C)$, $i=1,2,3.$
\subsubsection*{\textbf{Reducible case}} \label{linred}
We have $P(\lambda)=4+\prod_{i=1}^4(1-\lambda_i)$. Hence $G$ is reducible if and only if $\iota(\lambda)<4$.

In this case, the transformations of $\Gamma_\lambda$ have a common fixed point $p \in \Pu$ and we interpret them as transformations of the affine line 
$\Pu\setminus p$. If $\Gamma_\lambda$ contains a nontrivial translation, all the orbits in $\Pu\setminus p$ are infinite. 
It is the case in particular if $\Gamma_\lambda$ is non commutative, since $\bar{A}_1\bar{A}_2\bar{A}_1^{-1}\bar{A}_2^{-1}$ is a translation. 
By Proposition~\ref{propind2}, it implies $\iota(\lambda)=3$. 

If $\Gamma_\lambda$ contains no nontrivial translation, then $\Gamma_\lambda$ is abelian. If $\bar{A}_1$ or $\bar{A}_3$ is the identity, 
recalling $\lambda_1\dots \lambda_4=1$, it is clear from the expressions of $A_1$ and $A_3$ that $\iota(\lambda)\leq2$. If 
$\bar{A}_1,\bar{A}_3\neq\mathrm{Id}$, they both have exactly two fixed points, 
one of which being $p$. Since they commute, they have the same two fixed points, {\em i.e.} $A_1$ and $A_3$ have the same eigenvectors. 
Now $(0,1)$ is an eigenvector for $A_3$, thus it is an eigenvector for $A_1$, and $\lambda_2=1$. It implies that $(1,-1)$ is an eigenvector for $A_1$, 
hence also for $A_3$, and $\lambda_1\lambda_3=1$. It follows that $\lambda_4=1$ and $\iota(\lambda)=2$. 

The following easy lemma will be useful in the next cases.
\begin{lemma}
 Let $x\in\C^*$ and set $t=x+\frac{1}{x}$. 
 \begin{itemize} \label{lemmatraces}
  \item $t=0$ if and only if $x^2=-1$, 
  \item $t^2=1$ if and only if $x^2$ is a primitive root of unity of order 3, 
  \item $t^2=2$ if and only if $x^2$ is a primitive root of unity of order 4, 
  \item $t^2=\mu_1^2$ if and only if $x^2$ is a primitive root of unity of order 5 with positive real part, 
  \item $t^2=\mu_2^2$ if and only if $x^2$ is a primitive root of unity of order 5 with negative real part. 
 \end{itemize}
\end{lemma}

\subsubsection*{\textbf{Irreducible imprimitive case}}\label{linimprimitive}
The group $G\subset \SL_2(\C)$ is irreducible and imprimitive if and only if for a suitable coordinate $z : \Pu \rightarrow \C \cup \{\infty\}$, 
the projective transformation group $\Gamma_\lambda=\mathbb{P}G$ is given by $\cup_{u\in U} \{z\mapsto uz, z\mapsto u/z\}$, for some nontrivial subgroup $U$ of $\C^*$.
We see that $\{z=0,z=\infty\}$ is an orbit of order $2$ under the action of $\Gamma_\lambda$.
If $U$ is infinite, all the other orbits of $\Gamma_\lambda$ are infinite.
 
If $U$ is finite, set $U=\langle \mu\rangle$, $n=\ord(\mu)$, and fix a square root $\sqrt{\mu}$. A point $z$ has an orbit of size $n$ 
if and only if it is fixed by an element of $G$ of order $2$, necessarily of the form $z\mapsto \mu^{2\ell}/z$ or $z\mapsto \mu^{2\ell+1}/z$, 
hence $z=\pm \mu^{\ell}$ or $z=\pm \mu^{\ell}\sqrt{\mu}$, respectively. If $n$ is odd, $\sqrt{\mu}$ has the form $\mu^k$ and there are exactly two orbits of order $n$, 
the one of $z=1$ and the one of $z=-1$. If $n$ is even, $-1\in U$ and there are again exactly two size $n$ orbits: the ones of  $z=1$ and $z=\sqrt{\mu}$.
Other orbits have size $2n$.

By Theorem~\ref{thChur}, if $G$ imprimitive implies that at least two of the three traces $t_i$ vanish. Hence, for suitable indices $i,j,k$ with $\{i,j,k\}=\{1,2,3\}$, 
we have $t_i=0=t_j$. By Lemma~\ref{lemmatraces}, this is tantamount to $\lambda_k=-\frac{1}{\lambda_i}$ and $\lambda_j=\lambda_i$. 
Irreducibility of $G$ implies $\lambda_i^2\neq 1$. Finally, $(\lambda_i,\lambda_j,\lambda_k)=(a,a,-\frac{1}{a})$ for some $a\in\C^*\setminus\{\pm1\}$.

In this case, $t_k=a+\frac{1}{a}$. Hence, if $a$ is not 
a root of unity, $\Gamma_{\lambda}$ has a unique finite orbit in $\Pu\simeq \PH{4}$, of order $2$, given by the fixed points of $\bar{A}_k$. 
If $a$ is a root of unity, Theorem~\ref{thChur} implies that $\Gamma_{\lambda}$ is conjugate to the dihedral group of order $2n=2\, \ord(a^2)$. 
According to the above discussion, a system of representatives for the special orbits are then given as follows for $n\neq2$.
For the orbit of order $2$, take a fixed point of $\bar{A}_k$.
If $n$ is odd:  for the two orbits of order $n$ take the two fixed points of $\bar{A}_i$.
If $n$ is even: for the two orbits of order $n$ take one fixed point of $\bar{A}_i$ and one fixed point of $\bar{A}_j$. 
If $n=2$, the three traces vanish, and the transformation whose fixed points form the order 2 orbit is identified by its order, namely 4.

\subsubsection*{\textbf{Tetrahedral case}}
By Theorem~\ref{thChur}, $\Gamma_\lambda$ is tetrahedral if and only of $(t_i^2,t_j^2,t_k^2)=(0,1,1)$ or $(1,1,1)$, for suitable $i,j,k$ 
such that $\{i,j,k\}=\{1,2,3\}$, and $P(\lambda)=2$. 
 By Lemma~\ref{lemmatraces}, the equality $(t_i^2,t_j^2,t_k^2)=(0,1,1)$ occurs if and only if  we have $$\left\{\begin{array}{l} \lambda_j\lambda_k=-1 \\ 
\lambda_i\lambda_k=\xi_1 \\ \lambda_i\lambda_j=\xi_2 \end{array}\right.,$$ where $\xi_1$ and $\xi_2$ are primitive $3^{rd}$ roots of unity. 
It implies $\lambda_j^2=-\xi_1^2\xi_2$. If $\xi_1=\xi_2$, we obtain $(\lambda_i,\lambda_j,\lambda_k,\lambda_4)=(-\zeta\xi,\zeta,\zeta,-\zeta\xi^2)$ where 
$\zeta$ (resp. $\xi$) is a primitive $4^{th}$ (resp. $3^{rd}$) root of unity. Check that $P(\lambda)=2$. If $\xi_1\neq\xi_2$, 
we obtain $(\lambda_i,\lambda_j,\lambda_k,\lambda_4)=(\eta^3,\eta,\eta^5,\eta^3)$ where $\eta$ is a primitive $12^{th}$ root of unity. 
Note that it is a permutation of the previous case, with $\eta=-\zeta\xi$. 

Now assume $t_i=1$, $i=1,2,3$. By Lemma~\ref{lemmatraces}, the three products $\lambda_i\lambda_j$,$\lambda_i\lambda_k$,$\lambda_j\lambda_k$ must be third roots of unity.
Hence, either two of them are equal or they are all equal.
In the first case, for suitable choice of $i,j,k$ 
\[ \left \lbrace \begin{array}{l}
\lambda_j \lambda_k=\xi \\
 \lambda_i\lambda_k=\xi \\
  \lambda_i\lambda_j=\xi^2 \\
   \end{array}\right.,\]
with $\xi$ of order $3$. 
It implies $\lambda_i^2=\xi^2$, $\lambda_i=\xi$ cannot occur because it leads to $\lambda_k=1$ and $P(\lambda)=4$. We conclude
 $(\lambda_i,\lambda_j,\lambda_k,\lambda_4)=(-\xi,-\xi,-1,-\xi)$, which is compatible with $P(\lambda)=2$.
 In the second case the three products equal some $3^{rd}$ root of unity $\xi$, $\lambda_1^2=\xi$, and $\lambda_1$ cannot be $\xi^2$ because it leads to $P(\lambda)=4$. 
 We conclude $(\lambda_1,\lambda_2,\lambda_3,\lambda_4)=(-\xi^2,-\xi^2,-\xi^2,-1)$, which is compatible with $P(\lambda)=2$.

\subsubsection*{\textbf{Octahedral case}}
Assume $\Gamma_\lambda$ is octahedral. Then, by Theorem~\ref{thChur}, for suitable $i,j,k$ such that 
$\{i,j,k\}=\{1,2,3\}$, $(t_i^2,t_j^2,t_k^2)=(0,1,2)\textrm{ or }(1,2,2)$ and $P(\lambda)=3$. 
First assume $(t_i^2,t_j^2,t_k^2)=(0,1,2)$. By Lemma~\ref{lemmatraces}, we have
\[\left\{\begin{array}{l} \lambda_j\lambda_k=-1 \\ 
\lambda_i\lambda_k=\xi \\ \lambda_i\lambda_j=\zeta \end{array}\right.,\]
with $\ord(\xi)=3$ and $\ord(\zeta)=4$. It implies $\lambda_i^2=-\xi\zeta$, and we obtain $(\lambda_i,\lambda_j,\lambda_k,\lambda_4)=(\eta,\eta^{5},\eta^7,\eta^{11})$, 
where $\eta$ is a primitive $24^{th}$ root of unity. 
Check that $P(\lambda)=3$.

Now assume $t_i^2=1,\ t_j^2=t_k^2=2$. By Lemma~\ref{lemmatraces}, we have $$\left\{\begin{array}{l} \lambda_j\lambda_k=\xi \\ 
\lambda_i\lambda_k=\zeta_1 \\ \lambda_i\lambda_j=\zeta_2 \end{array}\right.,$$ where $\xi$ is a primitive $3^{rd}$ root of unity, 
and $\zeta_1$ and $\zeta_2$ are primitive $4^{th}$ roots of unity. It implies $\lambda_i^2=\zeta_1\zeta_2\xi^2$. If $\zeta_1=\zeta_2=\zeta$, then 
$\lambda_i=\eta:=\pm \zeta \xi$.  In particular, it has order $12$. We then find $(\lambda_i,\lambda_j,\lambda_k,\lambda_4)=(\eta,-\eta^2,-\eta^2,-\eta)$ or 
$(\lambda_i,\lambda_j,\lambda_k,\lambda_4)=(\eta,\eta^2,\eta^2,-\eta)$. Computation of $P(\lambda)$ excludes only the first case.

If $\zeta:=\zeta_1=-\zeta_2$, we obtain $\lambda_i=\pm\xi$, hence $(\lambda_i,\lambda_j,\lambda_k,\lambda_4)=\pm(\xi,-\zeta\xi^2,\zeta\xi^2,\xi)$. Setting 
$\eta:=\pm\zeta\xi^2$, we recover the previous cases up to permutation.

\subsubsection*{\textbf{Icosahedral case}}
Assume $\Gamma_\lambda$ is icosahedral. Then, by Theorem~\ref{thChur}, $\{t_1^2,t_2^2,t_3^2\}$ is one of $\{0,1,\mu_m^2\}$, 
$\{0,\mu_1^2,\mu_2^2\}$, $\{1,1,\mu_m^2\}$, $\{1,\mu_m^2,\mu_n^2\}$, $\{\mu_m^2,\mu_m^2,\mu_m^2\}$, for $m,n\in\{1,2\}$, and $P(\lambda)\in\{2-\mu_2,3,2+\mu_1\}$.

Assume $t_i^2=0$, $t_j^2=1$ and $t_k^2\in\{\mu_1^2,\mu_2^2\}$, for a suitable permutation  $\{i,j,k\}=\{1,2,3\}$. By Lemma~\ref{lemmatraces}, we have $$\left\{\begin{array}{l} \lambda_j\lambda_k=-1 \\ 
\lambda_i\lambda_k=\xi \\ \lambda_i\lambda_j=\zeta \end{array}\right.,$$ where $\xi$ (resp. $\zeta$) is a primitive $3^{rd}$ root (resp. $5^{th}$ root) of unity. 
It implies $\lambda_j^2=-\zeta\xi^2$. We obtain $(\lambda_i,\lambda_j,\lambda_k,\lambda_4)=(\eta^{-1}\zeta^3,\eta\zeta^3,-\eta^{-1}\zeta^2,-\eta\zeta^2)$, 
where $\eta$ is a primitive $12^{th}$ root of unity. Check that $P(\lambda)=2+\mu_1\textrm{ or }2-\mu_2$, depending on the value of $\zeta$. Setting $\alpha:=\eta\zeta^3$, this can be rewritten as
 $(\lambda_i,\lambda_j,\lambda_k,\lambda_4)=(\alpha^{11},\alpha,\alpha^{29},\alpha^{19})$, with $\alpha$ of order $60$.

Assume $t_i^2=0$, $t_j^2=\mu_1^2$ and $t_k^2=\mu_{2}^2$. By Lemma~\ref{lemmatraces}, we have $$\left\{\begin{array}{l} \lambda_j\lambda_k=-1 \\ 
\lambda_i\lambda_k=\zeta \\ \lambda_i\lambda_j=\zeta' \end{array}\right.,$$ where $\zeta$ is a primitive $5^{th}$ root of unity and $\zeta'=\zeta^2\textrm{ or }\zeta^3$. 
It implies $\lambda_j^2=-\zeta^{-1}\zeta'$. We obtain 
$(\lambda_i,\lambda_j,\lambda_k,\lambda_4)=(-\nu\zeta^4,\nu\zeta^3,\nu\zeta^2,-\nu\zeta)\textrm{ or }(-\nu\zeta^2,\nu\zeta,\nu\zeta^4,-\nu\zeta^3)$, where $\nu$ is a primitive $4^{th}$ root of unity. Setting $\alpha:=\lambda_j$, this can be written as $(\lambda_i,\lambda_j,\lambda_k,\lambda_4)=(\alpha^3,\alpha,\alpha^9,\alpha^7)$ and $(\lambda_i,\lambda_j,\lambda_k,\lambda_4)=(\alpha^7,\alpha,\alpha^9,\alpha^3)$ respectively, with $\alpha$ of order~$20$. This is enough to check $P(\lambda)=3$.

Assume $t_i^2=1$, $t_j^2=1$ and $t_k^2\in\{\mu_1^2,\mu_2^2\}$. By Lemma~\ref{lemmatraces}, we have $$\left\{\begin{array}{l} \lambda_j\lambda_k=\xi \\ 
\lambda_i\lambda_k=\xi_0 \\ \lambda_i\lambda_j=\zeta \end{array}\right.,$$ where $\xi$ and $\xi_0$ are primitive $3^{rd}$ roots of unity, 
and $\zeta$ is a primitive $5^{th}$ root of unity. It implies $\lambda_j^2=\zeta\xi_0^2\xi$. 
If $\xi_0=\xi$, then $\lambda_j=\pm\zeta^3$. The case $\lambda_j=\zeta^3$ provides $P(\lambda)=3+2\mu_1\textrm{ or }3-2\mu_2$, 
depending on the value of $\zeta$, and these are both unauthorized values for $P(\lambda)$. If $\lambda_j=-\zeta^3$, we obtain 
$(\lambda_i,\lambda_j,\lambda_k,\lambda_4)=(-\zeta^3,-\zeta^3,-\xi\zeta^2,-\xi^2\zeta^2)$ and $P(\lambda)=3$. 
Setting $\alpha:=\lambda_k$, this can be written as $(\lambda_i,\lambda_j,\lambda_k,\lambda_4)=(\alpha^9,\alpha^9,\alpha,\alpha^{11})$, with $\alpha$ of order~$30$.
If $\xi_0=\xi^2$, then $\lambda_j=\pm\zeta^3\xi$, and we recover the previous cases up to permutation and $\zeta\leftrightarrow \zeta^{-1}$.

Assume $t_i^2=1$, $t_j^2,t_k^2\in\{\mu_1^2,\mu_2^2\}$. By Lemma~\ref{lemmatraces}, we have $$\left\{\begin{array}{l} \lambda_j\lambda_k=\xi \\ 
\lambda_i\lambda_k=\zeta_1 \\ \lambda_i\lambda_j=\zeta_2 \end{array}\right.,$$ where $\xi$ is a primitive $3^{rd}$ root of unity, 
and $\zeta_1$ and $\zeta_2$ are primitive $5^{th}$ roots of unity. It implies $\lambda_j=\pm\xi^2\zeta_1^2\zeta_2^3$. 
Up to permutation and $\xi \leftrightarrow \xi^2$, we obtain 
$\lambda=\pm(\xi,\xi,\xi^2\zeta,\xi^2\zeta^4)$ or $\pm(\xi\zeta,\xi\zeta^4,\xi^2\zeta^2,\xi^2\zeta^3)$, 
where  $\zeta=\zeta_1$.
Among these possibilities, the ones that provide authorized values of $P(\lambda)$ are $-(\xi,\xi,\xi^2\zeta,\xi^2\zeta^4)$ and 
$(\xi\zeta,\xi\zeta^4,\xi^2\zeta^2,\xi^2\zeta^3)$. 
Setting $\alpha=-\xi^2\zeta$ and $\beta=\xi \zeta$, these are respectively rewritten as $\lambda=(\alpha^5,\alpha^5,\alpha,\alpha^{19})$ and 
$\lambda=(\beta,\beta^4,\beta^2,\beta^8)$, with $\ord(\alpha)=30$, and $\ord(\beta)=15$.

Assume $t_i^2=t_j^2=t_k^2=\mu_m^2$ for $m=1$ or $m=2$. By Lemma~\ref{lemmatraces}, we have 
$$\left\{\begin{array}{l} \lambda_j\lambda_k=\zeta \\ \lambda_i\lambda_k=\zeta^{\varepsilon_1} \\ \lambda_i\lambda_j=\zeta^{\varepsilon_2} \end{array}\right.,$$ 
where $\zeta$ is a primitive $5^{th}$ root of unity, 
and $\varepsilon_\ell=\pm1$ for $\ell=1,2$. It implies $\lambda_j=\pm\zeta^{3+2\varepsilon_1+3\varepsilon_2}$. We obtain, up to permutation and change 
of primitive $5^{th}$ root of unity, $\lambda=\pm(\zeta,\zeta,\zeta,\zeta^2)$. Only the ``$-$'' case provides an authorized value of $P(\lambda)$.

\subsubsection*{\textbf{Zariski dense case}} 
If a subgroup $\Gamma<\mathrm{SL}_2(\C)$ is Zariski dense, then so is any finite index subgroup of $\Gamma$. 
For any $z\in \Pu$, the stabilizer $\Gamma_z$ is reducible and cannot be Zariski dense. This implies that $\Gamma_z$ is not a finite index subgroup 
of $\Gamma$ and that the orbit $\Gamma \cdot z$ is infinite.

  \subsection{Translation parts}\label{translationparts section}
  
We now complete the data of Tables \ref{tabn4/1}, \ref{tabn4/2} and \ref{tabn4/3} by computing the translation part of a representative 
of a conjugacy class in each finite orbit of non-generic size.

Recall the matrices $A_3,A_1\in \mathrm{SL}_2(\C)$ given in (\ref{mat4trous}). They give respectively the action of $\sigma_1^2$ and $\sigma_2^2$ 
on the conjugacy class of a representation $\rho$ given by $(0,\tau_2,\tau_3)\in\C^3$. We also denote $A_2=A_1A_3$ the matrix for $\sigma_1^2\sigma_2^2$.
Remark that, if $\lambda_1\neq 1$, the conjugacy class of the representation given by $(\tau_2,\tau_3)=(1-\lambda_2,1-\lambda_3)$ (resp. $(1,0)$, $(0,1)$) 
is always a fixed point of $A_1$ 
(resp. $A_2$, $A_3$). The triple $(\tau_1,\tau_2,\tau_3)=(0,1-\lambda_2,1-\lambda_3)$ is conjugacy equivalent to $(\tau_1,\tau_2,\tau_3)=(1,0,0)$.
The value of $\tau_4$ can be recovered using the \emph{formula} $\tau_1+\lambda_1\tau_2+\lambda_1\lambda_2\tau_3+\lambda_1\lambda_2\lambda_3\tau_4=0$ 
which follows from the relation $\alpha_1\alpha_2\alpha_3\alpha_4=1$ in $\Lambda_4$.

In view of the discussion in Section \ref{linred}, the reducible case is easily treated using Proposition~\ref{propind2}. In the imprimitive case, 
the discussion in Section \ref{linimprimitive} reduces the description of special orbits to fixed point computations for the matrices $A_i$.

 We now focus on the Platonic cases, in which the action on $\Pu$ of any $\bar{A}_i$ has finite order. This order is equal to the order of $\lambda_j\lambda_k$ 
 where $\{i,j,k\}=\{1,2,3\}$. 
\subsubsection*{\textbf{Tetrahedral case}}
In this case, $\Gamma_\lambda$ acts as the group of rotations of a regular tetrahedron. It has order $12$, and it contains $8$ elements of order $3$, 
namely the rotations around an axis passing through a vertex and the center of a face, and $3$ elements of order $2$, namely the rotations around an axis passing 
through the centers of $2$ edges. The non-generic orbits are the orbits of the vertices and of the centers of faces, of order $4$, and the orbit of the centers 
of edges, of order~$6$. 

If $(\lambda_1,\ldots,\lambda_4)=(\eta,\eta^5,\eta^3,\eta^3)$ with $\eta$ a primitive $12^{th}$ root of unity, 
$\bar{A}_3$ has order $2$, hence its fixed points are elements of the orbit of order $6$, and $\bar{A}_1$ has order $3$, hence its fixed points give 
representatives for the two distinct orbits of order $4$. 

If $(\lambda_1,\ldots,\lambda_4)=(-1,\zeta,\zeta,\zeta)$ with $\zeta$ primitive $6^{th}$ root of unity, 
 $\bar{A}_1, \bar{A}_2$ and $\bar{A}_3$ have order $3$, hence the fixed points of one of those provide representatives for the two order $4$ orbits, but we need 
to find an element of order~$2$ to reach the order $6$ orbit. Recall that the tetrahedral group is isomorphic to the direct permutation group $\mathfrak{A}_4$, 
via the identification of a rotation with the induced permutation of its vertices. 
Easy computations in $\mathfrak{A}_4$ show that, given two distinct $3$-cycles $c_1$ and $c_2$, if the product $c_2c_1$ is again a $3$-cycle, then $c_2c_1^{2}$ has order $2$. 
Hence $\bar{A}_2\bar{A}_3=\bar{A}_1\bar{A}_3^2$ has order $2$, and its fixed points are contained in the order $6$ orbit.

\subsubsection*{\textbf{Octahedral case}}
In this case, $\Gamma_\lambda$ acts as the group of rotations of a regular octahedron. It has order $24$, and it contains $6$ elements of order $4$, namely 
the rotations around an axis passing through two vertices, $8$ elements of order $3$, namely the rotations around an axis passing through two centers of faces, 
and $9$ elements of order $2$, which are the $3$ squares of order $4$ rotations and the $6$ rotations around an axis passing through $2$ centers of edges. 
Recall that $\Gamma_\lambda$ is also isomorphic to the permutation group $\mathfrak{S}_4$, via the identification of a rotation with the induced 
permutation of the pairs of opposite faces. 

First assume $(\lambda_1,\dots,\lambda_4)=(\eta,\eta^5,\eta^7,\eta^{11})$ with $\eta$ primitive $24^{th}$ root of unity. Here $\bar{A}_3$ has order $4$, 
hence its fixed points are in the order $6$ orbit, and $\bar{A}_2$ has order $3$, hence its fixed points are in the order $8$ orbit. 
Now $\bar{A}_1$ has order $2$, and it induces a transposition in $\mathfrak{S}_4$, because the product of a $3$-cycle and a $4$-cycle in $\mathfrak{S}_4$ 
is never a double-transposition; hence its fixed points are in the order $12$ orbit. 

If $(\lambda_1,\dots,\lambda_4)=(\eta,-\eta,\eta^2,\eta^2)$ with $\eta$ primitive $12^{th} $ root of unity, $\bar{A}_3$ has order $3$ and $\bar{A}_1$ 
has order $4$. Here $\bar{A}_2$ also has order $4$, but $\bar{A}_2\bar{A}_3$ can be checked to have order $2$ 
hence, for the same reason as above, it is a transposition. 

\subsubsection*{\textbf{Icosahedral case}}
In this case, $\Gamma_\lambda$ acts as the group of rotations of a regular icosahedron. It has order $60$, and it contains $24$ elements of order $5$, namely 
the rotations around an axis passing through two vertices, $20$ elements of order $3$, namely the rotations around an axis passing through two centers of faces, 
and $15$ elements of order $2$, namely the rotations around an axis passing through $2$ centers of edges. 
Once again, we have to find elements of order $5$, $3$ and $2$ in $\Gamma_\lambda$, so that their fixed points provide elements in the orbits 
of order $12$, $20$ and $30$. 
Recall that $\Gamma_\lambda$ is also isomorphic to the direct permutation group $\mathcal{A}_5$, 
via the identification of a rotation with the induced permutation of the pairs of opposite faces. 

If $(\lambda_1,\ldots,\lambda_4)=(\alpha,\alpha^{29},\alpha^{11},\alpha^{19})$ with $\ord(\alpha)=60$, then $\bar{A}_1$ has order $3$, $\bar{A}_2$ has order $5$ and $\bar{A}_3$ has order~$2$.

If $(\lambda_1,\ldots,\lambda_4)=(\alpha,\alpha^{9},\alpha^{7},\alpha^{3})$ with $\ord(\alpha)=20$, $\bar{A}_3$ has order~$2$ and $\bar{A}_1$ has order $5$; 
$\bar{A}_2$ also has order $5$, but $\bar{A}_1^2\bar{A}_3$ has order $3$. 

Similarly, if $(\lambda_1,\ldots,\lambda_4)=(\alpha^9,\alpha^{9},\alpha,\alpha^{11})$
with $\ord(\alpha)=30$, we see that $\bar{A}_3$ has order~$5$, $\bar{A}_1$ has order $3$ 
and $\bar{A}_1^2\bar{A}_3^2$ has order~$2$. 

If $(\lambda_1,\dots,\lambda_4)=(\alpha^5,\alpha^{5},\alpha,\alpha^{19})$ or $(\beta,\beta^4,\beta^2,\beta^8)$ with $\ord(\alpha)=30$ and $\ord(\beta)=15$, $\bar{A}_3$ has order~$3$ and $\bar{A}_1$ has order $5$. 
In the first case, an element of order $2$ is $\bar{A}_2\bar{A}_3$. In the second case, such an element is given by $\bar{A}_1\bar{A}_2$.

If $(\lambda_1,\dots,\lambda_4)=(-\zeta,-\zeta,-\zeta,-\zeta^2)$ with $\ord(\zeta)=5$, $\bar{A}_3$
has order $5$,  $\bar{A}_2\bar{A}_3$ has order 3 and $\bar{A}_2\bar{A}_3^2$ has order~$2$.

\section{Coalescence and constraints on the linear part} \label{seccontraintes}
This section derives informations on finite orbits on character varieties for the $n$-punctured sphere ($n\geq 5$) from the finite orbits for less punctures.
The basic idea is to gather several punctures to consider them as a unique one. Such procedures are usually termed as \emph{coalescence} or \emph{confluence} of points.
Here, we apply this technique for $\mathrm{Aff}(\C)$-representations but it will certainly prove fruitfull for other types of representations.  
 
 \subsection{Induced representations}

In order to obtain the constraints on the linear part asserted in Theorem~\ref{thnqcq}, we will consider induced representations of our representations of $\gf$. 
To avoid confusion, denote by $\alpha_i^{(n)}$ the generators $\alpha_i$ of $\gf$ defined in Section~\ref{subsecprelim}. 
Let $n$, $k$ and $\ell$ be integers such that $4\leq k<n$ and $1\leq \ell\leq k$. Define an injective morphism from $\Lambda_k$ to $\gf$ by:
$$\begin{array}{l c c l}
   \psi_{k,\ell} : & \Lambda_k & \hookrightarrow & \gf \\
    & \alpha_i^{(k)} & \mapsto & \left\{ \begin{array}{l l} \alpha_i^{(n)} & \textrm{ if } i<\ell, \\ 
                                      \alpha_\ell^{(n)}\alpha_{\ell+1}^{(n)}\dots\alpha_{\ell+n-k}^{(n)} & \textrm{ if } i=\ell, \\
                                      \alpha_{i+n-k} & \textrm{ if } i>\ell. \end{array} \right.
  \end{array}$$
This morphism defines a surjective map $$r_{k,\ell}:\Hom(\gf,\Aff)\to\Hom(\Lambda_k,\Aff)$$ given by $r_{k,\ell}(\rho)=\rho\circ\psi_{k,\ell}$, which induces a map:
$$\begin{array}{c c l}
    \Hom(\gf,\Aff)/\Aff & \to & \Hom(\Lambda_k,\Aff)/\Aff \\
    {[\rho]} & \mapsto & [r_{k,\ell}(\rho)].
  \end{array}$$
These maps are $\mathrm{PB}_{k}\D$ equivariant in the following sense. Define a morphism from $\mathrm{PB}_{k}\D$ to $\mathrm{PB}_{n}\D$ by multiplying the $\ell$-th strand (see Figure~\ref{figphi}):
\begin{center}\scalebox{0.93}{$\begin{array}{c c c l}
   \varphi_{k,\ell} : &\mathrm{PB}_{k}\D & \to & \mathrm{PB}_{n}\D \\
    & \sigma_{i,j}^2 & \mapsto & \left\{ \begin{array}{l l} \sigma_{i,j}^2 & \textrm{if } i<j<\ell ~; \\
					(\sigma_i\dots\sigma_{\ell+n-k-1})(\sigma_{\ell+n-k-1}\dots\sigma_{\ell-1})(\sigma_i\dots\sigma_{\ell-2})^{-1} & \textrm{if } i<j=\ell~; \\
					\sigma_{i,j+n-k}^2 & \textrm{if } i<\ell<j~; \\
					(\sigma_{\ell+n-k-1}\dots\sigma_{j+n-k-1})^{-1}(\sigma_{\ell+n-k}\dots\sigma_\ell)(\sigma_\ell\dots\sigma_{j+n-k-1}) & \textrm{if } i=\ell<j~;\\
					\sigma_{i+n-k,j+n-k}^2 & \textrm{if } \ell<i<j~.
					 \end{array} \right.
  \end{array}$}
\end{center}
\begin{figure}[htb] 
\begin{center}
\begin{tikzpicture} [scale=0.35]
\begin{scope}
 \foreach \x in {2,5,7,8,11,12,13,14,17} \draw (\x,0) -- (\x,5);
 \draw (3.5,2.5) node {$\dots$};
 \draw (9.5,2.5) node {$\dots$};
 \draw (15.5,2.5) node {$\dots$};
 \draw[white,line width=5pt] (13.6,2.5) .. controls +(0,-1) and +(0,1.5) .. (6,0);
 \draw (13.6,2.5) .. controls +(0,-1) and +(0,1.5) .. (6,0);
 \draw[white,line width=5pt] (13,2.5) -- (13,0);
 \draw (13,2.5) -- (13,0);
 \draw[white,line width=5pt] (6,5) .. controls +(0,-1.5) and +(0,1) .. (13.6,2.5);
 \draw (6,5) .. controls +(0,-1.5) and +(0,1) .. (13.6,2.5);
 \draw (1,0) -- (18,0);
 \draw (1,5) -- (18,5);
 \draw (6,6) node {$i$};
 \draw (13,6) node {$\ell$};
 \draw (17,6) node {$k$};
 \draw (7.5,-2) node {$\sigma_{i,\ell}^2$};
\end{scope}
\draw[line width=3pt] (20.5,2.5) node {$\mapsto$};
\begin{scope} [xshift=22.5cm]
 \foreach \x in {2,5,7,8,11,12,13,14,17,18,21} \draw (\x,0) -- (\x,5);
 \draw (3.5,2.5) node {$\dots$};
 \draw (9.5,2.5) node {$\dots$};
 \draw (15.5,2.5) node {$\dots$};
 \draw (19.5,2.5) node {$\dots$};
 \draw[white,line width=5pt] (17.6,2.5) .. controls +(0,-1) and +(0,1.5) .. (6,0);
 \draw (17.6,2.5) .. controls +(0,-1) and +(0,1.5) .. (6,0);
 \foreach \x in {13,14,17}{
 \draw[white,line width=5pt] (\x,2.5) -- (\x,0);
 \draw (\x,2.5) -- (\x,0);}
 \draw[white,line width=5pt] (6,5) .. controls +(0,-1.5) and +(0,1) .. (17.6,2.5);
 \draw (6,5) .. controls +(0,-1.5) and +(0,1) .. (17.6,2.5);
 \draw (1,0) -- (22,0);
 \draw (1,5) -- (22,5);
 \draw (6,6) node {$i$};
 \draw (13,6) node {$\ell$};
 \draw (17,8) node {$\scriptstyle{\ell}$};
 \draw (17,7.4) node {$\scriptstyle{+}$};
 \draw (17,6.8) node {$\scriptstyle{n}$};
 \draw (17,6.2) node {$\scriptstyle{-}$};
 \draw (17,5.6) node {$\scriptstyle{k}$};
 \draw (21,6) node {$n$};
 \draw (11.5,-2) node {$(\sigma_i\dots\sigma_{\ell+n-k-1})(\sigma_{\ell+n-k-1}\dots\sigma_{\ell-1})(\sigma_i\dots\sigma_{\ell-2})^{-1}$};
\end{scope}
\end{tikzpicture}
\end{center} \caption{The morphism $\varphi_{k,\ell}$} \label{figphi}
\end{figure}
We have $\beta\cdot r_{k,\ell}(\rho)=r_{k,\ell}({\varphi_{k,\ell}(\beta)\cdot\rho})$ for all $\beta\in\mathrm{PB}_{k}\D$. 
It follows that if $[\rho]$ has a finite orbit under the action of $\mathrm{PB}_{n}\D$, then any $[r_{k,\ell}(\rho)]$ also has a finite 
orbit under the action of $\mathrm{PB}_{k}\D$. 
\begin{lemma} \label{lemmarepind}
 Let $\ell,k,n$ be integers such that $3\leq k<n$ and $1\leq\ell\leq k$. Let $\lambda\in\Hom(\gf,\C^*)$ be given by 
 $(\lambda_1,\ldots,\lambda_n)\in(\C^*)^n$. Let $\rho\in\Hom_{\lambda}(\gf,\Aff)$ be given by the translation part $(\tau_1,\ldots,\tau_n)$. 
 The representation $r_{k,\ell}(\rho)$ has a linear part given by $$(\lambda_1,\ldots,\lambda_{\ell-1},\lambda_\ell\dots\lambda_{\ell+n-k},\lambda_{\ell+n-k+1},\ldots,\lambda_n),$$ 
 and its translation part is $$(\tau_1,\ldots,\tau_{\ell-1},\sum_{i=\ell}^{\ell+n-k}\tau_i\prod_{j=\ell}^{i-1}\lambda_{j},\tau_{\ell+n-k+1},\ldots,\tau_n).$$ 
 If $[\rho]$ has a finite orbit under the action of $\mathrm{PB}_{n}\D$, then $[r_{k,\ell}(\rho)]$ has a finite orbit under the action of $\mathrm{PB}_{k}\D$, of order lower or equal.
\end{lemma}

  \subsection{Degenerate case}

We now prove Theorem~\ref{thnqcq}. We treat the case $\iota(\lambda)=3$ in the next lemma and the other cases in Lemmas~\ref{lemmaind4} and~\ref{lemmadeg}.
\begin{lemma}\label{lemindice3}
 Let $n\geq4$. Let  $(\lambda_1,\lambda_2,\ldots,\lambda_n)=(1,\ldots,1,a_1,a_2,a_3)$, with $a_i\neq1$ for all~$i$. 
 Then the action of $\PBn$ on $\PH{n}$ has exactly one finite orbit. This orbit consists in a fixed point, given by the class of the representation with  translation part 
 $(\tau_1,\ldots,\tau_n)=(0,\ldots,0,-a_2,1)$. 
\end{lemma}
\begin{proof}
 We proceed by induction on $n$. The case $n=4$ has been proved in Section~\ref{sec4}. Let $n\geq 5$.  Let $[\rho]\in\PH{n}$ with finite orbit under $\PBn$. 
 The conjugacy class of $r_{n-1,1}(\rho)$ has a finite orbit under $\PB{n-1}$, hence its translation part is either $(0,\ldots,0)$ or $(0,\ldots,0,-a_2,1)$, 
 up to conjugation. Using the formula of Lemma~\ref{lemmarepind}, we see that the translation part $(\tau_1,\ldots,\tau_n)$ of $\rho$ is either $(1,-1,0,\ldots,0)$ 
 or $(b,-b,0,\ldots,0,-a_2,1)$ with $b\in\C$, up to conjugation. Now the conjugacy class of $r_{n-1,2}(\rho)$ also has a finite orbit under $\PB{n-1}$, 
 hence its translation part is either $(0,\ldots,0)$ or $(0,\ldots,0,-a_2,1)$, up to conjugation. It follows that 
 $(\tau_1,\ldots,\tau_n)=(0,1,-1,0,\ldots,0)\textrm{ or }(0,b',-b',0,\ldots,0,-a_2,1)$ with $b'\in\C$, up to conjugation. Note that for $i\leq n-4$, 
 since $\lambda_i=1$, $\tau_i=0$ implies that $\tau_i$ is trivial for any conjugate of $\rho$. Hence the only possibility is 
 $(\tau_1,\ldots,\tau_n)=(0,\ldots,0,-a_2,1)$.
\end{proof}

\begin{lemma} \label{lemmaind4}
 Let $n\geq4$. Let $(\lambda_1,\lambda_2,...,\lambda_n)=(1,\ldots,1,a,a^{-1},a,a^{-1})$, with $a\neq1$. Then any $[\rho]\in\PH{n}$ has an infinite orbit under $\PBn$.
\end{lemma}
\begin{proof}
 We proceed by induction on $n$. The case $n=4$ was treated in Section \ref{sec4}. Assume $n>5$ and assume $[\rho]\in\PH{n}$ has a finite orbit under $\PBn$. 
 The representation $r_{n-1,1}(\rho)$ has a linear part given by $(1,\dots,1,a,a^{-1},a,a^{-1})$. Hence, by induction, 
 its translation part must be $(0,\dots,0)$ up to conjugation. It follows that $(\tau_1,\ldots,\tau_n)=(1,-1,0,\dots,0)$ 
 up to conjugation. In particular, since it is the only possibility, $[\rho]$ has to be a fixed point under $\PB{n}$. 
 But $r_{n-1,n-3}(\rho)$ has linear part given by $(1,\dots,1,a,a^{-1})$ and translation part given by $(1,-1,0,\dots,0)$, hence its conjugacy class is not 
 a fixed point under $\PB{n-1}$.  
\end{proof}

\begin{lemma} \label{lemmadeg}
 Let $k$ and $n$ be integers such that $n\geq5$ and $4\leq k<n$. Let 
 $(\lambda_1,\lambda_2,...,\lambda_n)=(1,\ldots,1,a_1,\ldots,a_k)$, with $a_i\neq1$ for all $i$. 
 Let $[\rho]\in\PH{n}$ be given by the translation part $(\tau_1,\ldots,\tau_n)$. 
 If $[\rho]$ has a finite orbit under $\PBn$, then $\tau_i=0$ for $1\leq i\leq n-k$. 
\end{lemma}
\begin{proof} 
 First assume that $a_i\dots a_{i+k-3}\neq1$ for some $i\in\{1,2,3\}$. Then the representation $r_{n-k+3,i+n-k}(\rho)$ has a linear part given 
 by $(1,\ldots,1,\hat{a}_1,\hat{a}_2,\hat{a}_3)$,  where $\hat{a}_i\neq1$ for all $i$. Thus, by Lemma \ref{lemindice3}, its translation part is given 
 by $(0,\ldots,0)$ or $(0,\ldots,0,-\hat{a}_2,1)$,  up to conjugation. It follows that $\tau_1=\dots=\tau_{n-k}=0$. 
 
 Now assume $a_1\dots a_{k-2}=a_2\dots a_{k-1}=a_3\dots a_k=1$. It gives $a_1=a_{k-1}=a$, $a_2=a_k=a^{-1}$ with $a\neq1$. 
 The case $k=4$ has been treated in Lemma \ref{lemmaind4}. The case $k=5$ does not appear, since it would imply $a_3=1$. 
 Take $k>5$. The representation $r_{n-k+5,n-k+3}$ has a linear part given by $(1,\ldots,1,a,a^{-1},1,a,a^{-1})$, hence, by Lemma~$\ref{lemmaind4}$, 
 its translation part must be $(0,\ldots,0)$. It follows that $\tau_1=\dots=\tau_{n-k}=0$.  
\end{proof}

  \subsection{Non degenerate case}

We now prove the part of Theorem~\ref{thcasfinis} which gives the constraints on the linear part when $\iota(\lambda)=n$. The case $n=4$ 
has been completely treated in Section~\ref{sec4}. The next two lemmas give the constraints for $n=5,6$, and we will see in Section~\ref{sec56} 
that the obtained linear parts indeed provide finite orbits. Lemma~\ref{lemmangeq7} shows that no additional finite orbit arises for $n\geq7$.

\begin{lemma}\label{constraints5punctures}
 Let $\lambda\in\Hom(\Lambda_5,\C^*)$ be such that $\iota(\lambda)=5$. 
 Let  $[\rho]\in\PH{5}$. 
 Assume $[\rho]$ has a finite orbit under $\PB{5}$. Then the linear part $\lambda$ is given by 
 $(\lambda_1,\ldots,\lambda_5)=(\zeta,\zeta,\zeta,\zeta,\zeta^2)$ up to permutation, where $\zeta$ is a primitive $6^{th}$ root of unity.
\end{lemma}
\begin{proof}
 Assume $\lambda_i\lambda_j=1$ for some $i\neq j$. Up to permutation, we can assume 
 $(\lambda_1,\ldots,\lambda_5)=(\lambda_1,\lambda_1^{-1},\lambda_3,\lambda_4,\lambda_5)$. 
 The induced representations $r_{4,3}(\rho)$ and $r_{4,4}(\rho)$ have linear parts $(\lambda_1,\lambda_1^{-1},\lambda_3\lambda_4,\lambda_5)$ 
 and $(\lambda_1,\lambda_1^{-1},\lambda_3,\lambda_4\lambda_5)$, thus their translation parts are both $(0,0,0,0)$, by Section \ref{sec4}. 
 But the unique common lift of these representations corresponds to $0\in \HH{5}$. Hence we have $\lambda_i\lambda_j\neq1$ for $i\neq j$. 
 
 Assume that one of the induced representations $r_{4,\ell}(\rho)$ fits in the imprimitive case, {\em i.e.} up to permutation, 
 $(\lambda_1\lambda_2,\lambda_3,\lambda_4,\lambda_5)=(a,a,-a^{-1},-a^{-1})$, with $a^2\neq1$. Then $(\lambda_1,\ldots,\lambda_5)=(\zeta,\zeta^{-1}a,a,-a^{-1},-a^{-1})$ 
 with $\zeta\neq1,a$. The representation $r_{4,3}(\rho)$ has a linear part given by $(\zeta,\zeta^{-1}a,-1,-a^{-1})$. If $\zeta=\zeta^{-1}a=-a^{-1}$ 
 and $\zeta$ is a primitive $6^{th}$ root of unity, then $(\lambda_1,\ldots,\lambda_5)=(\zeta,\zeta,\zeta^2,\zeta,\zeta)$. Apart from this case, by Section \ref{sec4},  
 the only representation whose class has a finite orbit is given by the translation part $(0,0,0,0)$. This implies $(\tau_1,\ldots,\tau_5)=(0,0,a,-1,0)$, 
 up to conjugation. Since only this conjugacy class may have a finit orbit, it has a finite orbit only if it is a fixed point. But $[r_{4,1}(\rho)]$ 
 is not a fixed point under $\PB{4}$, because there are no nontrivial fixed points in the imprimitive case. 
 
 If an induced representation of $\rho$ has a linear part which fits in the Zariski dense case, with a trivial 
 translation part, then $\rho$ is determined up to conjugation, and it has to be a fixed point. But the other induced representations will not 
 have a trivial translation part, up to conjugation, thus they are not fixed points.
 
 It remains to treat the cases where all induced representations $r_{4,\ell}(\rho)$ fit in the tetrahedral, octahedral or icosahedral case. 
 For $i\neq j$, $\lambda_i$ and $\lambda_j$ appear in the linear part of some $r_{4,\ell}(\rho)$, hence it follows from Lemma~\ref{lemmatraces} 
 and the computations of Section~\ref{subsec4linear} (see also Tables~\ref{tabn4/2} and~\ref{tabn4/3}) that the product $\lambda_i\lambda_j$ 
 is a root of unity of order $2$, $3$, $4$ or $5$. 
 Now, up to permutation, $\lambda_i\lambda_j$ also appears as a $\lambda_k'$ in the linear part $\lambda'$ of some $r_{4,\ell}(\rho)$, but no $3^{rd}$ or $5^{th}$ 
 root of unity appears in the platonic linear parts of Tables~\ref{tabn4/2} and~\ref{tabn4/3}. This implies that each product $\lambda_i\lambda_j$ 
 is either $-1$ or a primitive $4^{th}$ root of unity. If $\lambda_1\lambda_2=-1$, then $r_{4,1}(\rho)$ must fall in the second tetrahedral case 
 of the tables and $\lambda_3=\lambda_4=\lambda_5$ has order $6$ and $\ord(\lambda_3\lambda_4)=3$, which is a contradiction.
 If $\ord(\lambda_1\lambda_2)=4$, then $r_{4,1}(\rho)$ must fall in the first tetrahedral case, $\lambda_1\lambda_2=\eta^3$ and $(\lambda_3,\lambda_4,\lambda_5)$ 
 is a permutation of $(\eta,\eta^5,\eta^3)$, for some $\eta$ of order $12$. We see that there is some $\lambda_i\lambda_j$ of order $3$, contradiction.
\end{proof}

\begin{lemma}
 Let $n=6$. Let $\lambda\in\Hom(\Lambda_6,\C^*)$ be such that $\iota(\lambda)=6$. 
 Let $[\rho]\in\PH{6}$. 
 If $[\rho]$ has a finite orbit under $\PB{6}$, then the linear part $\lambda$ is given by 
 $(\lambda_1,\ldots,\lambda_6)=(\zeta,\zeta,\zeta,\zeta,\zeta,\zeta)$, where $\zeta$ is a primitive $6^{th}$ root of unity.
\end{lemma}
\begin{proof}
 Assume there are $1\leq i<j<6$ such that $\lambda_i\lambda_{i+1}\neq1$ and $\lambda_j\lambda_{j+1}\neq1$, and that the linear parts of $r_{5,i}(\rho)$ 
 and $r_{5,j}(\rho)$ are different from $(\zeta,\zeta,\zeta,\zeta,\zeta^2)$, even up to permutation and change of sixth root of unity. 
 It implies that $r_{5,i}(\rho)$ and $r_{5,j}(\rho)$ have trivial translation parts, up to conjugation. But these two representations 
 do not have a nontrivial common lift, up to conjugation. 
 
 If one of the induced representations $r_{5,\ell}(\rho)$ has a linear part given by $(\zeta,\zeta,\zeta,\zeta,\zeta^2)$ up to permutation, 
 then either $(\lambda_1,\ldots,\lambda_6)=(\zeta,\zeta,\zeta,\zeta,a,a^{-1}\zeta^2)$ with $a\neq1,\zeta^2$ or 
 $(\lambda_1,\ldots,\lambda_6)=(a,a^{-1}\zeta,\zeta,\zeta,\zeta,\zeta^2)$ with $a\neq1,\zeta$, up to permutation. 
 It follows that either $(\lambda_1,\ldots,\lambda_6)=(\zeta,\ldots,\zeta)$ or we recover the previous case.
 
 If $\lambda_i\lambda_{i+1}=1$ for all $i$, then $(\lambda_1,\ldots,\lambda_6)=(a,a^{-1},a,a^{-1},a,a^{-1})$ with $a\neq1$, and the induced representations 
 $r_{5,1}(\rho)$ and $r_{5,3}(\rho)$ must have a trivial translation part, up to conjugation. Once again, these two representations have no common lift, 
 up to conjugation. 
\end{proof}

\begin{lemma} \label{lemmangeq7}
 Let $n\geq7$. Let $\lambda\in\Hom(\gf,\C^*)$ be such that $\iota(\lambda)=n$. 
 For any $[\rho]\in\PH{n}$, $[\rho]$ has an infinite orbit under $\PBn$. 
\end{lemma}
\begin{proof}
 Assume $[\rho]\in\PH{n}$ has a finite orbit under $\PBn$. 
 Up to permutation, there are  $1\leq i<j<n$ such that $r_{n-1,i}(\rho)$ and $r_{n-1,j}(\rho)$ have a trivial translation part and these two representations 
 do not have any nontrivial common lift up to conjugation -- it is easy to check for $n=7$ and for $n\geq8$, by induction, any $(i,j)$ works. 
\end{proof}

\section{Case of five and six punctures} \label{sec56}

In this section, we treat the remaining cases, {\it i.e.} the representations with linear part $(\lambda_1,\ldots,\lambda_5)=(\zeta,\zeta,\zeta,\zeta,\zeta^2)$ 
or $(\lambda_1,\ldots,\lambda_6)=(\zeta,\ldots,\zeta)$ with $\zeta$ a primitive sixth root of unity. We see that the groups $\hat{\Gamma}_\lambda$ that appear are 
finite groups generated by complex reflections, and can also be seen as symmetry groups of regular complex polytopes. We make use of the rich theory of 
finite complex reflection groups, hence we begin with a short survey of the results we need. We work in this section with the group $\hat{\Gamma}_\lambda$ 
acting on $\C^{n-2}$, and we formally study the action on the lines of $\C^{n-2}$.

  \subsection{Finite complex reflection groups}\label{sec fcgr}
  
We compile in this part definitions and results about finite complex reflection groups. A good reference for this subject is \cite{LT}. 
\begin{definition}\ 
\begin{itemize}
 \item A {\em complex reflection} is a nontrivial linear transformation $A\in\GL_n(\C)$ that fixes a hyperplane pointwise. 
  For short, we will often say \emph{reflection} instead of complex reflection.
 \item A {\em finite complex reflection group (\CRG) of rank $n$} is a finite subgroup of $\GL_n(\C)$ generated by complex reflections.
\end{itemize} 
\end{definition}

\begin{rem} As any finite linear group, any \CRG is conjugate to a subgroup of $\U_n(\C)$. 

For a \CRG of rank $n$, if $\C^n$ is an irreducible $G$-module, {\em i.e.} if no proper subspace of $\C^n$ is preserved 
by $G$, we say that $G$ is an irreducible \CRG. Any \CRG is a direct product of irreducible \CRG.
\end{rem}
\begin{theorem}[Shephard-Todd \cite{ST}] 
 Let $G$ be an irreducible \CRG of rank $n$. Then, up to conjugacy, $G$ belongs to exactly one of the following classes.
 \begin{itemize}
  \item $n=1$ and $G$ is a cyclic group,
  \item $n\geq2$ and $G$ is the imprimitive group $G(m,p,n)$ for some $m>1$ and some divisor $p$ of $m$,
  \item $n\geq4$ and $G\simeq \mathfrak{S}_{n+1}$,
  \item $2\leq n\leq8$ and $G$ is one of the 34 exceptional groups denoted $G_k$ with $4\leq k\leq37$ in the Shephard-Todd classification.
 \end{itemize}
\end{theorem}
For the definition of the groups $G(m,p,n)$ and of the 34 exceptional groups, we refer the reader to the paper of Shephard and Todd \cite{ST} or 
to \cite{LT}.

An important notion concerning \CRG is that of the degrees and codegrees of the group.
\begin{theorem}[Shephard-Todd \mbox{[ST54, 5.1]}] \label{thST}
 If $G$ is a unitary \CRG acting on a complex vector space $V$ of dimension $n$, then the ring $J$ of $G$-invariant 
 polynomials is a polynomial algebra, {\em i.e.} it is generated by a collection of algebraically independent 
 homogeneous polynomials.
 
 Moreover, if $\{I_1,\ldots,I_r\}$ is such a set of generators, then $r=n$ and their degrees $(\deg(I_i))$ 
 are uniquely determined by $G$, up to permutation.
\end{theorem}
These degrees are called the {\em degrees of $G$}.
The definition of the codegrees of a \CRG $G$ involves more background, and we refer the reader to \cite[Chap.10]{LT}.
\begin{lemma}[Spr74, \S2] \label{lemmacardG}
 The cardinality of a \CRG is equal to the product of its degrees.
\end{lemma}

The non-generic orbits of the action of a \CRG on $\C^n$ are easily determined thanks to the following result \cite[Th.1.5]{Stein}.
\begin{theorem}[Steinberg] \label{thSt}
 Let $G$ be a \CRG of rank $n$. If $x\in\C^n$, the stabilizer $G_x=\{g\in G\,|\,gx=x\}$ is the \CRG generated by the reflections 
 which fix $x$.
\end{theorem}
It follows that the points of $\C^n$ that are fixed by some element of the group $G$ are the points lying on the union $\Hyp$ 
of the reflection hyperplanes associated with the reflections of~$G$. However, we are interested in the action on the lines 
of $\C^n$, hence we have to consider not only the points fixed by some element of $G$, but also the eigenvectors of elements of $G$. 
On $X=\C^n\setminus\Hyp$, we shall use the eigenspaces theory originally developed by Springer.

\begin{definition}
 Let $G$ be a \CRG of rank $n$.
\begin{itemize}
 \item A vector $v\in\C^n$ is {\em regular} if $v\in X$.
 \item An element $g\in G$ is {\em regular} if it admits a regular eigenvector.
 \item An integer $d$ is {\em regular} if some $g\in G$ admits a regular eigenvector associated with an eigenvalue  of order $d$.
 \item An eigenspace $V(g,\zeta)$ is {\em regular} if it contains a regular vector.
\end{itemize} 
\end{definition}

\begin{theorem}[Lehrer-Springer \cite{LS} Th.C] \label{thLS}
 Let $G$ be a \CRG of rank $n$. Let $d_1,\ldots,d_n$ be the degrees of $G$ and let $d_1^*,\ldots,d_n^*$ be the codegrees of $G$. 
 For a positive integer $d$, set $a(d)=\card\{i;\,d\textrm{ divides }d_i\}$ and $b(d)=\card\{i;\,d\textrm{ divides }d_i^*\}$. 
 Then $a(d)\leq b(d)$, and $d$ is regular if and only if $a(d)=b(d)$.
\end{theorem}
\begin{theorem}[Springer \cite{Sp} Th. $3.4$ $\&$ Th. $4.2$] \label{thSp}
 Let $G$ be a \CRG of rank $n$. Let $d$ be a regular integer. Define $a(d)$ as in the previous theorem. 
 Let $\zeta$ be a primitive $d$-th root of unity. The eigenspaces $V(g,\zeta)$ with $g\in G$ which are regular form a single conjugacy class under $G$, 
 and their dimension is $a(d)$.
\end{theorem}

\begin{theorem}[Springer \cite{Sp}, Th. $4.2$, \cite{LT} Th. $11.24$] \label{thSp2} 
 Let $G$ be a \CRG of rank $n$. Let $d$ be a regular integer. Let $\zeta$ be a primitive $d$-th root of unity. 
 Take $g\in G$ such that the eigenspace $E=V(g,\zeta)$ is regular. Then the stabilizer $S_G(E)=\{g\in G\,|\,gE\subset E\}$ 
 acts on $E$ as a \CRG whose degrees are those degrees of $G$ which are divisible by $d$.
\end{theorem}

\begin{lemma} \label{lemmaeigen}
 Let $G$ be a \CRG of rank $n$. Let $d$ and $d'$ be regular integers. Let $\zeta$ (resp. $\zeta'$) be a primitive $d$-th (resp. $d'$-th) root of unity. 
 Define $a(d)$ as in Theorem \ref{thLS}. 
 If $d'|d$ and $a(d')=a(d)$ -- in particular if $d'=d$ -- the set of regular eigenspaces $V(g,\zeta)$ coincides 
 with the set of regular eigenspaces $V(g,\zeta')$.
\end{lemma}
\begin{proof}
 It follows from the inclusion $V(g,\zeta)\subset V(g^k,\zeta^k)$ and Theorem \ref{thSp}.
\end{proof}

\begin{lemma} \label{lemmastab}
 Let $G$ be a \CRG of rank $n$. Let $x\in\C^n$ be a regular vector. The stabilizer $S_G(\C x)$ of $\C x$ in $G$ 
 is a cyclic group.
\end{lemma}
\begin{proof}
 The stabilizer $S_G(\C x)$ acts on $\C x$ as a finite group $S$ of unitary transformations of $\C$, {\em i.e.} as a cyclic group. 
 Hence there is $g\in S_G(\C x)$ such that $g(x)=\zeta x$ with $\zeta$ of order $|S|$. For $h\in S$, we have $h(x)=\zeta^k x=g^k(x)$, 
 thus, by Theorem \ref{thSt}, $h=g^k$.
\end{proof}

In this paper, we are particularly interested in two \CRG, the exceptional groups denoted $G_{25}$ and $G_{32}$ 
in the Shephard-Todd classification. These two groups can be interpreted as the symmetry groups of complex regular polytopes. 
We introduce this notion in the next section.

  \subsection{Regular complex polytopes} \label{subsecpolytopes}
  
For a detailed exposition of the theory of regular complex polytopes, we refer the reader to Coxeter \cite{Cox1}.

Let $V$ be a unitary $\C$-vector space of dimension $n$. An {\em $m$-flat in $V$} is an affine subspace of $V$ of dimension~$m$. 
We allow the value $m=-1$ for the empty set. An $m$-flat and an $m'$-flat are {\em incident} if one is a proper subset of the other. 
Let $F$ be a set of distinct $m$-flats containing the empty set and the whole space $V$. A subset of $F$ is {\em connected} if any two of its flats 
can be joined by a chain of successively incident flats of $F$. If $m'-m\geq2$, any two incident $m$-flat and $m'$-flat in $F$ 
determine a {\em medial figure} consisting of all $r$-flats in $F$, with $m<r<m'$, that are incident to both. 
\begin{definition}
 The set $F$ is called a {\em complex polytope} if, for any incident $m$-flat and $m'$-flat in $F$ with $m'-m\geq2$, 
 \begin{itemize}
  \item the associated medial figure contains at least two $r$-flats for each $r$ satisfying $m<r<m'$,
  \item if $m'-m\geq3$, the medial figure is connected.
 \end{itemize}
\end{definition}

The {\em symmetry group} of a complex polytope $F$ is the subgroup of $\U(V)$ formed by the transformations which permute the flats of $F$. 
A {\em flag} of $F$ is a subset consisting of one flat of each dimension, pairwise incident. A complex polytope $F$ is {\em regular} if its symmetry 
group is transitive on the flags of $F$. The symmetry group of a regular complex polytope is a \CRG.

We now describe briefly the two regular complex polytopes whose symmetry groups will appear in the sequel. 

The {\em Hessian polyhedron} is a regular complex polytope in $\C^3$ of generalized Schl\"afli symbol $3\{3\}3\{3\}3$ (the polytope $p_1\{q_1\}\dots p_n$ 
has a symmetry group generated by reflections of order $p_k$ such that two non-consecutive ones commute and two consecutive ones satisfy a relation which is 
the braid relation when $q_k=3$). Its 27 vertices (0-flats) are given by the coordinates $(0,\omega^j,-\omega^k)$ up to cyclic permutation, where 
$\omega=e^{\frac{2i\pi}{3}}$ and $j,k=0,1,2$. Its symmetry group is the group $G_{25}$. 
By \cite[10.3]{ST}, the group $G_{25}$ is generated by the following three reflections of order~3:
 $$R_1=\begin{pmatrix}1&&\\&1&\\&&\omega^2\end{pmatrix} \qquad R_2=\frac{\omega^2-\omega}{3}\begin{pmatrix}\omega&\omega^2&\omega^2\\ \omega^2&\omega&\omega^2\\ 
 \omega^2&\omega^2&\omega\end{pmatrix}
 \qquad R_3=\begin{pmatrix}1&&\\&\omega^2&\\&&1\end{pmatrix}$$
where $\omega=e^{\frac{2i\pi}{3}}$. It is the quotient of the braid group $\mathrm{B}_4\D$ by the relations $\sigma_i^3=1$ (see Coxeter \cite{Cox}). 
The group $G_{25}$ contains $24$ reflections, all of order $3$, and the $12$ associated reflection planes are given by the following equations:
$$x=0,\ y=0,\ z=0,\ x+\xi y+\xi' z=0\ \mathrm{with}\ \xi\ \mathrm{and}\ \xi'\ \mathrm{any\ third\ roots\ of\ unity}.$$

\begin{lemma} \label{lemmatrans25}
 The group $G_{25}$ acts transitively on the set of its reflection planes.
\end{lemma}
\begin{proof}
 It suffices to check that the orbit of the vector $(0,0,1)$ normal to the plane $\{z=0\}$ contains the normal directions to all other reflection planes. 
 The reflection $R_2$ sends $(0,0,1)$ on a scalar multiple of $(1,1,\omega^2)$. Then, by applications of $R_1$ and $R_3$, we get all the $(1,\xi,\xi')$. 
 Finally, $R_2$ sends $(1,\omega,1)$ on $(0,1,0)$ and $(1,\omega^2,\omega^2)$ on $(1,0,0)$, up to scalar multiplication.
\end{proof}

The {\em Witting polytope} is a regular complex polytope in $\C^4$ of generalized Schl\"afli symbol $3\{3\}3\{3\}3\{3\}3$. 
Its 240 vertices are given by the coordinates 
$$\begin{array}{c c c c}
   \pm(0,\omega^j,-\omega^k,\omega^\ell)&\pm(-\omega^j,0,\omega^k,\omega^\ell)&\pm(\omega^j,-\omega^k,0,\omega^\ell)&\pm(\omega^j,\omega^k,\omega^\ell,0) \\
   (\pm i\omega^j,0,0,0)&(0,\pm i\omega^j,0,0)&(0,0,\pm i\omega^j,0)&(0,0,0,\pm i\omega^j)
  \end{array}$$
where $\omega=e^{\frac{2i\pi}{3}}$ and $j,k,\ell=0,1,2$. Its 3-flats are Hessian polyhedra; one of which has the 27 vertices 
$$(0,\omega^j,-\omega^k,1)\qquad (-\omega^k,0,\omega^j,1)\qquad (\omega^j,-\omega^k,0,1)$$
with $j,k,\ell=0,1,2$ and lies in the affine hyperplane $\{z_4=1\}$. 
The symmetry group of the Witting polytope is the group $G_{32}$. By \cite[10.5]{ST}, the group $G_{32}$ is generated by the following four reflections of order~3:
 $$R_1=\scalebox{0.9}{$ \begin{pmatrix}1&&&\\&1&&\\&&\omega^2&\\&&&1\end{pmatrix} \qquad $}
 R_2=\scalebox{0.9}{$ \frac{\omega^2-\omega}{3}\begin{pmatrix}\omega&\omega^2&\omega^2&0\\ \omega^2&\omega&\omega^2&0\\ 
 \omega^2&\omega^2&\omega&0\\0&0&0&\omega-\omega^2\end{pmatrix}$}$$
 $$R_3=\scalebox{0.9}{$ \begin{pmatrix}1&&&\\&\omega^2&&\\&&1&\\&&&1\end{pmatrix}$} \qquad 
 R_4=\scalebox{0.9}{$ \frac{\omega^2-\omega}{3}\begin{pmatrix}\omega&-\omega^2&0&-\omega^2\\ -\omega^2&\omega&0&\omega^2\\0&0&\omega-\omega^2&0\\ 
 -\omega^2&\omega^2&0&\omega\end{pmatrix}$}$$
where $\omega=e^{\frac{2i\pi}{3}}$. It is the quotient of the braid group $\mathrm{B}_5\D$ by the relations $\sigma_i^3=1$ (see \cite{Cox}). 
The group $G_{32}$ contains 80 reflections, all of order 3, and the 40 associated reflection hyperplanes are given by the following equations:
$$\left\lbrace\begin{array}{l l}
z_i=0, & \mathrm{with}\ 1\leq i\leq4,\\ 
\hspace{-1ex}\begin{array}{l} z_1+\xi z_2+\xi' z_3=0,\\ z_1-\xi z_2-\xi' z_4=0,\\ z_1-\xi z_3+\xi' z_4=0,\\ z_2-\xi z_3-\xi' z_4=0,\end{array} & 
\mathrm{with}\ \xi\ \mathrm{and}\ \xi'\ \mathrm{any\ third\ roots\ of\ unity}. \end{array}\right.$$

\begin{lemma} \label{lemmatrans32}
 The group $G_{32}$ acts transitively on the set of its reflection hyperplanes.
\end{lemma}
\begin{proof}
 Check that the orbit of the vector $(0,0,1,0)$ normal to the plane $\{z_3=0\}$ contains the other normal directions to all other reflection hyperplanes. 
\end{proof}

\subsection{Five punctures: the Hessian group} \label{sec Hessian}

Fix the linear part $\lambda$ given by $(\lambda_1,\ldots,\lambda_5)=(\zeta,\zeta,\zeta,\zeta,\zeta^2)$ with $\zeta$ a primitive sixth root of unity. 
The group $\hat{\Gamma}_\lambda$ is generated by the images of the braids $\sigma_{i,j}^2$ with $1\leq i<j\leq4$ that are the generators of $\PBDnum{4}$. 
Since the $\lambda_i$ are equal for $i\leq4$, 
the braids $\sigma_i$ for $1\leq i\leq3$ preserve the linear part. Let $A_i$, $1\leq i\leq3$ be the matrices of their action. We have:
$$A_1=\begin{pmatrix}-\zeta&0&0\\-\zeta&1&0\\-\zeta&0&1\end{pmatrix} \qquad A_2=\begin{pmatrix}-\zeta^2&\zeta&0\\1&0&0\\0&0&1\end{pmatrix} 
 \qquad A_3=\begin{pmatrix}1&0&0\\0&-\zeta^2&\zeta\\0&1&0\end{pmatrix}$$

\begin{lemma}\label{lemgen5trous}
We have the following generating sets for $\hat{\Gamma}_\lambda\subset\GL_3(\C)$.
\begin{enumerate}
\item $\hat{\Gamma}_\lambda=\langle  A_1,A_2,A_3\rangle $,
 \item $\hat{\Gamma}_\lambda=\langle  M_{12},M_{13},M_{14}\rangle $, where $M_{ij}$ is the matrix of the transformation induced by the pure braid $\sigma_{i,j}^2$.
 \end{enumerate}
\end{lemma}
\begin{proof}
 The inclusions $\langle M_{12},M_{13},M_{14}\rangle \subset \hat{\Gamma}_\lambda\subset\langle  A_1,A_2,A_3\rangle $ are obvious. 
 Now note that all $A_i$ have order 3. Recalling the definition of the $\sigma_{i,j}$, it is easy to see that 
 $\langle M_{12},M_{13},M_{14}\rangle=\langle  A_1,A_2,A_3\rangle$. 
\end{proof}
\begin{proposition}
 The group $\hat{\Gamma}_\lambda$ is conjugate to the group $G_{25}$.
\end{proposition}
It follows that the group $\Gamma_\lambda\subset\PGL_3(\C)$ is conjugate to the so called Hessian group.
\begin{proof}
 Set:
 $$P=\begin{pmatrix}0&0&1-\zeta^2\\ \zeta^2&\zeta^2&1\\ \zeta^2&-\zeta&1\end{pmatrix}.$$
 We have the following relations:
 $$\mathrm{if}\ \zeta=-\omega,\ \left\lbrace \begin{array}{l}R_1=P^{-1}A_1^2P\\R_2=P^{-1}A_2P\\R_3=P^{-1}A_3^2P\end{array}\right.\quad
 \mathrm{and\ if}\ \zeta=-\omega^2,\ \left\lbrace \begin{array}{l}R_1=P^{-1}A_1P\\R_2=P^{-1}A_2^2P\\R_3=P^{-1}A_3P\end{array}\right.,$$
 where the $R_i$ are the generators of $G_{25}$ given above.
\end{proof}

The remainder of this subsection is devoted to the proof of the following proposition, which describes the non-generic orbits of 
lines in $\C^3$ under the action of $G_{25}$. We see the size of the orbit of a line is often characterized by the number of reflection planes 
and proper planes on which it lies. By {\em proper plane}, we mean a regular eigenspace associated with an eigenvalue of order $6$. 

\begin{proposition} \label{propn5}
 The orbits of complex lines of $\C^3$ under the action of $G_{25}$ are given in Table \ref{tabn5} with, for each orbit, its size, the number 
 of reflection planes and proper planes the lines of the orbit lie on, and the number of such orbits.
\begin{table}
 \begin{center}
 \scalebox{\scalefactor}{
  \begin{tabular}{|c|c|c|c|}
   \hline Order & Number & Reflection planes & Proper planes \\
   \hline 216 & Generic & \multirow{4}{*}{0} & \multirow{2}{*}{0} \\
   \cline{1-2} 72 & 1 & & \\
   \cline{1-2} \cline{4-4} 108 & Infinity of dim. 1 & & \multirow{2}{*}{1} \\
   \cline{1-2} 54 & 1 & & \\
   \hline 72 & Infinity of dim. 1 & \multirow{2}{*}{1} & 0 \\
   \cline{1-2} \cline{4-4} 36 & 1 & & 1 \\
   \hline 12 & 1 & 2 & 3 \\
   \hline 9 & 1 & 4 & 0 \\
   \hline 
  \end{tabular}} \vspace{1ex} \caption{Orbits in $\PC{2}$ under action of the Hessian group} \label{tabn5}
 \end{center}
\end{table}
 Moreover,
  \begin{itemize}
  \item the line $[\nu:\nu^2:1]$ with $\ord(\nu)=9$ represents the only orbit of order $72$ that intersects no reflection plane,
  \item the line $[0:\omega:1]$ represents the only orbit of order $54$.
 \end{itemize}
Given a line in $\Pd\setminus (G_{25}\cdot [\nu:\nu^2:1]\cup G_{25}\cdot [0:\omega:1])$, the size of its orbit is characterized by the number 
of reflection planes and proper planes it lies on.
\end{proposition}

We study the action of the group $G_{25}$ on $\C^3$ with the realization of Shephard and Todd, which has the advantage of being unitary. 
This group has order 648, and its center is $\Z/3\Z$. Its degrees are 6, 9 and 12, and its codegrees are 0, 3, and 6 \cite[Table D.3]{LT}. 
Hence, by Theorem \ref{thLS}, the regular numbers of $G_{25}$ are 1, 2, 3, 4, 6, 9 and 12. Denote by $\Hyp$ the union of its reflection planes, 
and set $X=\C^3\setminus\Hyp$. We first study the special orbits of complex lines in $X$. 
The generic orbits have order $|G_{25}|/|Z(G_{25})|=216$. 

If $x\in X$ is an eigenvector of some $g\in G_{25}$, then the associated eigenvalue is a root of unity $\zeta_d$ whose order $d$ is one of the above 
regular numbers, and the dimension of the associated eigenspace $V(g,\zeta_d)$ is the number of degrees of $G_{25}$ divisible by $d$. The case 
\mbox{$d=1,3$} is not relevant since the corresponding $g$ is a homothetie in $Z(G_{25})$. By Lemma~\ref{lemmaeigen}, the regular numbers 4 and 12 provide one orbit 
of complex lines. By Theorem~\ref{thSp2}, the stabilizer of a line in this orbit is a \CRG of rank 1 and degree 12, {\em i.e.} the cyclic group 
of order 12. Hence this orbit has order $648/12=54$. Similarly, the orbit of the regular eigenspaces associated with a root of unity of order 9 
has order $648/9=72$. 

The regular numbers 2 and 6 provide the same set of eigenspaces of dimension 2. Let $E=V(g,\zeta)$ be one of those, 
with $\ord(\zeta)=6$. Let $x$ be a regular vector in $E$. By Lemma~\ref{lemmastab}, the stabilizer of $\C x$ in $G_{25}$ is a cyclic group; 
denote by $s$ its order. Then $s$ is a regular number of $G_{25}$ and it is divisible by 6, hence $s=6\textrm{ or }12$. The case $s=12$ corresponds 
to the orbit of size 54 we have already identified. Thus for the other regular vectors in $E$, we have $s=6$. The lines in $E$ which are neither on 
the order 54 orbit nor on $\Hyp$ have an orbit of order $648/6=108$. 

The given representative of the order $72$ orbit in $X$ is the eigenvector of the matrix 
$\begin{pmatrix} 0 & 1 & 0 \\ 0 & 0 & \nu^3 \\ 1 & 0 & 0 \end{pmatrix}$ associated with the eigenvalue $\nu$ of order $9$, which is an element of $G_{25}$ 
since it permutes the vertices of the Hessian polyhedron.

The given representative of the order 54 orbit is an eigenvector of the transformation $R_1R_2^2R_3=\frac{\omega-\omega^2}{3}\begin{pmatrix}
                                                                                                         \omega^2 & 1 & \omega \\
                                                                                                         \omega & \omega & \omega \\
                                                                                                         1 & \omega^2 & \omega
                                                                                                        \end{pmatrix}$
associated with the eigenvalue $\frac{\sqrt{3}}{2}-\frac{1}{2}i$ of order 12.

It remains to treat the orbits of complex lines in $\Hyp$. Since the action of $G_{25}$ on the reflection planes is transitive, we focus 
on one of those. Let $\plan$ be the reflection plane of equation $z=0$. We first compute its stabilizer.
\begin{lemma} \label{lemmastab25}
 The stabilizer $S_{G_{25}}(\plan)$ of the plane $\plan$ in $G_{25}$ is the following subgroup of $G_{25}$:
 $$S_{G_{25}}(\plan)=\left\lbrace \left.\begin{pmatrix} \xi_1 & & \\ & \xi_2 & \\ & & \xi_3 \end{pmatrix},
   \begin{pmatrix} 0 & -\xi_1 & 0 \\ -\xi_2 & 0 & 0 \\ 0 & 0 & -\xi_3 \end{pmatrix}  \right|\, \xi_1, \xi_2, \xi_3 \textrm{ third roots of unity}\right\rbrace$$
\end{lemma}
\begin{proof}
 The unitary transformations which stabilize $\plan$ also stabilize its normal vector, hence they have the form:
 $$A=\begin{pmatrix}
    B & 0 \\ 0 & e
   \end{pmatrix}, \mbox{ with }B=\begin{pmatrix}a&b\\c&d\end{pmatrix}\in\U_2(\C),\mbox{ and } |e|=1.$$
The elements of $S_{G_{25}}(\plan)$ are the transformations 
of this form which stabilize the set of vertices of the Hessian polyhedron. 
The matrix $A$ sends the vertex $(0,1,-1)$ on $(b,d,-e)$, hence $b=0$ or $d=0$. If $b=0$, then $c=0$ and $e$ has order $1$ or $3$.
If $d=0$, then $a=0$ and $e$ has order 2 or 6. We easily conclude by looking at the action of $A$ on the vertex $(1,-1,0)$.
\end{proof} 
The pointwise stabilizer $F_{G_{25}}(\plan)$ of $\plan$ is the subgroup of $S_{G_{25}}(\plan)$ given by the diagonal matrices with diagonal $(1,1,\xi_3)$. 
Hence the quotient $S_{G_{25}}(\plan)/F_{G_{25}}(\plan)$ is the \CRG $G(3,1,2)$ defined as the semi-direct 
product of $(\Z/3\Z)^2$ by $\Z/2\Z$ acting by permutation of the coordinates. The center of $G(3,1,2)$ is $\Z/3\Z$, coinciding with the center of $G_{25}$, 
and the quotient of $G(3,1,2)$ by its center is the dihedral group of order 6. The orbits of the complex lines in $\plan$ under 
the action of its stabilizer $S_{G_{25}}(\plan)$ are the following.
\begin{itemize}
 \item An orbit of order 2 given by the reflection lines of reflections of order 3, intersections with $\plan$ of the two reflection 
 planes of $G_{25}$ normal to $\plan$, namely $\{x=0\}$ and $\{y=0\}$.
 \item An orbit of order 3 given by the reflection lines of reflections of order 2, each being the intersection with $\plan$ of three reflection 
 planes of $G_{25}$. These lines are given in $\plan$ by the equations $\{x+\omega^jy=0\}$ for $j=0,1,2$.
 \item Another orbit of order 3 given by the lines whose equations in $\plan$ are $\{x-\omega^jy=0\}$ for $j=0,1,2$.
 \item The generic orbits of order 6.
\end{itemize}
Since the group $G_{25}$ acts transitively on the set of reflection planes, for the lines of $\plan$ that lie on exaclty one reflection plane, 
the orbit under action of $G_{25}$ has order $12$ times the order of its orbit under the stabilizer; this provides an infinity of dimension $1$ of orbits of order $72$, 
and an orbit of order 36. The orbit of order $2$ under the action of $G(3,1,2)$ gives an orbit under the action of $G_{25}$ composed of all the lines in $\C^3$ 
which are intersection of exactly $2$ reflection planes of $G_{25}$. Hence a simple enumeration shows this orbit has order~$12$. Similarly, the remaining orbit of order $3$
under the action of $G(3,1,2)$ corresponds to the intersections of exactly 4 reflection planes of $G_{25}$ and it provides an orbit of order $9$ under the action 
of $G_{25}$.

The following result gives the equations of the proper planes.
\begin{lemma} \label{lemmaeigenplanes}
 The regular eigenspaces of dimension 2 of elements of $G_{25}$ associated with eigenvalues of order 6 are the following planes. 
 $$\{x-\omega^kz=0\}\qquad\{y-\omega^kz=0\}\qquad\{x-\omega^ky=0\}\qquad\textrm{with }k=0,1,2$$
\end{lemma}
\begin{proof}
 We have seen the matrix $\begin{pmatrix} 0 & -1 & 0 \\ -\xi & 0 & 0 \\ 0 & 0 & -\xi^2 \end{pmatrix}$ 
 with $\xi$ of order 3 defines an element $g\in G_{25}$. The eigenspace $E=V(g,-\xi^2)$ is the plane $\{x-\xi y=0\}$. 
 By Theorem \ref{thSp2}, the stabilizer of $E$ is a \CRG of degrees 6 and 12, thus of order 72. Since the group $G_{25}$ acts transitively 
 on the regular eigenspaces $V(h,-\xi^2)$ with $h\in G_{25}$, it follows that there are 9 such eigenspaces. It is easily checked, by making $R_1$, $R_2$ and $R_3$ 
 act on a vector normal to $\{x-\xi y=0\}$, that these 9 planes are the given ones.
\end{proof}

From the equations of the reflection planes and proper planes, and from the above study of the action of $G(3,1,2)$ on $\plan$, 
we see the proper planes only meet 3 by 3 along the lines that are intersection of two orthogonal reflection planes ({\em i.e.} with orthogonal normal vectors), 
and the other intersections of a proper plane with a reflection plane are the lines of the order $36$ orbit.

  \subsection{Six punctures: the simple group of order 25920} \label{sec 25920}

Fix the linear part $\lambda$ given by $(\zeta,\zeta,\zeta,\zeta,\zeta,\zeta)$ where $\zeta$ is a primitive sixth root of unity. 
Since all the $\lambda_i$ are equal, the action of $\BDnum{6}$ preserves the linear part, hence it is well-defined on $\HH{6}$. 
The following matrices $A_i$ give the action of the standard full braid group generators $\sigma_i, i=1,\ldots, 4$. 

{\small
$$A_1=\begin{pmatrix}
-\zeta & 0 & 0 & 0 \\
-\zeta & 1 & 0 & 0 \\
-\zeta & 0 & 1 & 0 \\
-\zeta & 0 & 0 & 1
\end{pmatrix}\qquad
A_2=\begin{pmatrix}
-\zeta^2 & \zeta & 0 & 0 \\
1 & 0 & 0 & 0 \\
0 & 0 & 1 & 0 \\
0 & 0 & 0 & 1
\end{pmatrix}$$
$$A_3=\begin{pmatrix}
1 & 0 & 0 & 0 \\
0 & -\zeta^2 & \zeta & 0 \\
0 & 1 & 0 & 0 \\
0 & 0 & 0 & 1
\end{pmatrix}\qquad
A_4=\begin{pmatrix}
1 & 0 & 0 & 0 \\
0 & 1 & 0 & 0 \\
0 & 0 & -\zeta^2 & \zeta \\
0 & 0 & 1 & 0
\end{pmatrix}$$
}
The order of these four linear transformations is $3$. 
Similarly to Lemma $\ref{lemgen5trous}$, we obtain the following.
\begin{lemma}\label{lemgen6trous}
 We have the following generating sets for $\hat{\Gamma}_\lambda\subset\GL_4(\C)$.
\begin{enumerate}
 \item $\hat{\Gamma}_\lambda=\langle  A_1,A_2,A_3,A_4\rangle $,
 \item $\hat{\Gamma}_\lambda=\langle  M_{12},M_{13},M_{14},M_{15}\rangle$, where $M_{ij}$ is the matrix of the transformation induced by the pure braid $\sigma_{i,j}^2$. 
\end{enumerate} 
\end{lemma}

\begin{proposition}
 The group $\hat{\Gamma}_\lambda$ is conjugate to the group $G_{32}$.
\end{proposition}
\begin{proof}
 Set:
 $$P=\scalebox{0.9}{$\begin{pmatrix}0&0&1-\zeta^2&0\\ \zeta^2&\zeta^2&1&0\\ \zeta^2&-\zeta&1&0\\-1&0&1&\zeta\end{pmatrix}$}.$$
 We have the following relations:
 $$\mathrm{if}\ \zeta=-\omega,\ \left\lbrace \begin{array}{l}R_1=P^{-1}A_1^2P\\R_2=P^{-1}A_2P\\R_3=P^{-1}A_3^2P\\R_4=P^{-1}A_4P\end{array}\right.\quad
 \mathrm{and\ if}\ \zeta=-\omega^2,\ \left\lbrace \begin{array}{l}R_1=P^{-1}A_1P\\R_2=P^{-1}A_2^2P\\R_3=P^{-1}A_3P\\R_4=P^{-1}A_4^2P\end{array}\right.,$$
 where the $R_i$ are the generators of $G_{32}$ given above.
\end{proof}

In the remainder of this subsection, we prove the following proposition. By {\em proper plane}, we mean a dimension $2$ eigenspace 
associated with an eigenvalue of order $12$.
\begin{proposition}\label{propn6}
 The orbits of complex lines of $\C^4$ under the action of $G_{32}$ are given in Table \ref{tabn6} with, for each orbit, its size, the number of
 reflection hyperplanes and proper planes the lines of the orbit lie on, and the number of such orbits. 
\begin{table}
 \begin{center}
 \scalebox{\scalefactor}{
  \begin{tabular}{|c|c|c|c|}
   \hline Order & Number & Reflection hyperplanes & Proper planes \\
   \hline 25920 & Generic & \multirow{4}{*}{0} & \multirow{2}{*}{0} \\
   \cline{1-2} 5184 & 1 & & \\
   \cline{1-2} \cline{4-4} 12960 & Infinity of dim. 1 &  & \multirow{2}{*}{1} \\
   \cline{1-2} 6480 & 1 & & \\
   \hline 8640 & Infinity of dim. 2 & \multirow{2}{*}{1} & \multirow{2}{*}{0} \\
   \cline{1-2} 2880 & 1 & & \\
   \hline 2880 & Infinity of dim. 1 & \multirow{2}{*}{2} & 0 \\
   \cline{1-2} \cline{4-4} 1440 & 1 & & 3 \\
   \hline 1080 & Infinity of dim. 1 & \multirow{2}{*}{4} & 0 \\
   \cline{1-2} \cline{4-4} 540 & 1 & & 6 \\
   \hline 360 & 1 & 5 & 0 \\
   \hline 40 & 1 & 12 & 0 \\
   \hline
  \end{tabular}} \vspace{1ex}  \caption{Orbits in $\PC{3}$ under action of the simple group of order 25920} \label{tabn6}
 \end{center}
\end{table}
Moreover,
 \begin{itemize}
  \item the line $[\nu^7+\nu-\nu^5:\nu^7-\nu^6:-\nu^8+\nu^7+\nu^6-\nu^4-\nu^2+1:1]$ with $\ord(\nu)=30$ represents the only orbit of order 5184,
  \item the line $[-\nu^{10}+\nu^3-1:\nu^{11}-\nu^9+\nu^6-\nu^4+\nu:\nu^{11}-\nu^{10}:1]$ with $\ord(\nu)=24$ represents the only orbit of order 6480,
  \item the line $[\nu:\nu^2:1:0]$ with $\ord(\nu)=9$ represents the only orbit of order 2880 whose lines lie on one reflection hyperplane.
 \end{itemize}
Given a line in the complement of the above three orbits, the size of its orbit is characterized by the number of reflection hyperplanes and proper planes 
it lies on.
\end{proposition}

Once again, we use the realization of the group $G_{32}$ given by Shephard and Todd.
This group has order $155520$, and its center is $\Z/6\Z$. Its degrees are 12, 18, 24 and 30, its codegrees are 0, 6, 12 and 18 \cite[Table D.3]{LT}, 
and thus its regular numbers are $1$, $2$, $3$, $4$, $5$, $6$, $8$, $10$, $12$, $15$, $24$ and $30$. As in the case of $G_{25}$, we use the notation
$\Hyp$ for the union of the reflection hyperplanes and $X=\C^4\setminus\Hyp$. 

The study of the orbits of lines in $X$ is completely similar to the case of $G_{25}$. The generic orbits have order $|G_{32}|/|Z(G_{32})|=25920$. 
The regular numbers 1, 2, 3 and 6 correspond to homotheties. There is one special orbit of order $5184$ corresponding to the regular numbers 
5, 10, 15 and 30, and another special orbit of order $6480$ corresponding to the regular numbers 8 and 24. The regular numbers 4 and 12, 
dividing two degrees, provide an infinity of  dimension $1$ of orbits of order $12960$.

We now treat the orbits of complex lines in $\Hyp$. Since the action of $G_{32}$ on the reflection hyperplanes is transitive, we focus 
on the hyperplane $\HR=\{z_4=0\}$. 
\begin{lemma}
 Let $v$ be the vector $(0,0,0,1)$ normal to $\HR$. Let $S_{G_{32}}(\HR)$ be the setwise stabilizer of $\HR$. Let $F_{G_{32}}(\HR)$ and $F_{G_{32}}(v)$ 
 be the pointwise stabilizers of $\HR$ and $v$ respectively. 
 \begin{itemize}
  \item We have $S_{G_{32}}(\HR)=F_{G_{32}}(v)\times F_{G_{32}}(\HR)\times \{\pm id\}$.
  \item The pointwise stabilizer $H:=F_{G_{32}}(v)$ is isomorphic to the group $G_{25}$.
 \end{itemize}
\end{lemma}
\begin{proof}
 Since the group $G_{32}$ is unitary, the product $F_{G_{32}}(v)\times F_{G_{32}}(\HR)\times \{\pm id\}$ is indeed direct. 
 Let $\varphi\in S_{G_{32}}(\HR)$. Since $v$ lies on the 3-flat $\{z_4=1\}$ of the Witting polytope, parallel to $\HR$, $\varphi(v)$ 
 lies on one of the six 3-flats of the Witting polytope parallel to $\HR$, namely the hyperplanes $\{z_4=\pm\omega^j\}$ with $j=0,1,2.$ 
 Hence one can obtain an element of $F_{G_{32}}(v)$ by composing if necessary $\varphi$ with $-id$ and with a reflexion of hyperplane $\HR$. 
 This proves the first point. 

 An element of $F_{G_{32}}(v)$ stabilizes the 3-flat $\{z_4=1\}$. 
 It induces a linear transformation of $\C^3\cong\{z_4=1\}$ which preserves the face $f$ defined as the polyhedron composed of the flats of the Witting polytope 
 contained in $\{z_4=1\}$. Reciprocally, such a symmetry of $f$ extends to a linear transformation  $g \in \mathrm{GL}_4(\C)$ and fixes $v$, the barycenter 
 of (the $0$-flats in) $f$. By transitivity of $G_{32}$ on the flags of the Witting polytope, there exists $h\in G_{32}$ such that $gh^{-1}$ fixes 
 a flag $F$ of $f$. Consequently, $gh^{-1}$ fixes all the barycenters of flats of $F$, which span $\C^4$, and $g=h$.
 Hence $F_{G_{32}}(v)$ is isomorphic to the symmetry group of the face $f$, which is a Hessian polyhedron and has $G_{25}$ as symmetry group.
\end{proof}

It follows that the action of $H$ on $\HR$ is the action of $G_{25}$ on $\C^3$ described in the previous subsection. 
It remains to deduce the orbits of lines in $\HR$ under action of $G_{32}$ from their orbit under action of $H$. 

The other reflection hyperplanes intersect $\HR$ along $21$ planes. Twelve of these planes are the intersection of $\HR$ with exactly one other reflection hyperplane, 
orthogonal to $\HR$; these are the reflection planes for the action of $H$ on $\HR$. The $9$ remaining  planes are the intersection of $\HR$ 
with exactly $3$ other reflection hyperplanes, non orthogonal to $\HR$. 

Since the group $G_{32}$ acts transitively on the set of reflection hyperplanes, for any line of $\HR$ that lie on exactly one reflection hyperplane, 
its orbit under action of $G_{32}$ has order 40 times the order of its orbit under $H$. Such lines are those which lie neither
on the reflection planes of the action of $H$, nor on the intersection of 4 distinct reflection hyperplanes. 
\begin{lemma} \label{lemmaeigenhyp}
 The planes that are intersections of $\HR$ with $3$ other reflection hyperplanes are the following planes of $\HR$. They coincide with the regular 
 eigenspaces of dimension $2$ of elements of $H\cong G_{25}$ associated with eigenvalues of order $6$.
 $$\{z_1-\omega^kz_3=0\}\qquad\{z_2-\omega^kz_3=0\}\qquad\{z_1-\omega^kz_2=0\}\qquad\textrm{with }k=0,1,2$$
\end{lemma}
\begin{proof}
 The equations of the intersections of $\HR$ with 3 other reflection hyperplanes follow from the equations of these hyperplanes. 
 Conclude with Lemma \ref{lemmaeigenplanes}.
\end{proof}
It follows that the orbits of order 216 and 72 under $G_{25}$ provide orbits of order $8640$ and $2880$ respectively under $G_{32}$. 
For the orbits of lines at the intersection of 4 distincts reflection hyperplanes, we need the following lemma.
\begin{lemma}
 If $x\in\HR$ lies on the intersection of $\HR$ with exactly 3 other reflection hyperplanes, then the orbit of $\C x$ under $H$ 
 is the intersection of $\HR$ with the orbit of $\C x$ under $G_{32}$.
\end{lemma}
\begin{proof}
 Let $g\in G_{32}$ satisfy $g(\C x)\subset\HR$. We shall prove there is $h\in S_{G_{32}}(\HR)$ such that $h(\C x)=g(\C x)$. 
 If $g\notin S_{G_{32}}(\HR)$, then $g(\HR)$ is another reflection hyperplane $\HR_2$, and $g(\C x)$ lies on 4 reflection hyperplanes 
 $\HR$, $\HR_2$, $\HR_3$ and $\HR_4$. The reflections of hyperplane $\HR_3$ fix $\HR_3$ pointwise and permute cyclically $\HR$, $\HR_2$ and $\HR_4$. 
 Let $r$ be such a reflection with $r(\HR_2)=\HR$ and set $h=rg$.
\end{proof}
It follows that the orbits of order 108 and 54 under $G_{25}$ provide orbits of order $1080$ and $540$ respectively under $G_{32}$. 
It remains to treat the intersections of $\HR$ with an orthogonal hyperplane. We will see that the above lemma does not hold in this case.

We shall have a closer look at the arrangement of reflection hyperplanes in $\C^4$. Let $\A$ be the set of reflection hyperplanes of $G_{32}$. 
Let $\Hyp(\A)$ be the lattice of $\A$, {\em i.e.} the set of non-empty intersections of hyperplanes in $\A$. The next result describes this lattice.
\begin{proposition}
 The elements of the lattice $\Hyp(\A)$ are the point $\{0\}$ and:
 \begin{itemize}
  \item 40 hyperplanes,
  \item 240 planes intersection of 2 orthogonal hyperplanes,
  \item 90 planes intersection of 4 non orthogonal hyperplanes,
  \item 360 lines intersection of 5 hyperplanes,
  \item 40 lines intersection of 12 hyperplanes.
 \end{itemize}
The group $G_{32}$ is transitive on each of these classes.
\end{proposition}
\begin{proof}
 We have seen in Lemma \ref{lemmatrans32} that $G_{32}$ is transitive on its 40 reflection hyperplanes. Hence we can restrict our study to the 
 hyperplane $\HR=\{z_4=0\}$. From the equations of the reflection hyperplanes given above, we see that the intersection planes of $\HR$ with other hyperplanes 
 are the 12 planes $\plan_i=\{z_i=0\}$ for $i=1,2,3$ and $\plan_{jk}=\{z_1+\omega^jz_2+\omega^kz_3=0\}$ for $j,k=0,1,2$, intersections 
 of $\HR$ with one orthogonal hyperplane, and the 9 planes $\plan^3_j=\{z_1-\omega^jz_2=0\}$, $\plan^2_j=\{z_1-\omega^jz_3=0\}$, 
 $\plan^1_j=\{z_2-\omega^jz_3=0\}$ for $j=0,1,2$, intersections of $\HR$ with 3 non orthogonal hyperplanes. The planes $\plan_i$ and $\plan_{jk}$ 
 are the reflection planes of the action of $G_{25}$ on $\HR$, hence the transitivity of the action of $G_{32}$ on these planes follows from Lemma \ref{lemmatrans25}. 
 The transitivity on the planes $\plan_j^k$ was proved in Lemma \ref{lemmaeigenplanes}. 
 
 Since the planes $\plan_j^k$ do not intersect in the complement of the $\plan_i$ and $\plan_{jk}$, and by transitivity on these 12 planes, 
 we can restrict our study of the lines in $\Hyp(\A)$ to the plane $\plan_3=\{z_3=z_4=0\}$. We have two situations for these lines. 
 The lines $\{z_1=0\}=\plan_3\cap\plan_1\cap(\cap_{j=0}^2\plan_j^2)$ and $\{z_2=0\}=\plan_3\cap\plan_2\cap(\cap_{j=0}^2\plan_j^1)$ 
 are intersections of 12 distinct reflection hyperplanes. They belong to the same orbit under $G_{25}$ acting on $\HR$, hence to the same orbit under $G_{32}$. 
 The lines $\{z_1+\omega^jz_2=0\}=\plan_3\cap(\cap_{k=0}^2\plan_{jk})$ and $\{z_1-\omega^jz_2=0\}=\plan_3\cap\plan_j^3$ for $j=0,1,2$ are 
 intersections of 5 distinct reflection hyperplanes. They define two distinct orbits under $G_{25}$, but only one under $G_{32}$. 
 Indeed, the transformation given by the matrix $g_0= \left (\begin{smallmatrix} 1 & 0 & & \\ 0 & -1 & & \\ & & 0 & 1 \\ & & 1 & 0 \end{smallmatrix}\right )$
 exchanges these two orbits, and it belongs to $G_{32}$ since it preserves the vertices of the Witting polytope. 
\end{proof}
The two families of lines in $\Hyp(\A)$ form two orbits under action of $G_{32}$. We have studied the case of lines lying on exactly 1 or 4 hyperplanes, 
so it remains to treat the case of lines lying on exactly 2 hyperplanes. The above transitivity result allows to restrict the study to one of these 
planes. Set $\HR'=\{z_3=0\}$ and $\plan=\HR\cap\HR'$. Let $x\in\plan\setminus\{0\}$ be a point lying on no other reflection hyperplane. 
The size of its orbit under action of $G_{32}$ is 240 times the size of the trace of this orbit on $\plan$. Take $g\in G_{32}$ 
such that $g(\C x)\subset\plan$. Since $x$ lies on exactly 2 reflection hyperplanes, $g(x)$ also does, and it follows that $g$ stabilizes $\plan$. 
Hence $g$ is in the stabilizer of $\HR$ or of the form $g=g_0h$ with $h$ in the stabilizer of $\HR$ and 
$g_0$ as above. 

The stabilizer of $\plan$ in $H$ has been computed in Lemma \ref{lemmastab25}. The orbit under this stabilizer of $[1:z]$ is 
$\{[1:z],[1:\omega z],[1:\omega^2z],[1:\frac{1}{z}],[1:\frac{\omega}{z}],[1:\frac{\omega^2}{z}]\}$. Note that the lines lying on exactly $2$ reflection hyperplanes 
are the lines $[1:z]$ with $z\neq0, z^6\neq1$, which we assume in the sequel. The action of $g_0$ exchanges the orbit of $[1:z]$ with the orbit of $[1:-z]$. 
If $z^{12}=1$, 
this is twice the same orbit and it provides an orbit of size $1440$ under $G_{32}$. 
For $z^{12}\neq1$, we obtain orbits of size $2880$ under $G_{32}$.

We now want to see which orbits are defined by intersections of reflection hyperplanes with proper planes. Consider the transformation 
\[T=R_1^2R_2R_3^2R_4R_2^2R_3=\frac{\omega^2-\omega}{3}\scalebox{0.9}{$\begin{pmatrix}
                                                       0 & -\omega & \omega^2 & \omega \\
                                                       \omega & 1 & \omega & 0 \\
                                                       -\omega^2 & 0 & \omega^2 & -\omega \\
                                                       -\omega^2 & \omega & 0 & \omega
                                                      \end{pmatrix}$}.\]
It has order 24 and its eigenvalues are four distinct primitive $24^{th}$ roots of unity. The eigenvalues of $T^2$ are the square roots $\pm \zeta$ of $-\omega^2$, 
both of multiplicity 2. The eigenspace of $T^2$ associated with $\zeta$ is one of our proper planes, that we denote $E$. 
It is given by the following equations:
$$\left\lbrace\begin{array}{l}
               z_1+\zeta^2 z_3+(\zeta^3-2\zeta+1)z_4=0, \\
               z_2-(\zeta^3-2\zeta+1)z_3-(2\zeta^3-2\zeta^2-\zeta+2)z_4=0.
              \end{array}\right.$$
Laborious but direct computation shows that this plane meets the reflection hyperplanes along 8 lines lying on exactly 2 reflection hyperplanes 
and 6 lines lying on exactly 4 reflection hyperplanes. By transitivity of $G_{32}$ on the proper planes, the same occurs for all proper planes. 
We shall check that these intersections are the special orbits appearing on the intersections of exactly 2 or 4 reflection hyperplanes. Again by transitivity, 
it suffices to focus on one case for each configuration. 

First consider the plane $\plan=\{z_3=0,\,z_4=0\}$. All the lines of this plane are preserved by 
the $54$ transformations given by diagonal matrices with diagonal $\pm(\xi_1,\xi_1,\xi_2,\xi_3)$ with $\xi_j^3=1$ for $j=1,2,3$. For cardinality reasons, 
this is the stabilizer of the lines of $\plan$ whose $G_{32}$-orbit has size $2880$. Now if $D$ is a line defined as the intersection of $\plan$ with a proper plane, 
then $D$ lies on exactly two reflection planes, and the orbit of $D$ has size $2880$ or $1440$. But there is an element of $G_{32}$ which acts on $D$ by 
multiplication by a primitive $12^{th}$ root of unity. It follows that the stabilizer of $D$ has cardinality strictly greater than $54$, 
and that $D$ must lie in the $G_{32}$-orbit of order $1440$. 

Now consider the plane $\plan'=\{z_2-z_3=0,\,z_4=0\}$, intersection of the reflection hyperplanes $\HR=\{z_4=0\}$ and $\HR_\xi=\{z_2-z_3-\xi z_4=0\}$ with $\xi^3=1$. 
The lines of this plane are preserved by the transformations 
$$\scalebox{0.9}{$ \pm\begin{pmatrix}\xi_1&0&0&0\\0&\xi_1&0&0\\0&0&\xi_1&0\\0&0&0&\xi_2\end{pmatrix}$} \qquad\textrm{and}\qquad
\scalebox{0.9}{$ \pm\begin{pmatrix}\xi_1&0&0&0\\0&0&\xi_1&0\\0&\xi_1&0&0\\0&0&0&\xi_2\end{pmatrix}$},$$
with $\xi_1^3=\xi_2^3=1$. Note these 36 transformations preserve $\HR$. Now, the reflections of hyperplane $\HR_\xi$ permute cyclically the hyperplanes 
$\HR$ and $\HR_{\xi'}$ with $\xi'\neq\xi$. Hence for each $\HR_\xi$, we have an element $r\in G_{32}$ which sends $\HR$ onto $\HR_\xi$ and fixes $\plan'$ pointwise. 
Hence we have $144=36\times 4$ elements in $G_{32}$ which act on $\plan'$ by multiplication by a sixth root of unity. For cardinality reasons, this is the whole 
stabilizer of the lines of $\plan'$ which have $G_{32}$-orbit of size $1080$. The same argument as above shows the special orbit 
of lines lying on exactly 4 reflection hyperplanes is given by the intersections of these hyperplanes with the proper planes. 

To compute the number of proper planes meeting at each line of these two special orbits, we need to know the total number of proper planes. 
By Theorem \ref{thSp2}, the stabilizer of a proper plane is a \CRG of degrees 12 and 24, hence of order 288. Thus there are 540 proper planes.

Finally, we compute a representative for the three orbits listed in Proposition \ref{propn6}. 
For the orbit of order 6480, the given representative is the eigenvector of $T$ associated with the eigenvalue $\nu$ of order $24$. 
The representative of the orbit of order 5184 is obtained as an eigenvector of 
$$R_1^2R_2R_3^2R_4=\scalebox{0.9}{$ \frac{\omega^2-1}{3}\begin{pmatrix}
                                        0 & \omega & -\omega & -1 \\
                                        -\omega & -\omega & -\omega & 0 \\
                                        \omega & 0 & -\omega & 1 \\
                                        \omega & -\omega & 0 & -1
                                       \end{pmatrix}$}$$
associated with the eigenvalue $\nu$ of order 30.
The particular orbit of order 2880 is the particular orbit of lines in $\HR$ whose orbit under $G_{25}$ has order 72. The given representative
is deduced from Proposition \ref{propn5}.

 \section{Applications to the theory of differential equations} \label{secappli}
 \subsection{An exhaustive description of reducible algebraic solutions to Garnier systems}\label{sec Garnier}
 In \cite{cousinisom}, the following statement is proven.

 \begin{theorem}\label{thmcousin}
Let $\left(q_i\left((t_j)_j\right)\right)_i$ be a  solution of a Garnier system governing the isomonodromic deformation of a trace free logarithmic connection $\nabla$ on $\mathcal O_{\Pu}^{\oplus 2}$  with no apparent pole.
The following are equivalent.
\begin{enumerate}
\item \label{appli1}The multivalued functions $q_i$  are algebraic functions.
\item \label{appli2}The functions $q_i$ have finitely many branches.
\item \label{appli3} The conjugacy class $[\rho]$ of the monodromy representation of $\nabla$ has finite orbit under $\mathrm{MCG}_n\Pu$.
\end{enumerate}
\end{theorem}
Recall that the monodromy representation of a trace free rank $2$ meromorphic connection on $\Pu$ with poles in $\{x_1,\ldots,x_n\}$ is a morphism $\rho : \gf \rightarrow  \mathrm{SL}_2(\C)$. 
 
The first motivation for the present article was to understand what are the reducible representations $\rho$ that yield finite orbits in item $(\ref{appli3})$ of Theorem $\ref{thmcousin}$. 
The reducible representations $\rho : \gf \rightarrow  \mathrm{SL}_2(\C)$ are the representations such that $\rho(\gf)$ fixes (globally) a line in $\C^2$, as such they are conjugate to morphisms of the 
form $\alpha \mapsto M_{\alpha}= \left (\begin{smallmatrix} \mu_{\alpha}& q_{\alpha}\\ 0 &\mu_{\alpha}^{-1} \end{smallmatrix} \right)$. 
As scalar representations of $\gf$ are abelian, they are fixed by the pure mapping class group $\mathrm{PMCG}_n\Pu$ and the tensor product 
$\tilde{\rho} : \alpha\mapsto \mu_{\alpha}\otimes M_{\alpha}$ 
satisfies $\card(\mathrm{PMCG}_n\Pu \cdot [\tilde{\rho}])=\card(\mathrm{PMCG}_n \Pu \cdot [\rho])$. 
Obviously, we have a one to one correspondence $\tilde{\rho}\mapsto \hat{\rho}$ between the representations of the form $\alpha \mapsto  \left (\begin{smallmatrix} \mu^2_{\alpha}& \mu_{\alpha}q_{\alpha}\\ 0 &1\end{smallmatrix} \right)$  and the elements of $\Hom(\gf,\Aff)$\footnote{Actually, the element $\hat{\rho}\in \Hom(\gf,\Aff)$ is nothing but the monodromy of the projective connection $\mathbb{P}\nabla$, for a suitable embedding of $\Aff$ in $\mathrm{PSL}_2(\C).$}. Conjugacy of two elements $\tilde{\rho}_1,\tilde{\rho}_2$
by an element in $\mathrm{SL}_2(\C)$ implies conjugacy of $\hat{\rho}_1,\hat{\rho}_2$ by some element of $\Aff$. 

This implies, for every reducible $\rho : \gf \rightarrow  \mathrm{SL}_2(\C)$,
 $\card (\mathrm{PMCG}_n\Pu\cdot[\rho])=\card (\mathrm{PMCG}_n\Pu\cdot[\hat{\rho}])$ and explains how our initial question reduces to considerations on the action of  $\mathrm{PMCG}_n\Pu$ on $\Hom(\gf,\Aff)/\Aff$.

In this way, an immediate consequence of Theorem $\ref{thnqcq}$ and Theorem $\ref{thcasfinis}$ is the following. For technicalities concerning connections and Garnier systems we refer to \cite{cousinisom}.

\begin{corollary}\label{corGarnier}
Let $n>4$.
Let  $q(t)=\left(q_i\left((t_j)_j\right)\right)_i$ be a (multivalued) solution of a Garnier system that governs the isomonodromic deformation of $\nabla$, 
a trace free logarithmic connection on $\mathcal O_{\Pu}^{\oplus 2}$, with  exactly $n$ poles in $P:=\{t_1^0,\ldots,t_{n-3}^0,0,1,\infty\}$, 
all of which are supposed non apparent.
Assume the monodromy of $\nabla$ is reducible and non abelian.

The function $q(t)$ is algebraic if and only if one of the following is true, up to a meromorphic gauge transformation 
$Z\mapsto G(x)\cdot \tilde{Z}$ with $G(x)\in \mathrm{GL}_2(\C(x))$ and up to tensor product with a meromorphic connection on $\mathcal O_{\Pu}$.

\begin{enumerate}
\item   There exists a permutation $\sigma$ of $P$ such that $\nabla$ takes the form\\ $\nabla : Z\mapsto dZ-\sum_{p\in \{t_1^0,\ldots,t_{n-1}^0\}} A_{\sigma(p)} \frac{dx}{x-p}\cdot Z$, with $\sum_{p\in P}A_p=0$, $A_0$ a trace free constant upper triangular matrix and $A_p=\left (\begin{smallmatrix} 0& c_p\\ 0 &0\end{smallmatrix} \right), c_p\in \C^*$ for $p\in P\setminus \{0,\infty\}$;

\item $n=5$ and  there exists a permutation $\sigma$ of $P$ such that $\nabla$ takes the form\\
$\nabla : Z\mapsto dZ\pm \sum_{p\in \{t_1^0,t_{2}^0,0,1\}} A_\sigma(p) \frac{dx}{x-p}\cdot Z$, with $\sum_{p\in P}A_p=0$ and $A_p=\left (\begin{smallmatrix} 1/12& c_p\\ 0 &-1/12\end{smallmatrix} \right),$ $c_p\in \C$, for $p\in P\setminus \{\infty\}$;
\item $n=6$ and 
$\nabla : Z\mapsto dZ\pm \sum_{p\in \{t_1^0,t_{2}^0,t_{3}^0,0,1\}} A_p \frac{dx}{x-p} \cdot Z$, with $A_p=\left (\begin{smallmatrix} 1/12& c_p\\ 0 &-1/12\end{smallmatrix} \right),c_p\in \C$, for $p\in P\setminus \{\infty\}$.
\end{enumerate}
\end{corollary}
Of course, our study of finite orbits also allows to give a similar result for the more classical case $n=4$, but we omit its lengthy statement.

\begin{rem}
Under the hypotheses of Corollary $\ref{corGarnier}$ (even allowing $n\geq 4$), we may as well see from the proof of \cite[Th. $2.6.3$]{cousinisom} that the number 
of branches for the tuple formed by all the elementary symmetric functions in the coordinates $q_i(t)$ of $q(t)$ equals the size of the orbit 
$\mathrm{PMCG}_n\Pu \cdot [\rho]$, where $\rho$ is the monodromy representation of $\nabla$. 

In particular, the knowledge of this number of branches gives quite restrictive information on the monodromy of $\nabla$. For example with $n=4$, if one knows that $\nabla$ has the form $$Z\mapsto dZ-\left [\left (\begin{smallmatrix} 1/48&*\\0&-1/48
\end{smallmatrix}\right )\frac{dx}{x-t}+\left (\begin{smallmatrix} 5/48&*\\0&-5/48
\end{smallmatrix}\right )\frac{dx}{x}+\left (\begin{smallmatrix} 7/48&*\\0&-7/48
\end{smallmatrix}\right )\frac{dx}{x-1}\right]\cdot Z$$ and that the Painlev\'e VI solution associated to $\nabla$  has exactly $8$ branches, 
the projective monodromy of $\nabla$ must be given by $\lambda=(\eta,\eta^5,\eta^7,\eta^{11})$, $\tau=(0,1,0,1)$, up to conjugation and up to 
the action of the pure braids, with $\eta=e^{-i\pi/12}$ -- the minus sign is present because we consider the monodromy representation, not anti-representation. 
\end{rem}

\subsection{Riemann-Hilbert Problem}

 In general it is difficult to compute the monodromy group of a given differential equation and, conversely, to write explicitely a differential equation for a given prescribed monodromy group. This latter problem is called the Riemann-Hilbert problem.
 Also, if a differential equation $\frac{dY}{dx}=A(x)\cdot Y$ on $\Pu$ has simple poles, then its differential Galois group is the Zariski closure of the monodromy group \cite[Th. $5.8$ p. 149]{MR1960772}.
 To this regard, finding an explicit equation with simple poles for a given finite subgroup of $\mathrm{GL}_m(\C)$ can be interpreted as the resolution 
 of a differential inverse Galois problem. An approach for the realization of finite groups is given in \cite{MR1752769}.
 In this section, we show how one can take benefit of our previous work for these questions.

The first point is to remark that the groups $\Gamma_{\lambda}$ and $\hat{\Gamma}_{\lambda}$ are tightly related with monodromy groups of certain explicit connections. 
We explain this in Sections \ref{linschles} and~\ref{quotconnections}.

 \subsubsection{Linearization of Schlesinger for reducible connections on $\mathcal O_{\Pu}^{\oplus 2}$}\label{linschles}
 For commodity of indexation, we shall use the notation $t_{n-2}=t_{n-2}^0=0,t_{n-1}=t_{n-1}^0=1,t_{n}=t_{n}^0=\infty$.  
 Fix $P:=\{t_1^0,\ldots,t_{n}^0\}\subset \Pu$, with cardinality $n$. 
 Let us consider $\nabla^0$ a trace free logarithmic reducible connection on $\mathcal O_{\Pu}^{\oplus 2}$ with poles in $P$.
 \[ \nabla^0 : Z \mapsto dZ-\sum_{i=1}^{n-1} A^0_i \frac{dx}{x-t_i^0}\cdot Z\] 
 
 The Schlesinger deformation of $\nabla^0$ is the family \[\nabla^t : Z \mapsto dZ-\sum_{i=0}^{n-1} A_i(t) \frac{dx}{x-t_i}\cdot Z\] parametrized by the universal covering $(\tilde{F}_{3,n-3}\Pu,\tilde{t^0})\rightarrow (F_{3,n-3}\Pu,t^0)$ of $F_{3,n-3}\Pu:=\{(t_1,\ldots,t_{n-3},0,1,\infty)\in F_{0,n}\Pu \}$ such that for $i=1,\ldots,n-1$:
\begin{equation} \label{Schlesinger}\left \lbrace \begin{array}{l}A_i(\tilde{t^0})=A^0_i,\\ 
 \displaystyle dA_i=-\sum_{j=1, j\neq i}^{n-1}[A_i,A_j]\frac{d(t_i-t_j)}{t_i-t_j}.
 \end{array} \right .\end{equation}
 The existence and uniqueness of this deformation is ensured by an integrability theorem of Schlesinger, compare \cite[Theorem 3.1]{MR0728431}. 
 This integrability is tantamount to the flatness of  the connection $\hat{\nabla}$ on $\mathcal O_{\tilde{F}_{3,n-3}\Pu\times \Pu}^{\oplus 2}$ defined by
 \[\hat{\nabla} : Z \mapsto dZ-\sum_{i=1}^{n-1} A_i(t) \frac{d(x-t_i)}{x-t_i}\cdot Z.\]
Outside the poles of the solution $(A_i(t))_i$, this deformation coincides with the universal isomonodromic deformation of $\nabla^0$.

From now on, assume that the first factor in the direct sum $\mathcal O_{\Pu}^{\oplus 2}$ is invariant by $\nabla^0$ -- \textit{i.e.} it is saturated by horizontal sections.
 By isomonodromy and the fact that the constant section given by $(1,0)$ is a horizontal section of  $\hat{\nabla}_{\vert \tilde{F}_{3,n-3}\Pu\times \{\infty\}}$ -- this general property of Schlesinger deformations is explained in \cite{MR0728431} -- the first factor of $\mathcal O_{\tilde{F}_{3,n-3}\Pu\times \Pu}^{\oplus 2}$ is again invariant by $\hat{\nabla}$. In particular, the matrices $A_i(t)$ are all upper triangular:
 \[A_i(t)=\left (\begin{smallmatrix} \theta_i/2& c_i(t)\\ 0 &-\theta_i/2\end{smallmatrix} \right ). \]
 
 From this observation, we readily derive that the Schlesinger system $(\ref{Schlesinger})$ takes the form
 \[\left \lbrace \begin{array}{l}
 c(\tilde{t}^0)=c^0 ,\\
 dc=\Omega\cdot c.
 \end{array}\right.\]
 where $c(t)$ is the column vector with lines $(c_i(t))_{i=1,\ldots,n-1}$, and $\Omega=\sum\limits_{1 \leq i<j \leq n-1}B^{i,j}\frac{d(t_i-t_j)}{t_i-t_j}$ with the entries of the residues $B^{i,j}$ specified as follows.
 \[\left \lbrace \begin{array}{l}
 B^{i,j}_{k,l}=0\mbox{ if }(k,l)\not \in \{(i,i),(i,j),(j,j),(j,i)\},\\
 B^{i,j}_{i,i}=\theta_j, B^{i,j}_{j,j}=\theta_i,\\
 B^{i,j}_{i,j}=-\theta_i, B^{i,j}_{j,i}=-\theta_j.
 \end{array} \right.\]
 
 The integrability of this system ($d\Omega =\Omega\wedge \Omega$) is automatic since Schlesinger system is integrable.
Hence, for fixed $(\theta_i)$, the family $c(t)=(c_i(t))$ defines an isomonodromic family of ``triangular rank $2$ systems'' as above if and only if $c(t)$ is a horizontal section for 
the connection on $\mathcal O^{\oplus n-1}_{F_{3,n-3}\Pu}$ defined as $D_{\theta} : c\mapsto dc-\Omega\cdot c$.

 This linearization phenomenon  is well known, see \cite{MR830631}, \cite{MR1892536}.
 In a quiet different language, it was also mentioned by Deligne and Mostow  in \cite[\S $3.2$]{MR849651}, as we shall see in the following section.
 \subsubsection{Connections on quotient bundles}\label{quotconnections}
 
 In the notation of \cite{MR849651}, our $F_{0,n}\Pu$ corresponds to $M$ and our $n$ corresponds to $N$.
 
 Fix a $1$-form $\eta^0=\sum_{j=1}^{n-1}\theta_j\frac{dx}{x-t^0_j}$ on $\Pu$, denote $\theta_{n}$ its residue at $\infty=t^0_n$.  The flat connection $z\mapsto dz-\eta^0 z$ on $\mathcal O_{\Pu}$ determines a rank one local system $L^0$
 on $\Pu\setminus \{t_1^0,\ldots,t_{n}^0\}$ with monodromy  $\lambda_j=e^{-2i\pi\theta_j}$ around $t_j^0$.
 Similarly, the closed $1$-form $\eta=\sum_{j=1}^{n-1}\theta_j\frac{d(x-t_j)}{x-t_j}$ on $F_{3,n-3}\Pu$ determines a local system  $L$ on $F_{3,n-3}\Pu\times \Pu$  that restricts to $L^0$ in the slice $t=t^0$. Denote $(L^t)_t$ (resp. $(\eta^t)_t$) the family of restrictions of $L$ (resp. $\eta$) to the levels of the projection 
 $F_{3,n-3}\Pu \times \Pu \rightarrow F_{3,n-3}\Pu$.
 
 Consider the set $E_{\eta}^t$ of Riccati equations $dz=\eta z +\omega$, with $\omega=\sum_{j=1}^{n-1} c_j \frac{dx}{x-t_j}$.
 It identifies with the set $\Omega^1_{\Pu}(log P_t )$ of logarithmic $1$-forms on $\Pu$ with poles in $P_t:=\{t_1,\ldots,t_{n}\}$.
 These Riccati equations are the projectivizations of the triangular rank $2$ systems considered in the previous section.
 
 Choosing a base point $\star$, we have a monodromy map 
 \[\begin{array}{lcl}E_{\eta}^t&\rightarrow& \Hom_{\lambda}(\pi_1(\Pu\setminus P_t,\star),\Aff)\\
 \omega=\sum_{j=1}^{n-1}c_j \frac{dx}{x-t_j}&\longmapsto& \rho_{\omega}. \end{array}\]
 If we encode the elements of the target space by the translation parts $(\tau_i)_{1\leq i \leq n-1}$ like in the previous sections, the image of $\omega$ is $(-\int_{\alpha_i}e^{-\int_{\star}\eta}\omega)_{1\leq i \leq n-1}$, in particular, this map is linear.
 
  Provided $\lambda$ is nontrivial and, for $i=1,\ldots,n$, $\theta_i \not \in \Z_{<0}$, the method of \cite[Th. $3$]{MR892029} proves surjectivy of this map.
Then, equality of the dimensions of the source and the target shows it is an isomorphism.
 The affine change of variable $\tilde{z}=az+b$ changes $dz=\eta^t z +\omega$ to $d\tilde{z}=\eta^t z +\tilde{\omega}$ with $\tilde{\omega}=a\omega-b\eta^t$. 
 In particular, considering the quotient by $(\C,+)$,  we get a comparison theorem: \[\mathrm{H}^1_{dR,log}(\Pu,(\mathcal O_{\Pu}, z\mapsto dz-\eta^t z)):=\Omega^1_{\Pu}(log P_t )/\C \eta^t \simeq \mathrm{H}^1(\Pu\setminus P_t,L^t).\]

By isomonodromy, a horizontal section $(c_j(t))$ of $D_{\theta}$ defines a family of Riccati equations $(dz=\eta^t z+\omega^t)_t$ such that, 
for neighboring $t^1,t^0$, identifying $\pi_1(\Pu\setminus P_{t^1},\star)$ and $\pi_1(\Pu\setminus P_{t^0},\star)$ \textit{via} a topological local 
trivialization of the family $(P_t)$, the monodromy representations $\rho_{\omega^{t^1}}$ and $\rho_{\omega^{t^0}}$ are conjugate by an element of $\Aff$. 
In particular, $\rho_{\omega^{t^1}}$ has one fixed point in $\C$ if and only if the same holds for $\rho_{\omega^{t^0}}$. This implies that the local system 
associated to $D_{\theta}$ descends to a local system that allows to identify  $\mathrm{H}^1(\Pu,L^{t^0})$ with $\mathrm{H}^1(\Pu,L^{t^1})$, 
for neighboring $t^1,t^0$, that must be projectively isomorphic to the local system $R^1\pi_*L$ mentioned in \cite[\S $3.2$]{MR849651} -- up to restriction to $F_{3,n-3}\Pu\subset M$.

For the connection $D_{\theta}$, this means the line bundle $\delta$ generated by $v_{\theta}\equiv(\theta_1,\ldots,\theta_{n-1})$ is invariant. 
Assuming $\theta_1\neq 0$, one can compute the matrix of $D_{\theta}$ in the basis $(v_{\theta},e_2,\ldots,e_{n-1})$ instead of the canonical basis $(e_1,\ldots,e_{n-1})$.
This matrix takes the block form \[\tilde{\Omega}=P^{-1}\Omega P=\left(\begin{array}{cc}*& *\cdots*\\ 0&C \end{array}\right), \mbox{ for some } C\in \mathrm{M}_{n-2}(\C).\]
The quotient connection on $\mathcal O^{\oplus n-1}_{F_{3,n-3}\Pu}/ \delta$ is isomorphic to the connection $ d_{\theta} :Z\mapsto dZ-C\cdot Z$ on $\mathcal O^{\oplus n-2}_{F_{3,n-3}\Pu}$.
Computation shows that $C=\sum_{1\leq i<j<n} C^{i,j}\frac{d(t_i-t_j)}{t_i-t_j}$ with 
 \begin{equation} \label{chitheta}
 \left \lbrace \begin{array}{l l}
 C^{1,j}_{k,l}=0&\mbox{ if }l \neq j-1,\\
 C^{1,j}_{k,j-1}=\theta_{k+1}& \mbox{ if } k\neq j-1,\\
 C^{1,j}_{j-1,j-1}=\theta_{k+1}+\theta_1 &\mbox{ if } k= j-1,\\
 C^{i,j}_{k,l}=B^{i,j}_{k+1,l+1}&\mbox{ if } 1<i<j.
 \end{array} \right.\end{equation}
 The induced flat projective space bundle on $F_{3,n-3}\Pu$ has fiber \[\PC{n-3}\simeq \mathrm{PH}^1(\Pu \setminus P_{t^0} ,L^{t^0})\simeq (\Hom_{\lambda}(\gf,\Aff)/\Aff)\setminus \{[0]\}\] and its monodromy representation
$$\pi_1(F_{3,n-3}\Pu)\rightarrow Aut(\mathrm{PH}^1(\Pu \setminus P_{t^0},L^{t^0}))\simeq \mathrm{PGL}_{n-2}(\C)$$
 is nothing but the pure braid group action described in $(\ref{descriptionaction1})$ and $(\ref{descriptionaction2})$, with $L^{t^0}\simeq\C_{\lambda}$, for $\lambda_j=e^{-2i\pi\theta_j}, j=1,\ldots,n$.
The induced morphism $\pi_1(F_{3,n-3}\Pu) \rightarrow \Gamma_{\lambda} $ is onto, because  $\pi_1(F_{3,n-3}\Pu)\rightarrow \mathrm{PMCG}_n\Pu$ is surjective 
-- actually an isomorphism, see Sublemma \ref{sublemma}.

  We turn to a comparison of the monodromy group of $d_\theta$ and the group $\hat{\Gamma}_{\lambda}$, which are equal modulo scalars, as we just explained. We first mention the following.
  \begin{lemma} \label{lem monod} Let $n>3$ and $\lambda \in \Hom(\gf,\C^*)$.
  Consider the subgroup $H=\pi_1(F_{2,n-3}\D',z')$ of $\PBDm$.
  This group surjects onto the subgroup $\pi_1(F_{3,n-3}\Pu,x)$ under the map $\PBDm\rightarrow \PBn$ of Section \ref{subsecprelim}.
  Let $\varphi :H\rightarrow \PBDm \rightarrow \hat{\Gamma}_{\lambda}$ be the composition of natural maps.
  We have $\hat{\Gamma}_{\lambda}=\langle \varphi(H),\lambda_n\rangle$.
   \end{lemma}
  \begin{proof}
  The surjectivity assertion follows from the definitions.
  For the second point, we use the identity $\langle \Delta^2_{\D'}\rangle \times \pi_1(F_{2,n-3}\D',z')=\PBDm$ \cite[Th. 4]{MR2092062}, where the braid 
  $\Delta^2_{\D'}$ corresponds to the full twist on the boundary of $\D'$. The action of this braid on $\Lambda_n$ is the conjugation by
 $\alpha_n$. This implies that its action on $\Hn$ is the scalar multiplication by $\lambda^{-1}_n$. The conclusion follows.
  \end{proof}
 
  \begin{proposition}\label{lemgalinv}
  Fix $\theta_j \in \C\setminus \Z_{<0}$, $\lambda_j=e^{-2i\pi\theta_j}, j=1,\ldots,n$. Assume $\theta_0\neq 0$ and $\lambda$ nontrivial. 
   Consider the monodromy representation $\rho_{d_{\theta}}$ of the connection $d_{\theta}$ defined by equation $(\ref{chitheta})$. 
With the notation of Lemma~$\ref{lem monod}$, the representation $\varphi : H \rightarrow \hat{\Gamma}_{\lambda}$ is isomorphic to the composition $H\rightarrow\pi_1(F_{3,n-3}\Pu)\stackrel{\rho_{d_{\theta}}}{\longrightarrow} \mathrm{GL}_{n-2}(\C)$.
  In particular, the natural map $\pi_1(F_{3,n-3}\Pu)=\mathrm{PMCG}_n(\Pu)\rightarrow \Gamma_{\lambda}$ lifts to a map $\mathrm{PMCG}_n(\Pu)\rightarrow \hat{\Gamma}_{\lambda}$.
\end{proposition}
 \begin{proof}
 We have seen both representations are projectively isomorphic. We make a local analysis at the component $D_{j,k}=\{t_j=t_k\}$ of 
 $D=(\Pu)^{n-3}\setminus F_{3,n-3}\Pu$. Let $p$ be a smooth point of $D$, $p\in D_{j,k}$. Restrict $d_{\theta}$ to an embedded disk centered at $p$, 
 transverse to $D$, with  coordinate $z$. After a local holomorphic gauge transformation, this restriction $\chi^0_{\theta}$ has the form  
 $Z\mapsto dZ-A(z)Z \frac{dz}{z}$ and  is in Poincaré-Dulac-Levelt reduced form; in particular $A(z)$ is an upper triangular matrix of monomials, 
 with constant diagonal part $\Lambda$. By \cite[Cor. $16.19$ p. $274$]{MR2363178}, the local monodromy of $d_{\theta}$ around 
 $D_{j,k}$ has Jordan decomposition $DU=UD$ with $D=exp(-2i \pi \Lambda), U=exp\left(2i \pi (\Lambda-A(1))\right)$. As $DU$ is a scalar multiple of a complex reflection, it is diagonalizable, 
$DU=D$, $A(1)$ is diagonal and $A(x)=\Lambda$. 
 Thus $d_{\theta}^0$  is $Z\mapsto dZ-\Lambda Z \frac{dz}{z}$. The residues $\Lambda$ and $C_{j,k}$ are conjugate, so that
 the monodromy of $d_{\theta}$ around $D_{j,k}$ is conjugate to $exp(-2i\pi C_{j,k})$. This matrix is the identity or a complex reflection since 
 $C_{j,k}$ has rank $\leq 1$ and it has the same eigenvalues as $\rho(\sigma_{j,k})$ since $trace(C_{j,k})=\theta_j+\theta_k$. This yields the conclusion, 
 for two proportional reflections with equal eigenvalues are equal.
 \end{proof}

 \subsubsection{Two connections on the Riemann Sphere}
 
 In section $\ref{sec56}$,  we have seen that if $n=5$, $\lambda=(\zeta,\zeta,\zeta,\zeta,\zeta^2)$ (resp. n=6, $\lambda=(\zeta,\zeta,\zeta,\zeta,\zeta,\zeta)$), for some $\zeta$ of order $6$, the group $\hat{\Gamma}_{\lambda}$ is conjugate to the primitive subgroup of $\mathrm{GL}_3(\C)$ of order $648$  (resp. primitive subgroup of $\mathrm{GL}_4(\C)$ of order $155520$) and is generated by the images of the braids $\sigma_{1,j}$, $1<j\leq4$ (resp. $1<j\leq5$).
 
In particular, if $\lambda_j=e^{-2i\pi\theta_j}, j=1,\ldots,n$, taking $x=t^0=(t_1^0,\ldots,t_{n-3}^0,0,1,\infty)$ as our base point in $F_{3,n-3}\Pu$, Proposition $\ref{lemgalinv}$ shows that the  restriction of the connection $d_{\theta}$ to the $n$-punctured projective line 
$\Pu_{t^0}:=\{t\in F_{3,n-3}\Pu \vert i>1 \Rightarrow t_i=t_i^0\}$ has monodromy group $\hat{\Gamma}_{\lambda}$, since the braids  $\sigma_{1,j} \in \pi_1(F_{0,n}\Pu,x)$ can be represented by elements of $\pi_1(F_{n-2,1}\D',z')$.
\begin{corollary}\label{corgalinv}
Let $s_1,s_2,0,1,\infty$ be five distinct points on the Riemann sphere.
\begin{enumerate}
\item \label{corgalinv1}
The  monodromy group of the following logarithmic connection on $\mathcal{O}^{\oplus 3}_{\Pu}$ is the primitive complex reflection group of order $648$.
 \[Z\mapsto 
dZ \pm \left(
\left (\begin{array}{ccc}
1/3&0&0\\
1/6&0&0\\
1/6&0&0\\  
\end{array}\right)\frac{dx}{x-s_1}
+
\left (\begin{array}{ccc}
0&1/6&0\\
0&1/3&0\\
0&1/6&0\\  
\end{array}\right)\frac{dx}{x}
+\left (\begin{array}{ccc}
0&0&1/6\\
0&0&1/6\\
0&0&1/3\\  
\end{array}\right)\frac{dx}{x-1}\right)\cdot Z\]

\item \label{corgalinv2}
The  monodromy group of the following logarithmic connection on $\mathcal{O}^{\oplus 4}_{\Pu}$ is the primitive complex reflection group of order $155520$.
 \[Z\mapsto 
dZ \pm \left( \begin{array}{cc}
\scalebox{0.9}{$ \left (\begin{array}{cccc}
1/3&0&0&0\\
1/6&0&0&0\\
1/6&0&0&0\\
1/6&0&0&0\\  
\end{array}\right)$}\frac{dx}{x-s_1}
+
\scalebox{0.9}{$ \left (\begin{array}{cccc}
0&1/6&0&0\\
0&1/3&0&0\\
0&1/6&0&0\\
0&1/6&0&0\\  
\end{array}\right)$}\frac{dx}{x-s_2}\\
+
\scalebox{0.9}{$\left (\begin{array}{cccc}
0&0&1/6&0\\
0&0&1/6&0\\
0&0&1/3&0\\
0&0&1/6&0\\
\end{array}\right)$}\frac{dx}{x}
+\scalebox{0.9}{$ \left (\begin{array}{cccc}
0&0&0&1/6\\
0&0&0&1/6\\
0&0&0&1/6\\
0&0&0&1/3\\
\end{array}\right)$}\frac{dx}{x-1}\\
\end{array}\right)\cdot Z\]
\end{enumerate}
\end{corollary}
\begin{proof}
The connection we propose in ($\ref{corgalinv1}$) (resp. ($\ref{corgalinv2}$)) is the connection ${d_{\theta}}_{\vert \Pu_{t^0}}$ for $n=5$, 
$(\theta_1,\ldots,\theta_4)=\pm(1/6,\ldots,1/6)$ and $t_2^0=s_1$ (resp. for $n=6$, $(\theta_1,\ldots,\theta_5)=\pm(1/6,\ldots,1/6)$ and $t_2^0=s_1,t_3^0=s_2$). 
The choice of sign corresponds to the choice of the sixth root of unity $\zeta$.
\end{proof}

\subsection{Relations with Lauricella hypergeometric functions}
Here we explain in elementary terms the link between the connection $d_{\theta}$ and some functions studied by Lauricella.

Fix a positive integer $N$ and consider the function $F_D$ introduced in \cite{Laur}.\[F_D(\alpha,\beta_1,\ldots,\beta_N,\gamma;t)=\sum_{m\in (\Z_{\geq 0})^N} \frac{(\alpha)_{\vert m\vert}}{(\gamma)_{\vert m\vert}}\prod_{i=1}^N\frac{(\beta_i)_{m_i}}{(1)_{ m_i}}t_i^{m_i};\]
where $\vert m \vert=\sum m_i$, $\gamma\not \in \Z_{<0}$ and $(a)_k=\Gamma(a+k)/\Gamma(a)$.

This series is  convergent in the polydisk $\{\vert t_i \vert <1, i=1,\ldots,N\}$.
By a recursive argument, the function $F_D$ can be characterized as the unique holomorphic solution of the system $(\mathcal L^iy=0)_i$ satisfying $y(0)=1$, with $\mathcal L^i$ as follows, for $\delta_j=t_j\frac{\partial}{\partial t_j}$.

\[\mathcal L^i=t_i(\beta_i+\delta_i)(\alpha+\sum_{j=1}^N \delta_j)-\delta_i(\gamma-1+\sum_{j=1}^N \delta_j),~~~~i=1,\ldots,N;\]
Of course we may, and we will, consider the system $(\mathcal L^iy=0)_i$ for any value of the parameter $\gamma$, including negative integers.

In \cite[p. 138--140]{Laur}, Lauricella explains that, if $\gamma-\alpha\neq 1$, $(\mathcal L^iy=0)_i$ implies
\[\delta_i\delta_j y=\frac{x_j\beta_j\delta_i y -x_i\beta_i\delta_j y}{x_i-x_j}, i\neq j. \]

Setting $u_0:=y, u_i:=\delta_i y, i=1,\ldots, N$, this allows to show that $(\mathcal L^iy=0)_i$ is tantamount to
$dU=E U$, where $U$ is the column vector with entries $u_0,\ldots,u_N$ and $E$ the size $N+1$ square matrix given by the following equations.

\begin{itemize} \itemsep 0.15cm
 \item $E=\sum_{1\leq i<j\leq N+2} E^{i,j} \frac{d(t_i-t_j)}{t_i-t_j}$
 \item If $j\leq N$, $E^{i,j}_{k,\ell}=0$, except for: \\ 
   $E^{i,j}_{i+1,i+1}=-\beta_j,\quad E^{i,j}_{j+1,j+1}=-\beta_j, \quad E^{i,j}_{i+1,j+1}=\beta_i, \quad E^{i,j}_{j+1,i+1}=\beta_j$.
 \item $E^{i,N+1}_{k,\ell}=0$, except for:\\ 
   $E^{i,N+1}_{i+1,i+1}=1-\gamma+\sum_{m\neq i}\beta_{m}, \quad E^{i,N+1}_{1,i+1}=1,\quad E^{i,N+1}_{k,i+1}=-\beta_{k-1}, k>1, k\neq i+1$.
 \item $E^{i,N+2}_{k,\ell}=0$, except for:\\ 
   $E^{i,N+2}_{i+1,i+1}=\gamma-(\alpha+\beta_k+1), \quad E^{i,N+2}_{i+1,1}=-\alpha\beta_i, \quad E^{i,N+2}_{i+1,\ell}=-\beta_{i}, \ell>1, \ell\neq i+1$.
\end{itemize}

Provided $\alpha \beta_1(\gamma-1-\sum \beta_i)\neq 0$, setting $n=N+3$,
we may conjugate the system $dU=EU$ to the system that describes the horizontal sections of $d_{\theta}$, for the unique tuple $(\theta_i)_{i=1,\ldots, n-1}$ that satisfies
\begin{equation}\label{theta/beta}
\left \lbrace
\begin{array}{l}
\beta_i=-\theta_i, i=1,\ldots,N;\\
\alpha=-\sum_{i=1}^{n-1}\theta_i;\\
\gamma=1-\sum_{i=1}^{n-2}\theta_i.\\
\end{array}
\right.
\end{equation}

Indeed, we give below a constant matrix $G$ such that $E^{i,j}G=C^{i,j}G$, for any $i\leq n-3, i<j\leq n-1$, where $C^{i,j}$, described in $(\ref{chitheta})$, are the residues of $d_{\theta}$. The matrix $G$ is defined as $G=K+L+M$, where
\[\begin{array}{l}
K_{1,j}=\theta_{N+1}, 1\leq j \leq N+1; K_{i,j}=0, i>1;\\
L_{1,N}=\alpha; L_{i,N}=\alpha\theta_{i-1}, 1<i\leq N+1; L_{i,j}=0, j\neq N;\\
M_{j+2,j}=-\alpha\theta_{N+1},1\leq j <N; M_{i,j}=0, i\neq j+2.\\
\end{array}
\]
We readily see $det(G)=(-1)^{N+1}\theta_1(\alpha\theta_{N+1})^N=\beta_1\alpha^N(\gamma-1-\sum\beta_i)^N$. It is a straightforward computation to check the claimed conjugation property.

The (multiform) solutions of $(\mathcal L^iy=0)_i$ are called Lauricella hypergeometric functions of type $F_D$ with parameters $(\alpha,\beta_1,\ldots,\beta_N,\gamma)$. 

In his paper, the aim of Lauricella was to pursue the line of Appel \cite{Appel} in studying generalizations of Gau\ss's hypergeometric functions \cite{Gauss}. For $N=1$ the system $(\mathcal L^iy=0)_i$ associated to the parameters $\alpha,\beta=\beta_1,\gamma$ is a single equation which takes the following elementary form, for $t=t_1, h'=dh/dt$.
\[t(1-t)y''+(\gamma-(\alpha+\beta+1)t)y'-\alpha\beta y=0\]
 This equation is Gau\ss's hypergeometric equation. The cases in which this equation possesses an algebraic solution have been studied by H.A. Schwarz in \cite{MR1579568}. The so called Schwarz's list is the list of cases that correspond to equations for which \emph{every solution} is algebraic and no solution has rational logarithmic derivative (irreducibility). According to his parameters, there are fifteen such cases.

From the above conjugation and our work concerning finite branching of horizontal sections of $d_{\theta}$ ($N=1$), one can recover Schwarz's list. In our tables $\ref{tabn4/1}$, $\ref{tabn4/2}$ and $\ref{tabn4/3}$,    there are eleven linear parts corresponding to a finite irreducible monodromy group. The discrepancy between eleven and fifteen is due to the fact that we do not use the same equivalence relation amongst the considered differential equations. This is not to be surprising, as we have a quite different approach to the problem.
Notice that our lists contain extra information, namely the translation parts corresponding to the special orbits of the various monodromy groups.

More easily, our study for $N>1$ allows to derive the following which is originally due to Bod \cite[Th. $2.1.11$ and Th. $2.1.12$]{MR2852217}. 
\begin{theorem}\label{th Bod}
Let $N>1$, $\alpha,\beta_1,\ldots,\beta_N,\gamma \in \C$ and define $\theta_{N+1}:=(1+\sum\beta_j-\gamma), \theta_{N+2}:=\gamma-\alpha-1$, $\theta_{N+3}=\alpha$ and $\theta_j:=-\beta_j$, $j=1,\ldots,N$.
If  $\theta_i\not \in \Z, i=1,\ldots,N+2$, the following are equivalent.
\begin{enumerate}
\item There exists a finite branching nontrivial hypergeometric function of type $F_D$ with parameters $(\alpha,\beta_1,\ldots,\beta_N,\gamma)$. 
\item All the hypergeometric functions of type $F_D$ with parameters $(\alpha,\beta_1,\ldots,\beta_N,\gamma)$ have finitely many branches. 
\item\label{3bod} \begin{enumerate}
\item $N=2$ and four of the elements $\theta_1,\ldots,\theta_{5}$ are equal  mod $\Z$, with common value $\pm \frac{1}{6}$ mod $\Z$ or
\item  $N=3$ and the elements  $\theta_1,\ldots,\theta_{6}$ are equal  mod $\Z$, with value $\pm \frac{1}{6}$ mod $\Z$.
\end{enumerate}
\end{enumerate}
\end{theorem}
The hypothesis on the parameters in the above theorem corresponds to irreducibility of the monodromy group \cite[Prop. $1$]{MR0498553}.

\begin{rem}
For $N>1$, assuming only  $\theta_i \not \in \Z^*, i=2,\ldots,N$ and
 $\theta_1\theta_{N+1}\theta_{N+2}\theta_{N+3}\neq0$, we may still obtain informations. Let $y(t)$ be a solution of $(\mathcal L^iy=0)_i$, $u$ be the corresponding column vector and $v:=Gu$.
Our Theorem $\ref{thnqcq}$ allows to infer that $y$ has finite branching if and only if $[\theta_i=0 \Rightarrow v_{i-1}\equiv 0]$.
In case of finite branching, if there exists $i>2$ with $\theta_i=0$, the corresponding linear relation among the derivatives of $y$ allows to see that we may obtain $y(t)=\tilde{y}(f_1(t),\ldots,f_{N-1}(t))$, for certain linear forms $(f_i)$. This reflects \cite[Th. $1.2$]{MR1892536}.
\end{rem}
\section*{Funding}
This work was supported by Italian Fondo per gli Investimenti della Ricerca di Base [RBFR12W1AQ to G.C., RBFR10GHHH to D.M.]; French Agence Nationale de la Recherche [ANR-13-JS01-0002-01 to G.C.]; labex IRMIA [G.C.]; and  University of Burgundy [D.M.].
\section*{Acknowledgements}
We are grateful to Ivan Marin for a detailed introduction to Springer Theory. We thank Frank Loray and Felix Ulmer for discussions around Corollary $\ref{corgalinv}$. We acknowledge the anonymous referee for her/his careful reading of the paper and for valuable suggestions.

\bibliographystyle{amsalpha}
\bibliography{biblioOA}

\end{document}